\documentclass[reqno,a4paper]{amsart}

\usepackage{graphics}
\usepackage{lscape}
\usepackage[latin1]{inputenc}
\usepackage[italian,english]{babel}
\usepackage{eucal,amsfonts,amssymb,amsmath,amsthm,epsfig,mathrsfs,cases,subfig}
\usepackage{epstopdf}
\usepackage{dsfont}
\usepackage{color}
\usepackage{amscd,amsxtra}
\usepackage{enumerate}
\usepackage{latexsym}
\usepackage{bm}

\allowdisplaybreaks

\newcounter{ipotesi}

\Alph{ipotesi}
 \makeatletter \@addtoreset{equation}{section}

\makeatother \makeatletter

\newtheorem{thm}{Theorem}[section]
{\rm}
{\rm}
\newtheorem{lemm}[thm]{Lemma}
\newtheorem{coro}[thm]{Corollary}
\newtheorem{prop}[thm]{Proposition}

\newtheorem{rmk}[thm]{Remark}{\rm}

\newtheorem*{thm-main}{Main Theorem}

\newcounter{parentenv}

\newcommand{\vp}{\varrho}

\newcommand{\R}{{\mathbb R}}

\newcommand{\N}{{\mathbb N}}
\newcommand{\Z}{{\mathbb Z}}
\newcommand{\C}{{\mathbb C}}

\newcommand{\Rd}{\mathbb R^d}

\newcommand{\f}{{\bf f}}
\newcommand{\uu}{{\bf u}}

\newcommand{\Le}{{\rm{{Le}}}}

\newcommand{\gap}{\hspace{1pt}}
\allowdisplaybreaks
\begin{document}

\title[Instabilities in a combustion model]{Instabilities in a combustion model\\ with two free interfaces}

\author[D. Addona]{Davide Addona}
\address{Dipartimento di Matematica e Applicazioni, Universit\`a degli Studi di Milano Bicocca, via R. Cozzi 55, I-20125 Milano (Italy)}

\author[C.-M. Brauner]{Claude-Michel Brauner}
\address{School of Mathematical Sciences, University of Science and Technology of China, Hefei 230026 (China), and
Institut de Math\'ematiques de Bordeaux, Universit\'e de Bordeaux, 33405 Talence Cedex (France).}

\author[L. Lorenzi]{Luca Lorenzi}
\address{Plesso di Matematica e Informatica, Dipartimento di Scienze Matematiche, Fisiche e Informatiche, Universit\`a di Parma, Parco Area delle Scienze 53/A, I-43124 Parma (Italy)}

\author[W. Zhang]{Wen Zhang$^\dag$}\address{School of Science, East China University of Technology, Nanchang 330013 (China). Former affiliation: School of Mathematical Sciences, Xiamen University, Xiamen 361005 (China).}
\thanks{$\dag$ Corresponding author}
\email{davide.addona@unimib.it}
\email{claude-michel.brauner@u-bordeaux.fr}
\email{luca.lorenzi@unipr.it}
\email{zhangwenmath@126.com}

\date{}


\maketitle

\begin{abstract}
We study in a strip of $\R^2$ a combustion model of flame propagation with stepwise temperature kinetics and zero-order reaction, characterized by two free interfaces, respectively the ignition and the trailing fronts. The latter interface presents an additional difficulty because the non-degeneracy condition is not met. We turn the system to a fully nonlinear problem which is thoroughly investigated. When the width $\ell$ of the strip is sufficiently large, we prove the existence of a critical value $\Le_c$ of the Lewis number $\Le$, such that the one-dimensional, planar, solution is unstable for $0<\Le<\Le_c$. Some numerical simulations confirm the analysis.
\end{abstract}

\section{Introduction}
This paper is devoted to the analysis of cellular instabilities of planar traveling fronts  for a thermo-diffusive model of flame
propagation with stepwise temperature kinetics and zero-order reaction. In non-dimensional form, the model reads:
\begin{eqnarray}\label{eq:i1}
\left\{
\begin{array}{l}
\Theta_t=\Delta \Theta+W(\Phi,\Theta), \\[2mm]
\Phi_t=\frac{1}{{\Le}} \Delta \Phi-W(\Phi,\Theta),
\end{array}
\right.
\end{eqnarray}
where $\Theta$ and $\Phi$ are appropriately normalized temperature and concentration of deficient reactant,
${\Le}$ is the Lewis number and $W(\Phi, \Theta)$ is a reaction rate given by
\begin{eqnarray}\label{eq:i2}
W(\Theta,\Phi)=
\left\{
\begin{array}{lllll}
A,  & \mbox{if} & \Theta\ge \Theta_i & \mbox{and}& \Phi>0, \\[2mm]
0, & \mbox{if} &\Theta<\Theta_i & \mbox{and/or} & \Phi=0.
\end{array}
\right.
\end{eqnarray}
Here, $0<\Theta_i<1$ is the ignition temperature and $A>0$ is a normalizing factor.

Combustion models involving discontinuous reaction terms, including the system \eqref{eq:i1}-\eqref{eq:i2},
have been used  by physicists and engineers since the very early stage of the development of the combustion science
(see Mallard and Le Ch\^atelier \cite{i4}),
primarily due to their relative simplicity and mathematical tractability (see, e.g., \cite{i7,i5,i6}, and
more recently \cite{i8, BGKS15}). These models have drawn several mathematical studies on systems with discontinuous nonlinearities and related Free Boundary Problems which include, besides the pioneering work of K.-C. Chang \cite{C78}, the references \cite{BRS95,G95,GH92,GM93a,GM93b,NS87,NS89}, to mention a few of them. In particular, models with ignition temperature were introduced in the mathematical description of the propagation of premixed flames to solve the so-called ``cold-boundary difficulty'' (see, e.g., \cite[Section 2.2]{BL83}, \cite{BSN85}).

More specifically, in this paper we consider the free interface problem associated to the model \eqref{eq:i1}-\eqref{eq:i2}. The domain is the strip $\mathbb{R} \times (-{\ell}/{2},{\ell}/{2})$, the spatial coordinates are denoted by $(x,y)$, $t>0$ is the time. The free interfaces are respectively the {\it ignition interface} $x =F(t,y)$ and the {\it trailing interface} $x = G(t,y)$, $G(t,y)<F(t,y)$, defined by
$\Theta(t, F(t,y),y) = \theta_i$, $\Phi(t, G(t,y),y) = 0$.
The system reads as follows, for $ t>0$ and $y\in (-\ell/2,\ell/2)$:
\begin{eqnarray}
\left\{
\begin{array}{ll}
\displaystyle\frac{\partial\Theta}{\partial t}(t,x,y) = \Delta \Theta(t,x,y),\quad &x<G(t,y),\\[3mm]
\Phi(t,x,y)=0, \quad &x<G(t,y),\\[2mm]
\displaystyle\frac{\partial\Theta}{\partial t}(t,x,y) = \Delta \Theta(t,x,y) + A, \quad &G(t,y)<x<F(t,y),\\[3mm]
\displaystyle\frac{\partial\Phi}{\partial t}(t,x,y) = (\Le)^{-1}\Delta \Phi(t,x,y) - A, \quad &G(t,y)<x<F(t,y), \\[3mm]
\displaystyle\frac{\partial\Theta}{\partial t}= \Delta \Theta(t,x,y), \quad &x>F(t,y),\\[3mm]
\displaystyle\frac{\partial\Phi}{\partial t}= (\Le)^{-1}\Delta \Phi(t,x,y), \quad &x>F(t,y),
\end{array}
\right.
\label{system-1}
\end{eqnarray}
where the normalizing factor $A$ will be fixed below. The functions $\Theta$ and $\Phi$ are continuous across the interfaces for $t>0$, as well as their normal derivatives. As $x\to \pm \infty$, it holds
\begin{equation}\label{infty}
  \Theta(t,-\infty,y) = \Phi(t,+\infty,y) =1, \qquad\;\, \Theta(t,+\infty,y) = 0.
\end{equation}
Finally, periodic boundary conditions are assumed at $y=\pm \ell/2$.

As was noted in earlier studies (see \cite{BGKS15,BGZ}), this system is very different from models arising in conventional thermo-diffusive combustion. Two are the principal differences. (i) The first one is that in the model considered here, the reaction zone is of order unity, whereas in the case of Arrhenius kinetics the reaction zone is infinitely thin. This fact suggests to refer to flame fronts for stepwise temperature kinetics as thick flames, in contrast to thin flames arising in Arrhenius kinetics.
(ii) The second, even more important difference, is that, in  the case of Arrhenius kinetics, there is a single interface separating burned and unburned gases. In contrast to that, in case of the stepwise temperature kinetics given by \eqref{eq:i2}, there are two interfaces, namely the {\it ignition interface} where $\Theta=\Theta_i$ located at $x=F(t,y)$, and {\it trailing interface} at $x=G(t,y)$ being defined as a largest value of $x$ where the concentration is
equal to zero. As a consequence of (i),  the normal derivatives are continuous across both interfaces, in contrast to classical models with Arrhenius kinetics where jumps occur at the flame front (see e.g.,
\cite[Section 11.8]{BL83} and \cite{MS79,S80} for the related Kuramoto-Sivashinsky equation). There have been a number of mathematical works in the latter case based on the method of \cite{BHL00}
that we are going to extend below, see in particular \cite{BHL13,BHLS10,BLSX10,BL00,lorenzi-1,lorenzi-2,lorenzi-3} for the flame front, and the references therein.
Finally, note that Free Boundary Problems with two interfaces have already been considered in the literature, especially in Stefan problems, see e.g.,
\cite{DL10,DL10-bis,WZ17} (one-dimensional problem) and \cite{DG11} (radial solutions).

The above system admits a one-dimensional traveling wave (planar) solution $(\Theta^{(0)},\Phi^{(0)})$ which propagates with constant positive velocity $V$ (see \cite[Section 4]{BGKS15}). It is convenient to choose the normalizing factor $A= 1/R$ in such a way that $V=1$, where the positive number $R = R(\theta_i)$ is given by:
\begin{equation}
\theta_i R = 1 - e^{-R}, \quad 0<\theta_i<1.
\label{theta-i-R}
\end{equation}
Thus, in the moving frame coordinate $x'=x-t$, the system for the travelling wave solution reads as follows:
\begin{eqnarray*}
\left\{
\begin{array}{ll}
D_{x'}\Theta^{(0)} + D_{x'x'}\Theta^{(0)}= 0,   & {\rm in}~(-\infty,0],\\[1mm]
D_{x'}\Theta^{(0)}+ D_{x'x'}\Theta^{(0)}=- R^{-1}, & {\rm in}~(0,R),\\[1mm]
\Le\, D_{x'}\Phi^{(0)}+D_{x'x'}\Phi^{(0)}=\Le R^{-1}, & {\rm in}~(0,R),\\[1mm]
D_{x'}\Theta^{(0)}+ D_{x'x'}\Theta^{(0)}= 0, &{\rm in}~[R,+\infty),\\[1mm]
\Le\, D_{x'}\Phi^{(0)}_{x'}+D_{x'x'}\Phi^{(0)}=0, &{\rm in}~[R,+\infty),
\end{array}
\right.
\end{eqnarray*}
whose solution is
\begin{equation}
\Theta^{(0)}(x')=
\left\{
\begin{array}{ll}
1, & x'\le 0,\\[2mm]
\displaystyle 1 + \frac{1-x' - e^{-x'}}{R}, &x'\in (0,R),\\[2mm]
\displaystyle \theta_i e^{R-x'}, & x'\ge R,
\end{array}
\right.
\qquad\;\,
\Phi^{(0)}(x')=
\left\{
\begin{array}{ll}
0, & x'\le 0,\\[2mm]
\displaystyle\frac{e^{-\Le x'} -1}{\Le\, R} + \frac{x'}{R}, &x'\in (0,R),\\[2mm]
\displaystyle 1 + \frac{1- e^{ \Le R}}{\Le\, R\, e^{\Le x'}}, & x'\ge R.
\end{array}
\right.
\label{planar-front}
\end{equation}

\begin{figure}[ht]
\centering
\includegraphics[height=5cm, width=8cm]{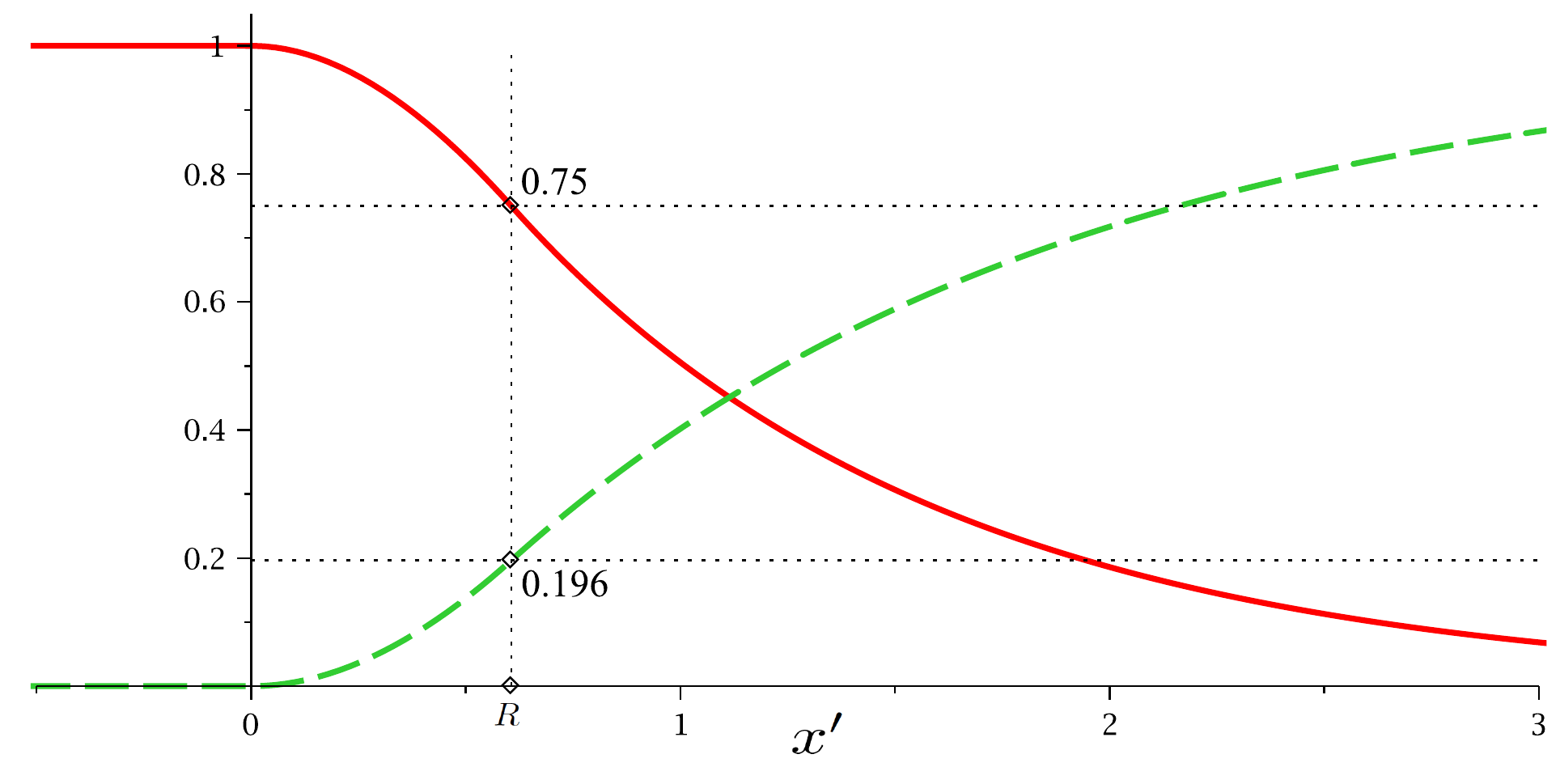}
\caption{$\Theta^{(0)}$ (solid curve) and $\Phi^{(0)}$ (dashed curve) with $\theta_i = 0.75 $, $ \Le = 0.75$ ($R = 0.60586$). } \label{Theta&Phi}
\end{figure}

The existence of traveling fronts poses a natural question of one and multidimensional stability,
or especially instabilities of such fronts. It is known (see \cite{MS79,S80}) that diffusional-thermal instabilities of planar flame fronts, when the Lewis number is less than unity, generates cellular flames and pattern formation. In this paper, we focus our attention on instabilities of the traveling wave $(\Theta^{(0)},\Phi^{(0)})$, and thus for the ignition and the trailing interfaces. Earlier studies have shown (see \cite{BGZ}) that instabilities depend on the Lewis number and occur only when the width of the strip $\ell$ is large enough (in \cite{BGKS15}, $\ell$ is taken to infinity), which motivates the present study.

The main result of the paper is the following:

\begin{thm-main}
\label{thm-main}
Let $0<\theta_i<1$ be fixed. There exist $\ell_0(\theta_i)$ sufficiently large, such that, whenever $\ell>\ell_0(\theta_i)$,
there exists a  critical value of the Lewis number, say $\Le_c\in (0,1)$ $($see \eqref{critic1}$)$. If $\Le\in (0,\Le_c)$,
then the traveling wave solution to problem \eqref{system-1}-\eqref{infty} is unstable with respect to smooth and sufficiently small two dimensional perturbation.
Further, also the ignition and the trailing interfaces are pointwise unstable.
\end{thm-main}

The paper is organized as follows. In Section \ref{sect-2}, we introduce the main notation and the functional spaces. Section \ref{sect-3} is devoted to transforming problem \eqref{system-1}-\eqref{infty} in a {\it fully nonlinear} for problem for the perturbation of the traveling wave solution $(\Theta^{(0)},\Phi^{(0)})$ in \eqref{planar-front}, set in a fixed domain (see \eqref{asta}). We determine that the ignition interface meets the transversality (or non-degeneracy) condition of \cite{BHL00}. Unfortunately, this is not the case of the trailing interface which is of different nature. In short, the idea is to differentiate the mass fraction equation, taking advantage of the structure of the problems.

Then, in Section \ref{sect-4} we collect some tools that are needed to prove the main result. The theory of analytic semigroups plays a crucial role in all our analysis. For this reason, one of the main tools of this section is a generation result: we will show that a suitable realization of the linearized (at zero) elliptic operator associated with the fully nonlinear problem \eqref{asta} generates an analytic semigroup and we characterize the interpolation spaces.
This will allow us to prove an optimal regularity result for classical solutions to problem \eqref{asta}.
Section \ref{sect-6} contains the proof of Theorem \ref{thm-main}. Finally, Section \ref{sect-7} is devoted to a numerical method and computational results which show two-cell patterns (see \cite{BGZ} for further results).

\section{Notation, functional spaces and preliminaries}
\label{sect-2}

In this section, we collect all the notation, the functional spaces and the preliminary results that we use throughout the paper

\subsection{Notation}
\label{sub-notation}
We find it convenient to set, for each $\tau>0$,
\begin{eqnarray*}
\begin{array}{lll}
S=\R\times (-\ell/2,\ell/2),\quad\;\, &S_{\tau}^+=(\tau,+\infty)\times (-\ell/2,\ell/2),\quad\;\, &S_{\tau}^-=(-\infty,\tau)\times (-\ell/2,\ell/2),\\[1mm]
H_{\tau}^{-}=(-\infty,\tau)\times\R,\quad\;\, &H_{\tau}^{+}=(\tau,+\infty)\times\R, \quad\;\, &R_T=[0,T]\times [-\ell/2,\ell/2].
\end{array}
\end{eqnarray*}

\paragraph{\it{\bf Functions.}}
Given a function $f:(a,b)\to\R$ and a point $x_0\in (a,b)$, we denote by $[f]_{x_0}$ the jump of $f$ at $x_0$, i.e., the difference $f(x_0^+)-f(x_0^-)$ whenever defined.
For each function $f:[-\ell/2,\ell/2)\to\C$ we denote by $f^{\sharp}$ its $\ell$-periodic extension to $\R$. If $f$ depends also on $x$ running in some interval  $I$,
we still denote by $f^{\sharp}$ its periodic (with respect to $y$) extension to $I\times\R$.
For every $f\in L^2((-\ell/2,\ell/2))$ and $k\in\Z$, we denote by $\hat f_k$ the $k$-th Fourier coefficient of $f$, i.e.,
\begin{eqnarray*}
\hat f_k=\frac{1}{\ell}\int_{-\frac{\ell}{2}}^{\frac{\ell}{2}}f\overline{e_k}dy,
\end{eqnarray*}
where $e_h(y)=e^{\frac{2h\pi i}{\ell}y}$ for each $h\in\Z$ and $y\in\R$.
When $f$ depends also on the variable $x$ running in some interval $I$, $\hat f_k(x)$ stands for the $k$-th Fourier coefficient of the function
$f(x,\cdot)$.

The time and the spatial derivatives of a given function $f$ are denoted by $D_tf$ ($=f_t$), $D_xf$ ($=f_x$), $D_yf$ ($=f_y$) $D_{xx}f$ ($=f_{xx}$) $D_{xy}f$ ($=f_{xy}$) and $D_{yy}f$ ($=f_{yy}$), respectively. If $\beta=(\beta_1,\beta_2)$ with $\beta_1,\beta_2\in\N\cup\{0\}$, then we set $D^{\gamma}=D^{\gamma_1}_xD^{\gamma_2}_y$.

Finally, we denote by $\chi_A$ the characteristic function of the set $A\subset\Rd$ ($d\ge 1$).

\medskip

\paragraph{\bf Miscellanea}
Throughout the paper, we denote by $c_{\lambda}$ a positive constant, possibly depending on $\lambda$ but being independent of $k$, $n$, $x$ and the functions that we will consider, which may vary from line to line. We simply write $c$ when the constant is independent also of $\lambda$.

The subscript ``$b$'' stands for bounded. For instance $C_b(\Omega;\C)$ denotes the set of bounded and continuous function from $\Omega$ to $\C$.
When we deal with spaces of real-valued functions we omit to write ``$\C$''.

Vector-valued functions are displayed in bold.

\subsection{Main function spaces}
Here, we collect the main function spaces used in the paper pointing out the (sub)section where they are used for the first time.
\medskip

\paragraph{\bf The spaces $\boldsymbol{\mathcal X}$ and $\boldsymbol{\mathcal X}_{k+\alpha}$  (Section \ref{sect-4})}

By $\boldsymbol{\mathcal X}$ we denote the set of all pairs ${\bf f}=(f_1,f_2)$, where $f_1:\overline{S}\to\C$ and $f_2:\overline{S_0^{+}}\to\C$ are bounded functions,
$f_1\in C(\overline{S_0^-};\C)\cap C([0,R]\times [-\ell/2,\ell/2];\C)\cap C(\overline{S_R^+};\C)$, $f_2\in C([0,R]\times [-\ell/2,\ell/2];\C)\cap C(\overline{S_R^+};\C)$
and $\lim_{x\to \pm\infty}f_1(x,y)=\lim_{x\to +\infty}f_2(x,y)=0$ for each $y\in [-\ell/2,\ell/2]$. It is endowed with the sup-norm, i.e., $\|\f\|_{\infty}=\|f_1\|_{L^{\infty}(S;\C)}+
\|f_2\|_{L^{\infty}(S_0^+;\C)}$.

For each $\alpha\in (0,1]$, $\boldsymbol{\mathcal X}_{\alpha}$ denotes the subset of $\boldsymbol{\mathcal X}$ of all $\f$ such that
(i) $f_1\in C^{\alpha}_b(\overline{S_0^-};\C)\cap C^{\alpha}_b([0,R]\times[-\ell/2,\ell/2];\C)\cap C^{\alpha}_b(\overline{S_R^+};\C)$, (ii) $f_2\in C^{\alpha}_b([0,R]\times[-\ell/2,\ell/2];\C)\cap C^{\alpha}_b(\overline{S_R^+};\C)$, (iii) $f_j(\cdot,-\ell/2)=f_j(\cdot,\ell/2)$ (and $\nabla f_j(\cdot,-\ell/2)=
 \nabla f_j(\cdot,\ell/2)$ if $\alpha=1$) for $j=1,2$. It is endowed with the norm
$\|{\bf f}\|_{\alpha}
=\|f_1\|_{C^{\alpha}_b(\overline{S_0^-};\C)}+\sum_{j=1}^2(\|f_j\|_{C^{\alpha}_b([0,R]\times[-\ell/2,\ell/2];\C)}
+\|f_j\|_{C^{\alpha}_b(\overline{S_R^+};\C)})$.

For $k\in\N$ and $\alpha\in (0,1)$,
$\boldsymbol{\mathcal X}_{k+\alpha}$ ($k\in\N$, $\alpha\in (0,1)$) denotes the set of all ${\bf f}\in \boldsymbol{\mathcal X}$ such that $D^\beta{\bf f}=(D^{\gamma}f_1,D^{\gamma}f_2)\in\boldsymbol{\mathcal X}$, $D^{\gamma}f_j(\cdot,-\ell/2)=D^{\gamma}f_j(\cdot,\ell/2)$ for each $|\gamma|\leq k$, $j=1,2$, and
$D^{\gamma}{\bf f}\in \boldsymbol{\mathcal X}_{\alpha}$ for $|\gamma|=k$. It is endowed with the norm
$\|{\bf f}\|_{k+\alpha}:=\sum_{|\gamma|< k}\|{D^{\gamma}{\bf f}}\|_{\infty}+\sum_{|\gamma|=k}\|D^{\gamma}{\bf f}\|_{\alpha}$.

\medskip

\paragraph{\bf The spaces $\boldsymbol{\mathcal Y}_{\alpha}(a,b)$ and $\boldsymbol{\mathcal Y}_{2+\alpha}(a,b)$ (Section \ref{sect-5})}

For $\alpha\in (0,1)$ and $0\le a<b$, we define by $\boldsymbol{\mathcal Y}_{\alpha}(a,b)$ the space of all pairs ${\bf f}=(f_1,f_2)$ such that
$f_1:[a,b]\times\overline{S}\to\R$, $f_2:[a,b]\times \overline{S_0^+}\to\R$ and
\begin{eqnarray*}
\|{\bf f}\|_{\boldsymbol{\mathcal Y}_{\alpha}(a,b)}=\sup_{a<t<b}\|{\bf f}(t,\cdot,\cdot)\|_{\alpha}+\sup_{(x,y)\in S}\|f_1(\cdot,x,y)\|_{C^{\alpha/2}((a,b))}+\sup_{(x,y)\in S^+_0}\|f_2(\cdot,x,y)\|_{C^{\alpha/2}((a,b))}<+\infty.
\end{eqnarray*}

Similarly, $\boldsymbol{\mathcal Y}_{2+\alpha}(a,b)$ denotes the space of all the pairs ${\bf u}$ such that $D^{\gamma_1}_tD^{\gamma_2}_xD^{\gamma_3}_y{\bf u}$ belongs to
$\boldsymbol{\mathcal Y}_{\alpha}(a,b)$ for every $\gamma_1, \gamma_2, \gamma_3\ge 0$ such that $2\gamma_1+\gamma_2+\gamma_3\leq 2$. These are Banach spaces with the norms
$\|\cdot\|_{\boldsymbol{\mathcal Y}_{\alpha}(a,b)}$ and
$\|{\bf u}\|_{\boldsymbol{\mathcal Y}_{2+\alpha}(a,b)}
=  \sum_{2\gamma_1+\gamma_2+\gamma_3\leq 2}\|D^{\gamma_1}_tD^{\gamma_2}_xD^{\gamma_3}_y{\bf u}\|_{\boldsymbol{\mathcal Y}_{\alpha}(a,b)}$.

If $a=0$ and $b=T$ then we simply write $\boldsymbol{\mathcal Y}_{\alpha}$ and $\boldsymbol{\mathcal Y}_{2+\alpha}$ instead of $\boldsymbol{\mathcal Y}_{\alpha}(0,T)$ and $\boldsymbol{\mathcal Y}_{2+\alpha}(0,T)$, respectively.

\section{Derivation of the fully nonlinear problem}\label{FNL1}
\label{sect-3}

\subsection{The system on a fixed domain}\label{fix-domain}
To begin with, we rewrite System \eqref{system-1} in the coordinates $t'=t$, $x'=x-t$, $y'=y$, $D_t=D_{t'}-D_{x'}$. Next, we look for the free interfaces respectively as:
\begin{equation*}
 G(t',y') = g(t',y'), \qquad\;\, F(t',y') = R + f(t',y'),
\end{equation*}
where $f$ and $g$ are small perturbations. In other words, the space variable $x'$ varies from $-\infty$ to $g(t',y')$, from $g(t',y')$ to $R + f(t',y')$, and eventually from $R + f(t',y')$ to $+\infty$. As usual, it is convenient to transform the problem on a variable domain to a problem on a fixed domain. To this end, we define a coordinate transformation in the spirit of \cite[Section 2.1]{BHL00}:
\begin{equation*}
t'=\tau, \qquad\;\, x'=\xi+\beta(\xi)g(\tau,\eta)+\beta(\xi-R)f(\tau,\eta),\qquad\;\, y'=\eta,
\end{equation*}
where $\beta$ is a smooth mollifier, equal to unity in a small neighborhood of $\xi=0$, say $[-\delta,\delta]$, and has compact support contained in $(-2\delta,2\delta)\subset (-R,R)$. When $x'=g$, $\xi=0$, and $\xi=R$ when $x'=R+f$.
Then, the trailing front and the ignition front are fixed at $\xi=0$ and $\xi=R$, respectively. Thanks to the translation invariance, \eqref{planar-front} holds with the variable $\xi$.
For convenience, we introduce the notation
\begin{equation}
\vp(\tau,\xi,\eta) = \beta(\xi)g(\tau,\eta)+\beta(\xi-R)f(\tau,\eta)
\label{notation}
\end{equation}
and we expand
$(1+\vp_\xi)^{-1}= 1-\vp_\xi + (\vp_{\xi})^2(1+\vp_\xi)^{-1}$. It turns out that
\begin{align*}
&D_{x'}=D_{\xi} - \vp_\xi D_{\xi}+(\vp_{\xi})^2(1+\vp_\xi)^{-1}D_{\xi},\\
&D_{t'}=D_{\tau}-\vp_\tau D_{\xi}+ \vp_\tau\vp_\xi D_{\xi}-\vp_\tau(\vp_{\xi})^2(1+\vp_\xi)^{-1}D_{\xi},\\
&D_{y'}=D_{\eta}-\vp_\eta D_{\xi}+ \vp_\eta\vp_\xi D_{\xi}-\vp_\eta(\vp_{\xi})^2(1+\vp_\xi)^{-1}D_{\xi}
\end{align*}
and System \eqref{system-1} reads:
\begin{gather}\label{Model2-i}
\begin{split}
\Theta_\tau = & \Theta_{\xi}+ \Delta\Theta+ (\vp_\eta^2 - \vp_\xi^2 -2 \vp_\xi)(1+\vp_\xi)^{-2}\Theta_{\xi\xi}-2 \vp_\eta (1+\vp_\xi)^{-1}\Theta_{\xi\eta}  \\
   &+ \big[(\vp_\tau - \vp_\xi-\vp_{\eta\eta})(1+\vp_\xi)^{-1} + 2 \vp_\eta \vp_{\xi\eta} (1+\vp_\xi)^{-2}  - \vp_{\xi\xi} (1+ \vp_\eta^2) (1+\vp_\xi)^{-3}  \big ] \Theta_\xi,\\[2mm]
   \Phi = & 0
  \end{split}
\end{gather}
in $(0,+\infty)\times (-\infty,0)\times (-\ell/2,\ell/2)$,
\begin{gather}\label{Model2-ii}
  \begin{split}
\Theta_\tau =& \Theta_{\xi}+ \Delta\Theta+(\vp_\eta^2 - \vp_\xi^2 -2 \vp_\xi)(1+\vp_\xi)^{-2}\Theta_{\xi\xi}  -2 \vp_\eta (1+\vp_\xi)^{-1}\Theta_{\xi\eta}  + R^{-1}  \\
   &+ \big [(\vp_\tau - \vp_\xi-\vp_{\eta\eta})(1+\vp_\xi)^{-1} + 2 \vp_\eta \vp_{\xi\eta} (1+\vp_\xi)^{-2}  - \vp_{\xi\xi} (1+ \vp_\eta^2) (1+\vp_\xi)^{-3}  \big ] \Theta_\xi,\\[2mm]
\Phi_\tau =& \Phi_{\xi}+ \Le^{-1}\Delta\Phi+\Le^{-1}(\vp_\eta^2 - \vp_\xi^2 -2 \vp_\xi)(1+\vp_\xi)^{-2}\Phi_{\xi\xi}-2\Le^{-1} \vp_\eta (1+\vp_\xi)^{-1}\Phi_{\xi\eta}  - R^{-1}\\
&+[(\vp_\tau\! -\! \vp_\xi\!-\!\Le^{-1}\vp_{\eta\eta})(1+\vp_\xi)^{-1}\! +\! 2\Le^{-1} \vp_\eta \vp_{\xi\eta} (1+\vp_\xi)^{-2}-\Le^{-1} \vp_{\xi\xi} (1+ \vp_\eta^2) (1+\vp_\xi)^{-3}] \Phi_\xi
\end{split}
\end{gather}
in $(0,+\infty)\times (0,R)\times (-\ell/2,\ell/2)$ and
\begin{gather}\label{Model2-iii}
  \begin{split}
\Theta_\tau =& \Theta_{\xi}+\Delta\Theta+ (\vp_\eta^2 - \vp_\xi^2 -2 \vp_\xi)(1+\vp_\xi)^{-2} \Theta_{\xi\xi}  -2 \vp_\eta (1+\vp_\xi)^{-1} \Theta_{\xi\eta}  \\
   &+ \big [(\vp_\tau - \vp_\xi-\vp_{\eta\eta})(1+\vp_\xi)^{-1} + 2 \vp_\eta \vp_{\xi\eta} (1+\vp_\xi)^{-2}  - \vp_{\xi\xi} (1+ \vp_\eta^2) (1+\vp_\xi)^{-3}  \big ] \Theta_\xi ,\\[2mm]
\Phi_\tau =& \Phi_{\xi}+ \Le^{-1}\Delta\Phi+\Le^{-1}(\vp_\eta^2 - \vp_\xi^2 -2 \vp_\xi)(1+\vp_\xi)^{-2}\Phi_{\xi\xi}-2\Le^{-1} \vp_\eta (1+\vp_\xi)^{-1}\Phi_{\xi\eta}\\
&+  \big [(\vp_\tau\! -\! \vp_\xi\!-\!\Le^{-1}\vp_{\eta\eta})(1+\vp_\xi)^{-1}\! +\! 2\Le^{-1} \vp_\eta \vp_{\xi\eta} (1+\vp_\xi)^{-2}-\Le^{-1} \vp_{\xi\xi} (1+ \vp_\eta^2) (1+\vp_\xi)^{-3}  \big ] \Phi_\xi
 \end{split}
\end{gather}
in $(0,+\infty)\times (R,+\infty)\times (-\ell/2,\ell/2)$.
Moreover, $\Theta$ and $\Phi$ are continuous at the (fixed) interfaces $\xi=0$ and $\xi=R$, and so are their first-order derivatives. Thus,
\begin{equation*}
[\Theta(\tau,\cdot,\eta)]_0= [\Theta_\xi(\tau,\cdot,\eta)]_0=\Phi(\tau,0,\eta)=\Phi_\xi(\tau,0,\eta)=0
\end{equation*}
and
\begin{equation*}
\Theta (\tau,R,\eta)=\theta_i, \quad [\Theta(\tau,\cdot,\eta)]_R=[\Theta_\xi(\tau,\cdot,\eta)]_R= [\Phi(\tau,\cdot,\eta)]=[\Phi_\xi(\tau,\cdot,\eta)]=0.
\end{equation*}
Conditions \eqref{infty} hold at $\xi=\pm\infty$ and periodic boundary conditions are assumed at $\eta=\pm \ell/2$.

\subsection{Elimination of the interfaces}
\label{sect-3.2}
From now on, with a slight abuse of notation, which does not cause confusion, we write $t$ instead of $\tau$ and $x$, $y$ instead of $\xi$ and $\eta$.

In the spirit of \cite{BHL00,lorenzi-2}, we introduce the splitting:
\begin{gather}
 \Theta(t,x,y) = \Theta^{(0)}(x) + \vp(t,x,y)\Theta^{(0)}_x(x) + u(t, x, y),\label{splitTheta}\\
\Phi(   t,x,y) = \Phi^{(0)}(x) + \vp(t,x,y) \Phi^{(0)}_x(x) + v(t,x, y),\label{splitPhi}
\end{gather}
which is a sort of Taylor expansion of $(\Theta,\Phi)$ around the travelling wave solution $(\Theta^{(0)},\Phi^{(0)})$. Thus, the pair $(u,v)$ plays the role of a remainder and, since we are interested in stability issues, we can assume that $u$ and $v$ are ``sufficiently small'' in a sense which will be made precise later on.

A long but straightforward computation reveals that the pair $(u,v)$ satisfies the differential equations
\begin{gather}\label{PertTheta+}
  \begin{split}
      u_t = & u_x + \Delta u+  \vp_t (1+\vp_x)^{-1}(\vp \Theta^{(0)}_{xx} + u_x)-(1+\vp_x)^{-3}\vp_{xx} (1+\vp_y^2) (\vp \Theta^{(0)}_{xx} + u_x)  \\
      & -(1+\vp_x)^{-1}\big [(\vp_{x} + \vp_{yy}) (\vp \Theta^{(0)}_{xx}+u_x) +2 \vp_y (\vp_y \Theta^{(0)}_{xx}+u_{xy})\big ]\\
      & + (1+\vp_x)^{-2} \big [2 \vp_{y} \vp_{xy} (\vp \Theta^{(0)}_{xx}+u_x) +   (\vp_{y}^2 -\vp_{x}^2) (\vp \Theta^{(0)}_{xxx}+\Theta^{(0)}_{xx}+u_{xx})\\
       &\phantom{+ (1+\vp_x)^{-2} \big [\;\,}- 2 \vp_x (\vp \Theta^{(0)}_{xxx} -\vp_y^2 \Theta^{(0)}_{xx}+u_{xx})\big ],
  \end{split}
\end{gather}
in $(0,+\infty)\times (\R\setminus\{0,R\})\times (-\ell/2,\ell/2)$ and
\begin{gather}\label{PertPhi}
  \begin{split}
     v_t = & v_x + \Le^{-1}\Delta v+  \vp_t (1+\vp_x)^{-1}
      (\vp \Phi^{(0)}_{xx} + v_x)-\Le^{-1}(1+\vp_x)^{-3}\vp_{xx} (1+\vp_y^2) (\vp \Phi^{(0)}_{xx} + v_x)\\
      & -  \Le^{-1}(1+\vp_x)^{-1}\big [(\Le\, \vp_{x} + \vp_{yy}) (\vp \Phi^{(0)}_{xx}+v_x) + 2 \vp_y (\vp_y \Phi^{(0)}_{xx}+v_{xy})\big ] \\
      & + \Le^{-1}(1+\vp_x)^{-2} \big [2 \vp_{y} \vp_{xy} (\vp \Phi^{(0)}_{xx}+v_x) +   (\vp_{y}^2 -\vp_{x}^2) (\vp \Phi^{(0)}_{xxx}+\Phi^{(0)}_{xx}+v_{xx})\\
       &\phantom{+ \Le^{-1}(1+\vp_x)^{-2}\big [\;\,}- 2 \vp_x (\vp \Phi^{(0)}_{xxx} -\vp_y^2 \Phi^{(0)}_{xx}+v_{xx})\big ]
  \end{split}
\end{gather}
in $(0,+\infty)\times [(0,R)\cup(R,+\infty)]\times (-\ell/2,\ell/2)$ and $v=0$ in $(0,+\infty)\times (-\infty,0)\times (-\ell/2,\ell/2)$.

Two steps are still needed:
\par
\noindent
(a) we have to determine the jump conditions satisfied by $u$ and $v$;
\par
\noindent
(b) again in the spirit of \cite{BHL00,lorenzi-2}, we have to get rid of the function $\varphi$ from the right-hand sides of \eqref{PertTheta+} and \eqref{PertPhi}.
As we will see, some difficulties appear and, to overcome them, we will differentiate the differential equation \eqref{PertPhi}.

\subsubsection{\bf The ignition interface  $\boldsymbol{x=R}$.} Note that $\Theta^{(0)}$, $\Phi^{(0)}$ belong to $C^1(\R)$. Thus, $[u]_R=[v]_R= 0$. Moreover,
$\Theta^{(0)}_{x}(R)= -\theta_i$ and $\Phi^{(0)}_{x}(R)=(1-\exp(-\Le R))/R$, so that they do not vanish at the interface $x=R$. The latter is a kind of transversality or non-degeneracy condition (see \cite{BHL00}). Evaluating \eqref{splitTheta} at $x=R$, we get $u(R) = \theta_i f$.

Next, differentiating
\eqref{splitTheta} and \eqref{splitPhi} for $x \neq R$, and taking the jumps across $x=R$, it is not difficult to show that
$[u_x]_R = -R^{-1}f$ and $[v_x]_R= R^{-1}\Le\, f$. This is a key point since we are able to express $f$ in terms of $u$ and write
\begin{equation}\label{elim-f}
f(t,y)={\theta_i}^{-1} u(t,R,y).
\end{equation}
Summing up, the interface conditions at $x=R$ are the following:
\begin{equation}\label{IC1}
[u]_R=[v]_R= 0, \quad u(R) +\theta_i R [u_x]_R =0,  \qquad\;\, \Le[u_x]_R + [v_x]_R =0.
\end{equation}

\subsubsection{\bf The trailing interface $\boldsymbol{x=0}$.} Taking the jump at $x=0$ of both sides of \eqref{splitTheta} and \eqref{splitPhi}, we get the conditions $[u]_0=v(0)=0$ for
$u$ and $v$. The trailing interface has a different nature with respect to the ignition interface. Indeed, since $\Theta^{(0)}_{x}(0)=\Phi^{(0)}_{x}(0)=0$, the non-degeneracy condition of \cite{BHL00} is not verified and  we are able to express $g$ in terms neither of $u$ or $v$.
On the other hand, $\Theta^{(0)}_{xx}(0^+)= -R^{-1}$ and $\Phi^{(0)}_{xx}(0^+)=R^{-1}\Le$, so that they do not vanish.
Differentiating \eqref{splitTheta} and \eqref{splitPhi} for $x \neq 0$, and taking the jumps yields:
$[u_x]_0 = R^{-1}g$, $v_x(\cdot,0^+,\cdot) = - R^{-1}g\Le$.
Hence, we get the additional interface condition $\Le\,[u_x]_0 + v_x(\cdot,0^+,\cdot) =0$, so that the interface conditions at $x=0$ are
\begin{equation}
[u]_0=v(0)=0,\quad \Le\,[u_x]_0 + v_x(\cdot,0^+,\cdot) =0.
\label{IC2}
\end{equation}
We can also write
\begin{equation}\label{elim-g}
g(t,y)= - R\,{\Le}^{-1} v_x(t, 0^+, y).
\end{equation}
Although the front $g$ could be eliminated, the method of \cite{BHL00} is not applicable since,
in contrast to \eqref{elim-f},  $g$ is related to the derivative of $v$ in the equation \eqref{elim-g}.

 \subsection{Differentiation and new interface conditions}
 To overcome the difficulty pointed out above, the trick is to differentiate \eqref{PertPhi} with respect to $x$, taking advantage of the structure of the system
 and consider the problem satisfied by the pair $(u,v_{x})$.
From \eqref{IC1} and \eqref{IC2} we get the following interface conditions for $u$ and $w=v_x$:
\begin{gather}\label{IC-diff}
\begin{split}
&[u]_0=0, \qquad\;\, \Le\, [u_x]_0 + w(0^+) =0,\\
&[u]_R= 0, \qquad\;\, u(R) +\theta_i R [u_x]_R =0,  \qquad\;\, \Le[u_x]_R + [w]_R =0.
\end{split}
\end{gather}

We missed two jump conditions: one at the trailing interface and the other one at the ignition interface.
To obtain the additional condition at the trailing interface $x=0$,
we differentiate \eqref{splitPhi} twice in a neighborhood of $x=0$ and take the trace at $x=0^+$. Using \eqref{elim-g}, we get
\begin{equation}\label{new-interface1}
\Phi_{xx}(\cdot,0^+,\cdot) = \Le\, R^{-1} +\Le\, w(\cdot,0^+,\cdot) + w_x(\cdot,0^+,\cdot).
\end{equation}
To get rid of $\Phi_{xx}(\cdot,0^+,\cdot)$ from the left-hand side of \eqref{new-interface1}, we observe that, for $x>0$ sufficiently small, the second equation in \eqref{Model2-ii} reduces to
\begin{gather}\label{neighbor-0}
\begin{split}
\Phi_t =& \Phi_{x}+ \Le^{-1}\Delta\Phi+ \Le^{-1}g_y^2 \Phi_{xx}
- 2\Le^{-1} g_y \Phi_{xy}  - R^{-1}+(g_t -\Le^{-1}g_{yy}) \Phi_x.
\end{split}
\end{gather}
Taking the trace of \eqref{neighbor-0} at $x=0^+$ it is easy to check that
$\Phi_{xx}(\cdot,0^+,\cdot)(1+g_y^2)  = \Le\, R^{-1}$.
Hence, using \eqref{elim-g} and \eqref{new-interface1} we get the additional interface condition
\begin{equation}\label{new-interface2}
\Le\, w(\cdot,0^+,\cdot) + w_x(\cdot,0^+,\cdot)= \Le\,R^{-1}\{[1+R^2\Le^{-2}(w_y(\cdot,0^+,\cdot))^2]^{-1}-1\}.
\end{equation}

We likewise identify the additional interface condition at the ignition interface $x=R$.
Differentiating twice \eqref{splitPhi} in a neighborhood of $x=R$, taking the jump at $x=R$ and using \eqref{elim-f}, \eqref{IC1} gives
\begin{equation}\label{new-interface3}
[\Phi_{xx}]_R = -R^{-1}\Le+\Le\, [w]_R + [w_x]_R.
\end{equation}
We need to compute $[\Phi_{xx}]_R$: in a neighborhood of $R^-$, the second equation in \eqref{Model2-ii} yields
\begin{align*}
\Phi_t =\Phi_{x}+ \Le^{-1}\Delta\Phi
+ \Le^{-1}(f_y)^2 \Phi_{xx}- 2\Le^{-1} f_y \Phi_{xy}  - R^{-1}+(f_t -\Le^{-1}f_{yy}) \Phi_x,
\end{align*}
while in a neighborhood of $R^+$ from the second equation in \eqref{Model2-iii} we get
\begin{align*}
\Phi_t =& \Phi_{x}+ \Le^{-1}\Delta\Phi+ \Le^{-1}(f_y)^2 \Phi_{xx}
- 2\Le^{-1} f_y \Phi_{xy}+(f_t -\Le^{-1}f_{yy}) \Phi_x.
\end{align*}
Using the previous two equations it can be easily shown that $[\Phi_{xx}]_R (1+ (f_y)^2) = -R^{-1}\Le$, which, together with \eqref{elim-f} and \eqref{new-interface3}, gives
\begin{equation}\label{new-interface4}
\Le [w]_R + [w_x]_R=- \Le\, R^{-1}\{[1+\theta_i^{-2}(u_y(\cdot,R,\cdot))^2]^{-1}-1\}.
\end{equation}
This is the additional condition we were looking for.

\subsection{Elimination of $\vp$ and its time and spatial derivatives} Formulae \eqref{elim-f} \eqref{elim-g} enable the elimination of the fronts $f$ and $g$ from the differential equations satisfied by $u$ and $w$. First, they allow to write the following formula for $\vp$ (see \eqref{notation}):
\begin{equation}\label{vp1}
\vp(t,x,y) = {\theta_i}^{-1}\beta(x-R)u(t,R,y) - R\,{\Le}^{-1}\beta(x) w(t, 0^+, y).
\end{equation}
Differentiation of \eqref{vp1} with respect to $x$ and $y$ is benign.
The right-hand sides of \eqref{PertTheta+} and \eqref{PertPhi} depend also on $\varrho_t$. Hence, we need to compute such a derivative and
express it in terms of (traces of) spatial derivatives of $u$ and $w$. Since
\begin{align}\label{stefan1}
\vp_{t}(t,x,y) = {\theta_i}^{-1}\beta(x-R)u_{t}(t,R,y) - R\,{\Le}^{-1}\beta(x) w_{t}(t, 0^+, y),
\end{align}
we need to get rid of $u_{t}(t,R,y)$ and $w_{t}(t,0^+,y)$. For simplicity, we forget the arguments $t$ and $y$.
We evaluate \eqref{PertTheta+} at $x=R^+$ (it would be equivalent at $x=R^-$). Recalling that all the derivatives of $\vp$ with respect of $x$ vanish and taking
\eqref{vp1} into account, we get
\begin{align*}
u_t(R) = & u_x(R^+) +\Delta u(R^+)+\theta_i^{-1}u_t(R) (u(R)+ u_x(R^+))-\theta_i^{-1}u_{yy}(R)u_x(R^+)\\
& -2\theta_i^{-1}u_y(R)u_{xy}(R^+)+\theta_i^{-2}(u_{y}(R))^2(u_{xx}(R^+)-u(R)-\theta_i)-\theta_i^{-1}u(R)u_{yy}(R).
\end{align*}
Since $\theta_i$ is fixed, assuming that the perturbations are small we may invert and write
\begin{align}\label{stefan2}
u_t(R)=[1\!-\!{\theta_i}^{-1}(u(R)\!+\!u_x(R^+))]^{-1}
[&u_x(R^+)+ \Delta u(R^+)-\theta_i^{-1}u_{yy}(R)u_x(R^+)-\theta_i^{-1}u(R)u_{yy}(R)\notag\\
&-2\theta_i^{-1}u_y(R)u_{xy}(R^+)+\theta_i^{-2}(u_{y}(R))^2(u_{xx}(R^+)-u(R)-\theta_i)].
\end{align}

Similarly, differentiating and evaluating \eqref{PertPhi} at $x=0^+$ we get
\begin{align*}
w_t(0^+) = & w_x(0^+) + \Le^{-1}\Delta w(0^+)-R\Le^{-1}w_t(0^+)(\Le\, w(0^+)+w_{x}(0^+))\\
& +R\Le^{-2}\big [w_{yy}(0^+)(\Le\, w(0^+)+w_{x}(0^+))+2w_y(0^+) w_{xy}(0^+)\big ] \\
& + R^2\Le^{-3}(w_{y}(0^+))^2(-\Le^2w(0^+)+R^{-1}\Le^2+w_{xx}(0^+)),
\end{align*}
so that
\begin{align}
w_t(0^+)=\{&\Le\, w_x(0^+) +\Delta w(0^+)+R\Le^{-1}\big [w_{yy}(0^+)(\Le\, w(0^+)+w_{x}(0^+))+2w_y(0^+) w_{xy}(0^+)\big ]\notag\\
&+R^2\Le^{-2}(w_{y}(0^+))^2(-\Le^2w(0^+)+R^{-1}\Le^2+w_{xx}(0^+))\}\notag\\
&\qquad\quad\times[\Le + R(\Le\, w(0^+) +w_{x}(0^+))]^{-1}.
\label{core-2}
\end{align}
A related remark is that Equation \eqref{stefan1} for $\vp_{t}$, together with formulas \eqref{stefan2}-\eqref{core-2}, may be viewed as a \textit{second-order Stefan condition}, see \cite{BL18}.

\subsection{The final system} Using \eqref{PertTheta+}, \eqref{PertPhi}, \eqref{IC-diff}, \eqref{new-interface2}, \eqref{new-interface4}-\eqref{core-2}, we can
write the final problem for ${\bf u}=(u,w)$, which is {\it fully nonlinear} since the nonlinear part of the differential equations contains traces at $\xi=0^+$ and $R$ of (first- and) second-order derivatives of the unknown $\uu$ itself.

Summing up, the pair ${\bf u}=(u,w)$ solves the nonlinear system
\begin{align}
\left\{
\begin{array}{lll}
D_t{\bf u}(t,\cdot,\cdot)={\mathscr L}{\bf u}(t,\cdot,\cdot)+\mathscr F({\bf u}(t,\cdot,\cdot)),& t\geq0,\\[1mm]
\mathscr B({\bf u}(t,\cdot))=\boldsymbol{\mathscr H}(\uu(t,\cdot)), & t\geq0,
\end{array}
\right.
\label{asta}
\end{align}
and satisfies periodic boundary conditions at $y=\pm\ell/2$, where
\begin{equation}
{\mathscr L}{\bf v}=(\Delta\zeta+\zeta_x,\Le^{-1}\Delta \upsilon+\upsilon_x),
\label{operatore-L}
\end{equation}
\begin{equation}
{\mathscr B}{\bf v}=\left (
\begin{array}{l}
\zeta(0^+,\cdot)-\zeta(0^-,\cdot)\\[1mm]
\zeta(R^+,\cdot)-\zeta(R^-,\cdot)\\[1mm]
{\rm Le}[\zeta_x(0^+,\cdot)-\zeta_x(0^-,\cdot)]+\upsilon(0^+,\cdot)\\[1mm]
{\rm Le}\,\upsilon(0^+,\cdot)+\upsilon_x(0^+,\cdot)\\[1mm]
\zeta(R^+,\cdot)+\theta_iR[\zeta_x(R^+,\cdot)-\zeta_x(R^-,\cdot)]\\[1mm]
{\rm Le}[\zeta_x(R^+,\cdot)-\zeta_x(R^-,\cdot)]+\upsilon(R^+,\cdot)-\upsilon(R^-,\cdot)\\[1mm]
{\rm Le}[\upsilon(R^+,\cdot)-\upsilon(R^-,\cdot)]+\upsilon_x(R^+,\cdot)-\upsilon_x(R^-,\cdot)
\end{array}
\right ),
\label{boundary-B}
\end{equation}
\begin{align*}
\mathscr F({\bf v})
= & (\Lambda({\bf v})\mathscr F_1({\bf v})-\mathscr F_2({\bf v}),
\Lambda({\bf v})D_x\mathscr G_1({\bf v})+D_x\Lambda({\bf v})\mathscr G_1({\bf v})-{\rm Le}^{-1}D_x\mathscr G_2({\bf v})), \\[1mm]
{\mathscr F}_1({\bf v})=& \frac{\theta_i^{-1}\beta_R\Theta^{(0)}_{xx}\zeta(R^+,\cdot)-R{\rm Le}^{-1}\beta\Theta^{(0)}_{xx}\upsilon(0,\cdot)+\zeta_x}{1+\theta_i^{-1}\beta'(\cdot-R)\zeta(R^+,\cdot)-R{\rm Le}^{-1}\beta'\upsilon(0,\cdot)},\\[2mm]
{\mathscr F}_2({\bf v})=&\{\big (\theta_i^{-1}\beta'(\cdot-R)\zeta(R^+,\cdot)-R{\rm Le}^{-1}\beta'\upsilon(0,\cdot)+\theta_i^{-1}\beta_R \zeta_{yy}(R^+,\cdot)-R{\rm Le}^{-1}\beta \upsilon_{yy}(0,\cdot)\big )\\
&\qquad\quad\times(\theta_i^{-1}\beta_R\Theta^{(0)}_{xx}\zeta(R^+,\cdot)-R{\rm Le}^{-1}\beta\Theta^{(0)}_{xx}\upsilon(0,\cdot)+\zeta_x)\\
&\;\,+2(\theta_i^{-1}\beta_R \zeta_y(R^+,\cdot)-R{\rm Le}^{-1}\beta \upsilon_y(0,\cdot))^2\Theta^{(0)}_{xx}\\
&\;\,+2(\theta_i^{-1}\beta_R \zeta_y(R^+,\cdot)-R{\rm Le}^{-1}\beta \upsilon_y(0,\cdot))\zeta_{xy}\}(1+\theta_i^{-1}\beta'(\cdot-R)\zeta(R^+,\cdot)-R{\rm Le}^{-1}\beta'\upsilon(0,\cdot))^{-1}\\
&-\big\{2\big (\theta_i^{-1}\beta_R \zeta_y(R^+,\cdot)-R{\rm Le}^{-1}\beta \upsilon_y(0,\cdot)\big )\big (\theta_i^{-1}\beta'(\cdot-R)\zeta_y(R^+,\cdot)-R{\rm Le}^{-1}\beta'\upsilon_y(0,\cdot)\big )\\
&\qquad\quad\times
\big (\theta_i^{-1}\beta_R\Theta^{(0)}_{xx}\zeta(R^+,\cdot)-R{\rm Le}^{-1}\beta\Theta^{(0)}_{xx}\upsilon(0,\cdot)+\zeta_x\big )\\
&\;\;\;\;\;\;+\big [\big (\theta_i^{-1}\beta_R \zeta_y(R^+,\cdot)-R{\rm Le}^{-1}\beta \upsilon_y(0,\cdot)\big )^2-\big (\theta_i^{-1}\beta'(\cdot-R)\zeta(R^+,\cdot)-R{\rm Le}^{-1}\beta'\upsilon(0,\cdot)\big )^2\big ]\\
&\qquad\quad\;\;\;\;\times\big (\theta_i^{-1}\beta_R \Theta^{(0)}_{xxx}\zeta(R^+,\cdot)-R{\rm Le}^{-1}\beta \Theta^{(0)}_{xxx}\upsilon(0,\cdot)+\Theta^{(0)}_{xx}+\zeta_{xx}\big )\\
&\;\;\;\;\;\;-2\big (\theta_i^{-1}\beta'(\cdot-R)\zeta(R^+,\cdot)-R{\rm Le}^{-1}\beta'\upsilon(0,\cdot)\big )\\
&\qquad\quad\;\;\;\;\times\big [\theta_i^{-1}\beta_R\Theta^{(0)}_{xxx}\zeta(R^+,\cdot)-R{\rm Le}^{-1}\beta\Theta^{(0)}_{xxx} \upsilon(0,\cdot)\\
&\qquad\qquad\quad\;\;\;\;-\big (\theta_i^{-1}\beta_R \zeta_y(R^+,\cdot)-R{\rm Le}^{-1}\beta \upsilon_y(0,\cdot)\big )^2\Theta^{(0)}_{xx}+\zeta_{xx}\big ]\big\}\\
&\qquad\qquad\;\;\;\;\;\;\times(1+\theta_i^{-1}\beta'(\cdot-R)\zeta(R^+,\cdot)-R{\rm Le}^{-1}\beta'\upsilon(0,\cdot))^{-2}\\
&+\big\{\big (\theta_i^{-1}\beta''(\cdot-R)\zeta(R^+,\cdot)-R{\rm Le}^{-1}\beta''\upsilon(0,\cdot)\big )\big [1+\big (\theta_i^{-1}\beta_R \zeta_y(R^+,\cdot)-R{\rm Le}^{-1}\beta \upsilon_y(0,\cdot)\big )^2\big ]\\
&\qquad\quad\times (\theta_i^{-1}\beta_R\Theta^{(0)}_{xx}\zeta(R^+,\cdot)-R{\rm Le}^{-1}\beta\Theta^{(0)}_{xx}\upsilon(0,\cdot)+\zeta_x)\big\}\\
&\qquad\qquad\qquad\times(1+\theta_i^{-1}\beta'(\cdot-R)\zeta(R^+,\cdot)-R{\rm Le}^{-1}\beta'\upsilon(0,\cdot))^{-3},\\[2mm]
\mathscr G_1({\bf v})= &\frac{\theta_i^{-1}\beta_R\Phi^{(0)}_{xx}\zeta(R^+,\cdot)-R{\rm Le}^{-1}\beta\Phi^{(0)}_{xx}\upsilon(0,\cdot)+\upsilon}{1+\theta_i^{-1}\beta'(\cdot-R)\zeta(R^+,\cdot)-R{\rm Le}^{-1}\beta'\upsilon(0,\cdot)},\\[2mm]
\mathscr G_2({\bf v})= &
\{[{\rm Le}\,\theta_i^{-1}\beta'(\cdot-R)\zeta(R^+,\cdot)-R\beta'\upsilon(0,\cdot)+\theta_i^{-1}\beta_R \zeta_{yy}(R^+,\cdot)-R{\rm Le}^{-1}\beta \upsilon_{yy}(0,\cdot)]\\
&\qquad\quad\times[\theta_i^{-1}\beta_R\Phi^{(0)}_{xx}\zeta(R^+,\cdot)-R{\rm Le}^{-1}\beta\Phi^{(0)}_{xx}\upsilon(0,\cdot)+\upsilon]\\
&\;\,+2(\theta_i^{-1}\beta_R \zeta_y(R^+,\cdot)-R{\rm Le}^{-1}\beta \upsilon_y(0,\cdot))^2\Phi^{(0)}_{xx}\\
&\;\,+2(\theta_i^{-1}\beta_R \zeta_y(R^+,\cdot)-R{\rm Le}^{-1}\beta \upsilon_y(0,\cdot))\upsilon_y\}
(1+\theta_i^{-1}\beta'(\cdot-R)\zeta(R^+,\cdot)-R{\rm Le}^{-1}\beta'\upsilon(0,\cdot))^{-1}\\
&-\big\{2\big (\theta_i^{-1}\beta_R \zeta_y(R^+,\cdot)-R{\rm Le}^{-1}\beta \upsilon_y(0,\cdot)\big )\big (\theta_i^{-1}\beta'(\cdot-R)\zeta_y(R^+,\cdot)-R{\rm Le}^{-1}\beta'\upsilon_y(0,\cdot)\big )\\
&\qquad\quad\times
\big (\theta_i^{-1}\beta_R\Phi^{(0)}_{xx}\zeta(R^+,\cdot)-R{\rm Le}^{-1}\beta\Phi^{(0)}_{xx}\upsilon(0,\cdot)+\upsilon\big )\\
&\;\;\;\;\;\;+\big [\big (\theta_i^{-1}\beta_R \zeta_y(R^+,\cdot)-R{\rm Le}^{-1}\beta \upsilon_y(0,\cdot)\big )^2-\big (\theta_i^{-1}\beta'(\cdot-R)\zeta(R^+,\cdot)-R{\rm Le}^{-1}\beta'\upsilon(0,\cdot)\big )^2\big ]\\
&\qquad\quad\;\;\;\;\times\big (\theta_i^{-1}\beta_R\Phi^{(0)}_{xxx}\zeta(R^+,\cdot)-R{\rm Le}^{-1}\beta \Phi^{(0)}_{xxx}\upsilon(0,\cdot)+\Phi^{(0)}_{xx}+\upsilon_x\big )\\
&\;\;\;\;\;\;-2\big (\theta_i^{-1}\beta'(\cdot-R)\zeta(R^+,\cdot)-R{\rm Le}^{-1}\beta'\upsilon(0,\cdot)\big )\\
&\qquad\quad\;\;\;\;\times\big [\theta_i^{-1}\beta_R\Phi^{(0)}_{xxx}\zeta(R^+,\cdot)-R{\rm Le}^{-1}\beta\Phi^{(0)}_{xxx} \upsilon(0,\cdot)\\
&\qquad\qquad\quad\;\;\;\;-\big (\theta_i^{-1}\beta_R \zeta_y(R^+,\cdot)-R{\rm Le}^{-1}\beta \upsilon_y(0,\cdot)\big )^2\Phi^{(0)}_{xx}+\upsilon_x\big ]\big\}\\
&\qquad\qquad\;\;\;\;\;\;\times(1+\theta_i^{-1}\beta'(\cdot-R)\zeta(R^+,\cdot)-R{\rm Le}^{-1}\beta'\upsilon(0,\cdot))^{-2}\\
&+\big\{\big (\theta_i^{-1}\beta''(\cdot-R)\zeta(R^+,\cdot)-R{\rm Le}^{-1}\beta''\upsilon(0,\cdot)\big )\big [1+\big (\theta_i^{-1}\beta_R\zeta_y(R^+,\cdot)-R{\rm Le}^{-1}\beta \upsilon_y(0,\cdot)\big )^2\big ]\\
&\qquad\quad\times (\theta_i^{-1}\beta_R\Phi^{(0)}_{xx}\zeta(R^+,\cdot)-R{\rm Le}^{-1}\beta\Phi^{(0)}_{xx}\upsilon(0,\cdot)+\upsilon)\big\}\\
&\qquad\qquad\qquad\times(1+\theta_i^{-1}\beta'(\cdot-R)\zeta(R^+,\cdot)-R{\rm Le}^{-1}\beta'\upsilon(0,\cdot))^{-3},\\[1mm]
\Lambda({\bf v})=& (\theta_i\!-\!\zeta(R^+,\cdot)\!-\!\zeta_x(R^+,\cdot))^{-1}\\
&\qquad\quad\times \big [\zeta_x(R^+,\cdot)+ \Delta \zeta(R^+,\cdot)-\theta_i^{-1}\zeta_{yy}(R,\cdot)\zeta_x(R^+,\cdot)-2\theta_i^{-1}\zeta_y(R,\cdot)\zeta_{xy}(R^+,\cdot)\\
&\qquad\qquad\;\,+\theta_i^{-2}(\zeta_{y}(R,\cdot))^2(\zeta_{xx}(R^+,\cdot)-\zeta(R,\cdot)-\theta_i)-\theta_i^{-1}\zeta(R,\cdot)\zeta_{yy}(R,\cdot)\big ]\\
&-[{\rm Le}^2+R\Le({\rm Le}\, \upsilon(0^+,\cdot)+\upsilon_x(0^+,\cdot))]^{-1}\\
&\qquad\quad\times \big\{\Le\, \upsilon_x(0^+,\cdot) +\Delta \upsilon (0^+,\cdot)\\
&\qquad\qquad\;\,+R\Le^{-1}\big [\upsilon_{yy}(0^+,\cdot)(\Le\, \upsilon(0^+,\cdot)+\upsilon_{x}(0^+,\cdot))+2\upsilon_y(0^+,\cdot) \upsilon_{xy}(0^+,\cdot)\big ]\notag\\
&\qquad\qquad\;\,+R^2\Le^{-2}(\upsilon_{y}(0^+,\cdot))^2(-\Le^2\upsilon(0^+,\cdot)+R^{-1}\Le^2+\upsilon_{xx}(0^+,\cdot))\big\},\\[2mm]
\mathscr H_j({\bf v})=&0~{\rm if}~j\neq 4,7,\qquad
\mathscr H_4({\bf v})=-\frac{R{\rm Le}(\upsilon_y(0,\cdot))^2}{{\rm Le}^2+R^2(\upsilon_y(0,\cdot))^2},\qquad
\mathscr H_7({\bf v})= \frac{{\rm Le}(\zeta_y(R^+,\cdot))^2}{R(\theta_i^2+(\zeta_y(R^+,\cdot))^2)},
\end{align*}
on smooth enough functions ${\bf v}=(\zeta,\upsilon)$, where $\beta_R=\beta(\cdot-R)$.

\begin{rmk}
\label{rmk-opB}
{\rm Note that each smooth enough function $\uu$, which solves problem \eqref{asta}, has its first component $u_1$ which is continuous on $\{R\}\times [-\ell/2,\ell/2]$.
Therefore, the operator ${\mathscr B}$ can be replaced with the operator $\widetilde {\mathscr B}$ which is defined as ${\mathscr B}$ with the fifth equation being replaced by the condition
$\frac{1}{2}(v_1(R^+,\cdot)+v_1(R^-,\cdot))+\theta_iR[D_xv_1(R^+,\cdot)-D_xv_1(R^+,\cdot)]=0$.}
\end{rmk}

We will use the above remark in Subsection \ref{subsection-lifting}.

\section{Tools}
\label{sect-4}
In this section we collect some technical results which are used in the next (sub)sections.

\subsection{Preliminary results needed to prove Theorems \ref{banca} and Proposition \ref{perdita}}

We find it convenient to set
\begin{align*}
({\mathcal F}_{k,\rho}f)(x):=\frac{1}{Z_k}\int_{\R}e^{-\frac{\rho}{2}s}e^{-\frac{1}{2}Z_k|s|}\hat f_k(x-s)ds,\qquad\;\,x\in\R,\;\,f\in C_b(\overline{S};\C),\;\,k\in\Z,
\end{align*}
where $Z_k=Z_k(\lambda,\rho)=\sqrt{\rho^2+4\lambda\rho+4\lambda_{k}}$ for every $k\in\Z$ and
$\lambda_k=(4\ell^2)^{-1}k^2\pi^2$. Moreover, we denote by $\Sigma_0$ the set of $\lambda\in\C$ with positive real part.

\begin{lemm}
\label{finito}
For every $\rho\in (0,+\infty)$, $\lambda\in\C$ with ${\rm Re}\lambda>-\rho^{-1}({\rm Im}\lambda)^2$ and $f\in C_b(\overline{S};\C)$, the series
$\ell^{-1}\sum_{k\in\Z}(\mathcal F_{k,\rho}f)e_k$
defines a bounded and continuous function ${\mathscr R}_{\lambda,\rho}f$ in $\R^2$ which, clearly, is periodic with respect to $y$.
Moreover,
\begin{enumerate}[\rm (i)]
\item
${\mathscr R}_{\lambda,\rho}f\in \bigcap_{p<+\infty} W^{2,p}_{\rm loc}(\R^2;\C)$ and $\lambda {\mathscr R}_{\lambda,\rho}f-\rho^{-1}\Delta {\mathscr R}_{\lambda,\rho}f-D_x{\mathscr R}_{\lambda,\rho}f=f$ in $S$;
\item
$\nabla {\mathscr R}_{\lambda,\rho}f\in C_b(\R^2;\C)\times C_b(\R^2;\C)$ and
\begin{equation}
|\lambda|\|{\mathscr R}_{\lambda,\rho}f\|_{\infty}+\sqrt{|\lambda|}\|\nabla{\mathscr R}_{\lambda,\rho}f\|_{\infty}\le c\|f\|_{\infty},\qquad\;\,\lambda\in\Sigma_0;
\label{braccio}
\end{equation}
\item
if further $\lim_{x\to -\infty}f(x,y)=0$ {\rm (resp.} $\lim_{x\to +\infty}f(x,y)=0)$ for every $y\in [-\ell/2,\ell/2]$, then $({\mathscr R}_{\lambda,\rho}f)(\cdot,y)$
vanishes as $x\to -\infty$ {\rm (resp.} $x\to +\infty)$ for every $y\in\R$;
\item
for every $f\in C^{\alpha}_b(S;\C)$, such that $f(\cdot,-\ell/2)=f(\cdot,\ell/2)$, and $\lambda\in\C$ with ${\rm Re}\lambda>-\rho^{-1}({\rm Im}\lambda)^2$, the function $\mathscr R_{\lambda,\rho}f$ admits classical derivatives up to the second-order which belong to $C^{\alpha}_b(\R^2;\C)$. Moreover,
\begin{align}
\label{cotto}
\|\mathscr R_{\lambda,\rho}f\|_{C^{2+\alpha}_b(\R^2;\C)}\leq c_{\lambda}\|f\|_{C^{\alpha}_b(S;\C)}.
\end{align}
\end{enumerate}
\end{lemm}

\begin{proof}
To begin with, we claim that, for each $f\in C_b(\overline{S};\C)$ and $\lambda\in\C$, such that $\rho{\rm Re}\lambda+({\rm Im}\lambda)^2>0$, the series in the statement converges uniformly in $\R^2$. To prove the claim, we observe that ${\rm Re}(Z_k)>\rho$ and $|Z_k|\ge {\rm Re}(Z_k)\ge c_{\lambda}(k+1)$ for every $k\in\Z$. Thus, we can estimate
\begin{align}
|(\mathcal F_{k,\rho}f)(x)|\le \sup_{x\in\R}|\hat f_k(x)|
\frac{1}{|Z_k|}\int_{\R}e^{\frac{\rho-{\rm Re}(Z_k)}{2}|s|}ds
\le \frac{c_{\lambda}}{k^2+1}\|f\|_{C_b(\overline{S};\C)},\qquad\;\,x\in\R,\;\,k\in\Z,
\label{spugna}
\end{align}
and this is enough to infer that the series converges locally uniformly on $\R^2$ and, as a byproduct, that the operator ${\mathscr R}_{\lambda,\rho}$ is bounded from $C_b(\overline{S};\C)$ into $C_b(\R^2;\C)$. Moreover, if $f(\cdot,y)$ vanishes at $-\infty$ (resp. $+\infty$) for each $y\in [-\ell/2,\ell/2]$, then by dominated convergence the function ${\mathcal F}_kf$ vanishes at $-\infty$ (resp. $+\infty$) and, in view of the uniform convergence of the series which defines the function ${\mathscr R}_{\lambda,\rho}$, this is enough to conclude that this latter function tends to $0$ as $x\to-\infty$ (resp. $x\to+\infty$) for each $y\in\R$.

Now, we prove properties (i), (ii) and (iv).

(i) Let us prove that the function ${\mathscr R}_{\lambda,\rho} f$ is the unique solution to the equation $\lambda u-\rho^{-1}\Delta u-D_xu=f^{\sharp}$ in
${\mathcal D}=\{u\in C^1_b(\R^2;\C)\cap \bigcap_{p<+\infty}W^{2,p}_{\rm loc}(\R^2;\C): \Delta u\in L^{\infty}(\R^2;\C)\cap C_b(\overline{S};\C), u(\cdot,\cdot+\ell)=u\}$.
For this purpose, for every $n\in\N$ we introduce the functions $u_n=\ell^{-1}\sum_{k=-n}^n(\mathcal F_{k,\rho}f)e_k$ and $f_n=\ell^{-1}\sum_{k=-n}^n\hat f(\cdot,k)e_k$.
Note that $\lambda u_n-\rho^{-1}\Delta u_n-D_xu_n=f_n$ in $\R^2$, for every $n\in\N$, since the function ${\mathcal F}_{k,\rho}f$ $(k\in\Z)$ solves the differential equation $(\lambda+\rho^{-1}\lambda_k)w-\rho^{-1}w''-w'=f_k$ in $\R$. Thus,
\begin{align*}
\langle u_n,\rho^{-1}\Delta\varphi-\varphi_x\rangle
=\int_{\R^2}u_n(\rho^{-1}\Delta\varphi-\varphi_x)dxdy=\int_{\R^2}(\lambda u_n-f_n)\varphi dxdy
=:\langle \lambda u_n-f_n,\varphi\rangle
\end{align*}
for every $\varphi\in C^{\infty}_c(\R^2;\C)$.
Letting $n$ tend to $+\infty$ and applying the dominated convergence theorem, it follows that ${\mathscr R}_{\lambda,\rho}f$ is a distributional solution to the equation $\lambda {\mathscr R}_{\lambda,\rho}f-\rho^{-1}\Delta {\mathscr R}_{\lambda,\rho}f-D_x {\mathscr R}_{\lambda,\rho} f=f^{\sharp}$. By elliptic regularity (see e.g., \cite{gilbarg}), we can infer that ${\mathscr R}_{\lambda,\rho}f\in\bigcap_{p<+\infty}W^{2,p}_{\rm loc}(\R^2;\C)$.
Since ${\mathscr R}_{\lambda,\rho}f$ is bounded and continuous in $\R^2$ and $\Delta {\mathscr R}_{\lambda,\rho}f+D_x{\mathscr R}_{\lambda,\rho}f=\lambda {\mathscr R}_{\lambda,\rho}f-f^{\sharp}$ belongs to $L^{\infty}(\R^2;\C)$, again by classical results we can infer that
${\mathscr R}_{\lambda,\rho}f\in C^{1+\gamma}_b(\R^2;\C)$ for each $\gamma\in (0,1)$ and, as a byproduct, that $\Delta {\mathscr R}_{\lambda,\rho} f\in L^{\infty}(\R^2;\C)\cap C_b(\overline{S};\C)$. We can thus conclude that ${\mathscr R}_{\lambda,\rho} f$ belongs to ${\mathcal D}$.

To prove uniqueness, we assume that $v$ is another solution in ${\mathcal D}$ of the equation $\lambda u-\rho^{-1}\Delta u-u_x=f^{\sharp}$. The smoothness of $v$ implies that, for each $k\in\Z$, the function $\hat v_k$ belongs to $C^1_b(\R;\C)$. Moreover, integrating by parts we obtain that
\begin{align*}
\int_{\R}\hat v_k\varphi''dx=&\frac{1}{\ell}\int_{S}\varphi''v\overline{e_k}dx dy
=\frac{1}{\ell}\int_{S}\varphi v_{xx}\overline{e_k}dx dy
\end{align*}
for each $\varphi\in C^{\infty}_c(\R)$.
By Fubini theorem, $\hat v_k$ belongs to $W^{2,p}_{\rm loc}(\R;\C)$.
Since $\lambda v-\rho^{-1}\Delta v-v_x=f$ in $S$, we can write
\begin{align}
\int_{S}\varphi(x)v_{xx}\overline{e_k}dx dy
=&\lim_{n\to +\infty}\int_{S}\varphi(x)v_{xx}(x,y)\psi_n(y)\overline{e_k(y)}dx dy\notag\\
=&\lambda\rho\int_{\R}\varphi v_kdx
-\lim_{n\to +\infty}\int_{S}\varphi v_{yy}\psi_n\overline{e_k}dx dy-\rho\int_{\R}\varphi\hat v'_kdx
-\rho\int_{\R}\varphi f_kdx,
\label{clonata-1}
\end{align}
where $\psi_n(y)=\psi(n|y|/\ell+1-n/2)$ for each $y\in [-\ell/2,\ell/2)$, $n\in\N$, and $\psi$ is a smooth function such that $\psi=1$ in $(-\infty,1/2]$ and $\psi=0$ outside $(-\infty,3/4]$. Clearly, $\psi_n$ converges to $1$ in $L^1((-\ell/2,\ell/2))$ as $n$ tends to $+\infty$. An integration by parts shows that
\begin{align}
&\lim_{n\to+\infty}\int_{S}\varphi(x)v_{yy}(x,y)\psi_n(y)\overline{e_k(y)}dx dy\notag\\
=&-\lim_{n\to+\infty}\frac{n}{\ell}\int_{\R^2}\varphi(x)\chi_{A}(x){\rm signum}(y)v_y(x,y)\psi'\bigg (\frac{n}{\ell}|y|+1-\frac{n}{2}\bigg )
\chi_{B_n}(y)\overline{e_k(y)}dx dy\notag\\
&-\lambda_k\int_{\R}\varphi(x)v_k(x)dx,
\label{clonata-2}
\end{align}
where $A={\rm supp}(\varphi)$ and $B_n=\left\{y\in\R: \frac{\ell}{2}-\frac{\ell}{2n}\le |y|\le \frac{\ell}{2}-\frac{\ell}{4n}\right\}$ for every $n\in\N$.
We claim that the first term in the last side of \eqref{clonata-2} is zero. For this purpose, we split
\begin{align*}
&\frac{n}{\ell}\int_{\R^2}\varphi(x)\chi_A(x){\rm signum}(y)v_y(x,y)\psi'\bigg (\frac{n}{\ell}|y|+1-\frac{n}{2}\bigg )
\chi_{B_n}(y)\overline{e_k(y)}dx dy\\
=&\frac{n}{\ell}\int_{\R^2}\varphi(x)\chi_{A}(x){\rm signum}(y)(v_y(x,y)-v_y(x,\ell/2))\psi'\bigg (\frac{n}{\ell}|y|+1-\frac{n}{2}\bigg )
\chi_{B_n}(y)\overline{e_k(y)}dx dy\\
&+\frac{n}{\ell}\int_A\varphi(x)v_y(x,\ell/2)dx\int_{\R}{\rm signum}(y)\psi'\bigg (\frac{n}{\ell}|y|+1-\frac{n}{2}\bigg )
\chi_{B_n}(y)\overline{e_k(y)}dy\\
=&\!:{\mathcal K}_{1,n}(x,y)+{\mathcal K}_{2,n}(x,y)
\end{align*}
for every $(x,y)\in\R^2$.
Since $v$ is $\ell$-periodic with respect to $y$, $D_yv$ is $\ell$-periodic with respect to $y$ as well. Moreover, this latter function is $1/2$-H\"older continuous in $\R^2$ since
it belongs to ${\mathcal D}$. Hence, we can estimate
\begin{align*}
|v_y(x,y)-v_y(x,\ell/2)|=|v_y(x,y)-v_y(x,-\ell/2)|
\le [v_y]_{C^{1/2}(A\times [-\ell/2,\ell/2];\C)}\min\{|y-\ell/2|,|y+\ell/2|\}^{\frac{1}{2}}
\end{align*}
for every $x\in A$ and $y\in [-\ell/2,\ell/2]$.
Thus,
\begin{eqnarray*}
|{\mathcal K}_{1,n}|\le c\|\varphi\|_{\infty}\|\psi'\|_{\infty}n^{\frac{1}{2}}m(A\times B_n)
\le c\|\varphi\|_{\infty}\|\psi'\|_{\infty}n^{-\frac{1}{2}},\qquad\;\,n\in\N,
\end{eqnarray*}
where $m(A\times B_n)$ denotes the Lebesgue measure of the set $A\times B_n$,
so that ${\mathcal K}_{1,n}$ vanishes as $n$ tends to $+\infty$. As far as ${\mathcal K}_{2,n}$ is concerned, we observe that
\begin{align*}
{\mathcal K}_{2,n}(x,y)
=&-\frac{2in}{\ell}\int_AD_yv(x,\ell/2)\varphi(x)dx\int_{\frac{\ell}{2}-\frac{\ell}{2n}}^{\frac{\ell}{2}-\frac{\ell}{4n}}\psi'\bigg (\frac{n}{\ell}y+1-\frac{n}{2}\bigg )\sin\bigg (\frac{2k\pi}{\ell}y\bigg )dy\\
=&c\bigg [\frac{4k\pi i}{\ell}\int_{\frac{\ell}{2}-\frac{\ell}{2n}}^{\frac{\ell}{2}-\frac{\ell}{4n}}\psi\bigg (\frac{n}{\ell}y+1-\frac{n}{2}\bigg )\cos\bigg (\frac{2k\pi}{\ell}y\bigg )dy-2(-1)^k\sin\left (\frac{k\pi}{n}\right )\bigg ].
\end{align*}
Hence, ${\mathcal K}_{2,n}$ vanishes as $n$ tends to $+\infty$.
The claim is so proved.

From \eqref{clonata-1} and \eqref{clonata-2}, we can now infer that
$\rho^{-1}\hat v_k''=(\lambda+\rho^{-1}\lambda_k)\hat v_k-\hat v_k'-\hat f_k$. Thus, $v_k={\mathcal F}_kf$ for every $k\in\Z$ and, as a byproduct, we deduce that $v={\mathscr R}_{\lambda,\rho}f$.

(ii) By classical results, the realization $A_{\rho}$ in $L^{\infty}(\R^2;\C)$ of the operator $\rho^{-1}\Delta+D_x$, with domain
 $D(A_{\rho})=\{u\in C^1_b(\R^2;\C)\cap\bigcap_{p<+\infty}W^{2,p}_{\rm loc}(\R^2;\C): \Delta u\in L^{\infty}(\R^2;\C)\}\supset {\mathcal D}$, generates an analytic semigroup. Moreover, $\|R(\lambda,A_{\rho})\|_{L(L^{\infty}(\R^2))}\le c|\lambda|^{-1}$ for every $\lambda\in\Sigma_0$ and $\|\nabla v\|_{\infty}\le c\|v\|_{\infty}^{1/2}\|A_\rho v\|_{\infty}^{1/2}$ for every $v\in D(A_{\rho})$.
Note that the function $R(\lambda,A_{\rho})f^{\sharp}$ is $\ell$-periodic with respect to $y$. Indeed,
$R(\lambda,A_{\rho})f^{\sharp}$ and $(R(\lambda,A_{\rho})f^{\sharp})(\cdot,\cdot+\ell)$ both solve (in $D(A_{\rho})$) the equation $\lambda u-\rho^{-1}\Delta u-u_x=f^{\sharp}$ and, by uniqueness, they coincide. Finally,  since $f^{\sharp}$ is continuous in $\overline{S}$, $\Delta R(\lambda,A_{\rho})f^{\sharp}$ belongs to $C_b(\overline{S};\C)$.
Hence, $R(\lambda,A_{\rho})f^{\sharp}\in {\mathcal D}$ and, by $(i)$, it coincides with ${\mathscr R}_{\lambda,\rho}f^{\sharp}$. Using the above estimates, inequality \eqref{braccio} follows immediately.

(iv) Since $f\in C^{\alpha}_b(S;\C)$ and $f(\cdot,-\ell/2)=f(\cdot,\ell/2)$,
the function $f^{\sharp}$ belongs to $C^{\alpha}_b(\R^2;\C)$. Hence, $\Delta{\mathscr R}_{\lambda,\rho}f=\rho(\lambda {\mathscr R}_{\lambda,\rho}f- f^{\sharp}-D_x{\mathscr R}_{\lambda,\rho}f)$ is an element of $C^{\alpha}_b(\R^2;\C)$.
Classical results (see e.g., \cite{krylov}) yield \eqref{cotto}.
\end{proof}

\begin{lemm}
\label{finire}
For $g\in C_b(\overline{S_0^+};\C)$ and $\lambda\!\in\!\C$, such that ${\rm Re}\lambda>-{\rm Le}^{-1}({\rm Im}\lambda)^2$, the function ${\mathscr S}_{\lambda}g=
\ell^{-1}\sum_{k\in\Z}({\mathcal G}_kg)e_k$ is bounded and continuous in $\R^2$.
Here, ${\mathcal G}_kg:={\mathcal F}_{k,{\rm Le}}(\overline g)$ for each $k\in\Z$, and $\overline{g}$
is the trivial extension of $g$ to $\R\times [-\ell/2,\ell/2]$. Moreover,
\begin{enumerate}[\rm (i)]
\item
${\mathscr S}_{\lambda}g$ belongs to $\bigcap_{p<+\infty} W^{2,p}_{\rm loc}(\R^2;\C)$ and $\lambda {\mathscr S}_{\lambda}g-{\rm Le}^{-1}\Delta {\mathscr S}_{\lambda}g-D_x{\mathscr S}_{\lambda}g=g$ in $S_0^+$;
\item
there exists a positive constant $c_1$, independent of $\lambda$, such that
\begin{equation}
|\lambda|\|{\mathscr S}_{\lambda}g\|_{\infty}+\sqrt{|\lambda|}\|\nabla{\mathscr S}_{\lambda}g\|_{\infty}\le c_1\|g\|_{\infty},\qquad\;\,\lambda\in\Sigma_0;
\label{braccio-3}
\end{equation}
\item
if further $\lim_{x\to +\infty}g(x,y)=0$ for each $y\in [-\ell/2,\ell/2]$, then $({\mathscr S}_{\lambda}g)(\cdot,y)$
vanishes as $x\to +\infty$ for each $y\in\R$;
\item
if $g\in C^{\alpha}_b(S_0^+;\C)$ and $g(\cdot,-\ell/2)=g(\cdot,\ell/2)$, then ${\mathscr S}_{\lambda}g$ belongs to $C^{2+\alpha}_b(\R^2;\C)$
for every $\lambda\in\C$ such that ${\rm Re}\lambda>-{\rm Le}^{-1}({\rm Im}\lambda)^2$. Moreover,
\begin{align}
\label{cotto-11}
\|\mathscr S_{\lambda}g\|_{C^{2+\alpha}_b(\R^2;\C)}\leq c_{2,\lambda}\|g\|_{C^{\alpha}_b(S_0^+;\C)},
\end{align}
with the constant $c_{2,\lambda}$ being independent of $g$.
\end{enumerate}
\end{lemm}

\begin{proof}
We split the proof into two steps: in the first one we prove properties (i), (ii) and (iii) and in the second step we prove property (iv).

{\em Step 1}. Let us fix $g$ and $\lambda$ as in the statement. Arguing as in the first part of the proof of Lemma \ref{finito}, taking the continuity of the functions
${\mathcal G}_kg$ $(k\in\Z)$ into account, it can be checked that the series in the statement converges uniformly in $\R^2$, so that the function ${\mathscr S}_{\lambda}g$ is well defined and
it vanishes as $x\to +\infty$ for every $y\in \R$, if $\lim_{x\to +\infty}g(x,y)=0$ for every $y\in [-\ell/2,\ell/2]$.

To check properties (i) and (ii), for each $n\in\N$ we set $g_n:=\overline{g}\star_x\psi_n$, where $\star_x$ stands for convolution with respect to the variable $x$ and $(\psi_n)$ is a standard sequence of mollifiers.
Clearly, ${\mathscr S}_{\lambda}g_n={\mathscr R}_{\lambda,{\rm Le}}g_n$.
The sequence $(g_n)$ converges to $\overline{g}$ pointwise in $\R^2$ as $n\to +\infty$ and $\|g_n\|_{\infty}\le\|g\|_{\infty}$ for each $n\in\N$. Thus, we can infer that ${\mathcal F}_{k,\rho}g_n$ converges to ${\mathcal G}_kg$ pointwise in $\R$ as $n$ tends to $+\infty$, for every $k\in\N$ and, by dominated convergence, ${\mathscr R}_{\lambda,{\rm Le}}g_n$ tends to ${\mathscr S}_{\lambda}g$ pointwise in $\R^2$.

Applying the classical interior $L^p$-estimates for the operator ${\rm Le}^{-1}\Delta +D_x$ and using \eqref{braccio}, which allows us to write
\begin{align}
& |\lambda|\|\nabla{\mathscr R}_{\lambda, {\rm Le}}g_n\|_{\infty}+
|\lambda|^{\frac{1}{2}}\|\nabla{\mathscr R}_{\lambda, {\rm Le}}g_n\|_{\infty} \leq c\|g_n\|_{\infty}\leq c\|g\|_{\infty},\qquad\;\;n\in\N,
\label{nessuno}
\end{align}
we can estimate
\begin{align*}
\|{\mathscr R}_{\lambda,{\rm Le}} g_n\|_{W^{2,p}(B(0,r);\C)}\le &c_{p,r}(\|{\mathscr R}_{\lambda,{\rm Le}} g_n\|_{L^p(B(0,2r);\C)}+\|{\rm Le}^{-1}\Delta{\mathscr R}_{\lambda,{\rm Le}}g_n+D_x{\mathscr R}_{\lambda,{\rm Le}}g_n\|_{L^p(B(0,2r);\C)})\\
\le &c_{p,r}[(1+|\lambda|)\|{\mathscr R}_{\lambda,{\rm Le}}g_n\|_{C(\overline{B(0,2r)};\C)}+\|g_n\|_{C(\overline{B(0,2r)};\C)}]\\
\le &c_{p,r,\lambda}\|g_n\|_{\infty}\le c_{p,r,\lambda}\|g\|_{\infty}
\end{align*}
for every $p\in [1,+\infty)$ and $r>0$. Hence, by compactness, we conclude that ${\mathscr R}_{\lambda,{\rm Le}} g_n$ converges to ${\mathscr S}_{\lambda}g$ in $C^1(\overline{B(0,r)};\C)$ for each $r>0$, ${\mathscr S}_{\lambda}g\in W^{2,p}_{\rm loc}(\R^2;\C)$ for every $p\in [1,+\infty)$ and $\lambda {\mathscr S}_{\lambda}g-{\rm Le}^{-1}\Delta {\mathscr S}_{\lambda}g-D_x{\mathscr S}_{\lambda}g=g$ in $S_0^+$. Finally, estimate \eqref{braccio-3} follows at once from \eqref{nessuno}.

{\em Step 2.} To complete the proof, here we check property (iv), which demands some additional effort.
We begin by checking that the function $\zeta={\mathscr S}_{\lambda}g(0,\cdot)$ belongs to $C^{2+\alpha}_b(\R;\C)$. For this purpose,
we set $\zeta_n=\ell^{-1}\sum_{k=-n}^n(\mathcal G_kg)(0)e_k$ for $n\in\N$. Clearly, each function $\zeta_n$ is smooth and
\begin{align*}
\zeta_n'=&
\frac{2\pi i}{\ell^2}\sum_{|k|<k_0}\frac{k}{Z_k}e_k\int_0^{+\infty}e^{\frac{{\rm Le}}{2}s}e^{-\frac{1}{2}Z_ks}\hat g_k(s)ds+\frac{2\pi i}{\ell^2}\sum_{k_0\le |k|\le n}\!\!\bigg  (\frac{k}{Z_k}-\frac{\ell}{\pi}\bigg )e_k\int_0^{+\infty}e^{\frac{{\rm Le}}{2}s}e^{-\frac{1}{2}Z_ks}\hat g_k(s)ds\\
&+\frac{2i}{\ell}\!\sum_{|k|=k_0}^n\!\!e_k\!\int_0^{+\infty}\!e^{\frac{{\rm Le}}{2}s}\Big (e^{-\frac{1}{2}Z_ks}\!-\!e^{-\frac{\pi|k|}{2\ell }s}\Big )\hat g_k(s)ds\!
+\!\frac{2i}{\ell}\!\sum_{|k|=k_0}^n\!\!e_k\int_0^{+\infty}\!e^{\frac{\rm Le}{2}s-\frac{\pi|k|}{2\ell}s}\hat g_k(s)ds \\
=: & \mathcal I_0+\sum_{h=1}^3{\mathcal I}_{h,n},
\end{align*}
where $k_0\in\N$ is chosen so that $\pi k_0>\ell {\rm Le}$. As it is easily seen,
\begin{equation}
\left |\frac{k}{Z_k}-\frac{\ell}{\pi}\right |\le \frac{c_{\lambda}}{k^2+1},\qquad\;\,k\in\Z,
\label{grignani}
\end{equation}
so that ${\mathcal I}_{1,n}$ converges uniformly in $\R^2$ as $n\to +\infty$.
On the other hand,
\begin{align*}
\Big |e^{-\frac{1}{2}Z_ks}-e^{-\frac{\pi|k|}{2\ell}s}\Big |
= \bigg |\int_0^1\frac{d}{dr}e^{-\frac{1}{2}Z_k(r)s}dr\bigg |
\le s\int_0^1
\bigg |\frac{{\rm Le}^2+4\lambda{\rm Le}}{4Z_k(r)}\bigg |e^{-\frac{1}{2}{\rm Re}(Z_k(r))s}dr,
\end{align*}
where we have set $Z_k(r)=\left (({\rm Le}^2+4\lambda{\rm Le})r+\frac{k^2\pi^2}{\ell^2}\right )^{\frac{1}{2}}$. Note that
${\rm Re}(Z_k(r))\ge c_{\lambda}|k|$ for $|k|\ge k_0$. For such values of $k$ and for  ${\rm Re}\lambda>0$ (which implies that
${\rm Re}(Z_k(r))>{\rm Le}$ for every $r\in [0,1]$ and $|k|\ge k_0$) we can estimate
\begin{align}
\bigg |\int_0^{+\infty}e^{\frac{{\rm Le}}{2}s}\Big (e^{-\frac{1}{2}Z_ks}-e^{-\frac{\pi|k|}{2\ell }s}\Big )\hat g_k(s)ds\bigg |
\le &\frac{|{\rm Le}^2+4\lambda{\rm Le}|}{4|k|}\|g\|_{\infty}\int_0^1r^{-\frac{1}{2}}dr\int_0^{+\infty}se^{\frac{\rm Le}{2}s-\frac{1}{2}({\rm Re}(Z_k(r))s}ds\notag\\
\le &\frac{c_\lambda}{k^3}\|g\|_{\infty},
\label{tokyo}
\end{align}
so that the sequence $({\mathcal I}_{2,n})$ converges uniformly in $\R^2$.
Next, we observe that
\begin{align*}
{\mathcal I}_{3,n}(y)=\frac{2i}{\ell}\int_0^{+\infty}ds\int_{-\frac{\ell}{2}}^{\frac{\ell}{2}}K_n(s,\eta)g_n(s,y-\eta)d\eta,\qquad\;\,y\in\R,\;\,n\in\N,
\end{align*}
where $K_n(x,y)=H_n(x,y)+H_n(x,-y)$ and
\begin{eqnarray*}
H_n(x,y)=e^{\frac{\rm Le}{2}x}\frac{e^{-\frac{\pi}{2\ell}(x-4iy)(k_0-1)}-e^{-\frac{\pi}{2\ell}(x-4iy)n}}{e^{\frac{\pi}{2\ell}(x-4iy)}-1},\qquad\;\,g_n(x,y)=\ell^{-1}\sum_{k=-n}^n\hat g(x,k)e^{\frac{2k\pi i}{\ell}y}
\end{eqnarray*}
for $x\ge 0$, $y\in\R^2$ and $n\in\N$.
We set
\begin{eqnarray*}
K(x,y)=e^{\left (\frac{\rm Le}{2}-\frac{\pi(k_0-1)}{2\ell}\right )x}\bigg (\frac{e^{-\frac{2\pi (k_0-1) i}{\ell}y}}{
e^{\frac{\pi}{2\ell}(x-4iy)}-1}+\frac{e^{\frac{2\pi(k_0-1) i}{\ell}y}}{
e^{\frac{\pi}{2\ell}(x+4iy)}-1}\bigg ),\qquad\;\,x\ge 0,\;\,y\in\R,
\end{eqnarray*}
and prove that ${\mathcal I}_{3,n}$ converges pointwise in $\R$ to the function ${\mathcal I}_3$, defined by\footnote{${\mathcal I}_3$ belongs to $C_b(\R;\C)$. Indeed, $H\in L^1(S^{+}_0;\C)$ as it follows observing that
$|e^{\frac{\pi}{2\ell}(x\pm 4iy)}-1|\ge e^{\frac{\pi}{2\ell}x}-1$, which implies the inequality
$|K(x,y)|\le e^{\left (\frac{\rm Le}{2}-\frac{\pi(k_0-1)}{2\ell}\right )x}\min \{2(e^{\frac{\pi}{2\ell}x}-1)^{-1},c(x^2+y^2)^{-1/2}\}$ for every $(x,y)\in S_0^{+}$. This shows that ${\mathcal I}_3$ is bounded in $\R^2$. Moreover, ${\mathcal I}_3$
is the uniform limit as $\varepsilon\to 0^+$ of the function $\int_{\varepsilon}^{+\infty}ds\int_{-\ell/2}^{\ell/2}K(s,\cdot-\eta)g(s,\eta)d\eta$, which is clearly continuous in $\R$ thanks to the above estimate for $K$. Hence, ${\mathcal I}_3$ is itself continuous in $\R$.}
\begin{eqnarray*}
{\mathcal I}_3(y)=\frac{i\ell}{\pi}\int_0^{+\infty} ds\int_{-\frac{\ell}{2}}^{\frac{\ell}{2}}K(s,\eta)g^{\sharp}(s,y-\eta)d\eta,\qquad\;\,y\in\R.
\end{eqnarray*}
For this purpose, we split
\begin{align*}
{\mathcal J}_{3,n}-{\mathcal J}_3=&\int_0^{+\infty} ds\int_{-\frac{\ell}{2}}^{\frac{\ell}{2}}(g^{\sharp}_n(s,\cdot-\eta)-g^{\sharp}(s,\cdot-\eta))K_n(s,\eta)d\eta\\
&+\int_0^{+\infty} ds\int_{-\frac{\ell}{2}}^{\frac{\ell}{2}}g^{\sharp}(s,\cdot-\eta)(K_n(s,\eta)-K(s,\eta))d\eta
=:{\mathcal A}_{1,n}+{\mathcal A}_{2,n}
\end{align*}
for every $n\in\N$ and observe that
\begin{align*}
\|{\mathcal A}_{1,n}\|_{\infty}\le \int_0^{+\infty}\|g_n(s,\cdot)-g(s,\cdot)\|_{L^2((-\ell/2,\ell/2))}\|K_n(s,\cdot)\|_{L^2((-\ell/2,\ell/2))}ds,\qquad\;\,n\in\N.
\end{align*}
Since $(i)$ $\|g_n(0,\cdot)-g(0,\cdot)\|_{L^2((-\ell/2,\ell/2);\C)}$ vanishes as $n\to +\infty$,
$(ii)~\|g_n(x,\cdot)-g(x,\cdot)\|_{L^2((-\ell/2,\ell/2);\C)}\le 2\|g(x,\cdot)\|_{L^2((-\ell/2,\ell/2);\C)}\le 2\ell\|g\|_{\infty}$, for $x\ge 0$ and $n\in\N$, and
$(iii)~|K_n|\le 2|K|$ in $\R_+\times\R$ for every $n\in\N$, the dominated convergence theorem shows that
${\mathcal A}_{1,n}$ converges to zero pointwise in $\R$ as $n$ tends to $+\infty$. Moreover, $\|{\mathcal A}_{1,n}\|_{\infty}\le c\|g\|_{\infty}$.
That theorem also shows that ${\mathcal A}_{2,n}$ converges to zero pointwise in $\R$ as $n$ tends to $+\infty$; moreover,
$\|{\mathcal A}_{2,n}\|_{\infty}\le c\|g\|_{\infty}$ for each $n\in\N$. Now, writing
\begin{align*}
\zeta_n(y)=\zeta_n(0)+\int_0^y\bigg(\mathcal I_0(r)+\sum_{h=1}^3{\mathcal I}_{h,n}(r)\bigg)dr,\qquad\;\,y\in\R,\;\,n\in\N,
\end{align*}
and letting $n$ tend to $+\infty$, again by dominated convergence we conclude that
\begin{align*}
\zeta'=&\frac{2\pi i}{\ell^2}\sum_{|k|<k_0}\frac{k}{Z_k}e_k\int_0^{+\infty}e^{\frac{{\rm Le}}{2}s}e^{-\frac{1}{2}Z_ks}\hat g_k(s)ds
+\frac{2\pi i}{\ell^2}\sum_{|k|\ge k_0}\left (\frac{k}{Z_k}-\frac{\ell}{\pi}\right )e_k\int_0^{+\infty}e^{\frac{{\rm Le}}{2}s}e^{-\frac{1}{2}Z_ks}\hat g_k(s)ds\\
&+\frac{2i}{\ell}\sum_{|k|\ge k_0}e_k\int_0^{+\infty}e^{\frac{{\rm Le}}{2}s}\Big (e^{-\frac{1}{2}Z_ks}-e^{-\frac{\pi k}{2\ell }s}\Big )\hat g_k(s)ds+\frac{i\ell}{\pi}\int_0^{+\infty}ds\int_{-\frac{\ell}{2}}^{\frac{\ell}{2}}K(s,\eta)g^{\sharp}(s,\cdot-\eta)d\eta.
\end{align*}
Denote by $\phi_1,\ldots,\phi_4$ the four terms in the right-hand side of the previous formula.
Clearly, $\phi_1$ belongs to $C^{\infty}_b(\R;\C)$. In particular, $\|\phi_1\|_{C^{1+\alpha}_b(\R;\C)}\le c\|g\|_{\infty}$. As far as $\phi_2$ and $\phi_3$ are concerned, using
\eqref{grignani}, \eqref{tokyo}, the same arguments here above and in the first part of the proof of Lemma \ref{finito}, it can be easily shown that such functions belong to $C^{1+\alpha}_b(\R;\C)$ and $\|\phi_2\|_{C^{1+\alpha}_b(\R;\C)}+\|\phi_3\|_{C^{1+\alpha}_b(\R;\C)}\le c\|g\|_{\infty}$.

The function $\phi_4$ is the limit in $C_b(\R;\C)$ of the sequence $(\phi_{4,n})$ defined by
\begin{eqnarray*}
\phi_{4,n}=\frac{i\ell}{\pi}\int_{\frac{1}{n}}^{+\infty}ds\int_{-\frac{\ell}{2}}^{\frac{\ell}{2}}K(s,\eta)g^{\sharp}(s,\cdot-\eta)d\eta,\qquad\;\,n\in\N.
\end{eqnarray*}
Clearly, each function $\phi_{4,n}$ is continuously differentiable in $\R$ and
\begin{eqnarray*}
\phi_{4,n}'=\frac{i\ell}{\pi}\int_{\frac{1}{n}}^{+\infty}ds\int_{-\frac{\ell}{2}}^{\frac{\ell}{2}}K_y(s,\eta)g^{\sharp}(s,\cdot-\eta)d\eta,
\end{eqnarray*}
where
\begin{align}
K_y(x,y)=&\frac{2\pi i}{\ell}(k_0-1)e^{\left (\frac{\rm Le}{2}-\frac{\pi(k_0-1)}{2\ell}\right )x}\bigg (
\frac{e^{\frac{2\pi (k_0-1)i}{\ell}y}}{e^{\frac{\pi}{2\ell}(x+4iy)}-1}-\frac{e^{-\frac{2\pi (k_0-1)i}{\ell}y}}{e^{\frac{\pi}{2\ell}(x-4iy)}-1}\bigg )\notag\\
&+\frac{2\pi i}{\ell}e^{\left (\frac{\rm Le}{2}-\frac{\pi(k_0-1)}{2\ell}\right )x}\bigg (
\frac{e^{\frac{\pi}{2\ell}(x-4ik_0y)}}{(e^{\frac{\pi}{2\ell}(x-4iy)}-1)^2}-
\frac{e^{\frac{\pi}{2\ell}(x+4ik_0y)}}{(e^{\frac{\pi}{2\ell}(x+4iy)}-1)^2}\bigg )=:L_1(x,y)+L_2(x,y)
\label{presepe}
\end{align}
for $(x,y)\in \overline{H_0^+}$. Since $K$ is $\ell$-periodic with respect to $y$, it follows that $\int_{-\ell/2}^{\ell/2}K_y(s,\eta)d\eta=0$. Hence, we can write
\begin{eqnarray*}
\phi_{4,n}'(y)=\frac{i\ell}{\pi}\int_{\frac{1}{n}}^{+\infty}ds\int_{-\frac{\ell}{2}}^{\frac{\ell}{2}}K_y(s,\eta)(g^{\sharp}(s,y-\eta)-g(s,y))d\eta,\qquad\;\,
y\in\R,\;\,n\in\N.
\end{eqnarray*}
By assumptions, $g\in C^{\alpha}_b(S_0^+;\C)$ and this allows us to estimate
\begin{eqnarray*}
|K_y(s,\eta)(g^{\sharp}(s,y-\eta)-g^{\sharp}(s,y))|\le c\min\{(s^2+\eta^2)^{\frac{\alpha}{2}-1},(e^{\frac{\pi}{2\ell}|s|}-1)^{-1}\}\|g\|_{C^{\alpha}_b([0,+\infty)\times\R;\C)}
\end{eqnarray*}
for every $(s,\eta)\in\R\times\R_+$ and $y\ge 0$.
Thus, we can let $n$ tend to $+\infty$ in \eqref{presepe} and conclude that $\phi_{4,n}'$ converges uniformly in $\R^2$.
As a byproduct, $\phi_4$ is continuously differentiable in $\R$,
\begin{eqnarray*}
\phi_4'=\int_0^{+\infty}ds\int_{-\frac{\ell}{2}}^{\frac{\ell}{2}}K_y(s,\eta)(g^{\sharp}(s,\cdot-\eta)-g(s,\cdot))d\eta
\end{eqnarray*}
and $\|\phi_4'\|_{\infty}\le c\|g\|_{C^{\alpha}_b(S^+_0;\C)}$.

To prove that $\phi_4'$ belongs to $C^{\alpha}_b(\R;\C)$ we split
\begin{align}
\phi_4'=&\int_0^{+\infty}ds\int_{-\frac{\ell}{2}}^{\frac{\ell}{2}}(L_1(s,\eta)+L_2(s,\eta)\chi_{(1,+\infty)}(s))(g^{\sharp}(s,\cdot-\eta)-g(s,\cdot))d\eta\notag\\
&+\int_0^1ds\int_{-\frac{\ell}{2}}^{\frac{\ell}{2}}L_2(s,\eta)(g^{\sharp}(s,\cdot-\eta)-g(s,\cdot))d\eta.
\label{alba}
\end{align}
Since the function $L_1+L_2\chi_{(1,+\infty)}\in L^1(S^+_0;\C)$, the first term in the right-hand side of
\eqref{alba} belongs to $C^{\alpha}_b(\R;\C)$ and its $C^{\alpha}_b(\R;\C)$-norm can be estimated from above by $c\|g\|_{C^{\alpha}_b(S_0^+;\C)}$.
To estimate the other term, which we denote by $\Psi$, we observe that
\begin{eqnarray*}
\frac{e^{\frac{\pi}{2\ell}(x-4ik_0y)}}{(e^{\frac{\pi}{2\ell}(x-4iy)}-1)^2}=\frac{\ell^2}{\pi^2}\frac{4x^2+32ixy-64y^2}{(x^2+16y^2)^2}
+\psi(x,y),\qquad (x,y)\in (0,1)\times (-\ell/2,\ell/2),
\end{eqnarray*}
for some function $\psi\in L^1((0,1)\times (-\ell/2,\ell/2);\C)$. Thus,
\begin{align*}
\Psi(y)=&-\frac{128i\ell^2}{\pi^2}\int_0^1ds\int_{-\frac{\ell}{2}}^{\frac{\ell}{2}}\frac{s(y-\eta)}{(s^2+16(y-\eta)^2)^2}(g^{\sharp}(s,\eta)-g(s,y))d\eta\\
&-2\int_0^1ds\int_{-\frac{\ell}{2}}^{\frac{\ell}{2}}(\psi(s,\eta)-\psi(s,-\eta))(g^{\sharp}(s,y-\eta)-g(s,y))d\eta=:\Psi_1(y)+\Psi_2(y)
\end{align*}
for every $y\in\R$.
The function $\Psi_2$ is clearly $\alpha$-H\"older continuous in $[-\ell/2,\ell/2]$ since $\psi\in L^1((0,1)\times (-\ell/2,\ell/2);\C)$.
Moreover, $\|\Psi_2\|_{C^{\alpha}_b(\R;\C)}\le c\|g\|_{C^{\alpha}_b(\R;\C)}$.
As far as the function $\Psi_1$ is concerned, we approximate it with the family of functions
$\Psi_{1,h}$ defined by
\begin{eqnarray*}
\Psi_{1,h}(y)=-\frac{128i\ell^2}{\pi^2}\int_0^1ds\int_{-\frac{\ell}{2}}^{\frac{\ell}{2}}\frac{s(y-\eta)}{(s^2+16(y-\eta)^2+h^2)^2}(g(s,\eta)-g(s,\cdot))d\eta,\qquad\;\,h>0.
\end{eqnarray*}
Each of these functions is continuously differentiable in $\R$ with bounded derivative, so that, we can estimate
\begin{align}
|\Psi_1(y_2)-\Psi_1(y_1)|\le &|\Psi_1(y_2)-\Psi_{1,h}(y_2)|+|\Psi_{1,h}(y_2)-\Psi_{1,h}(y_1)|+|\Psi_{1,h}(y_1)-\Psi_1(y_1)|\notag\\
\le &2\|\Psi_1-\Psi_{1,h}\|_{\infty}+\|\Psi_{1,h}'\|_{\infty}|y_2-y_1|
\label{swear}
\end{align}
for $y_1,y_2\in [-\ell/2,\ell/2]$. Note that
\begin{align*}
\|\Psi_1-\Psi_{1,h}\|_{\infty}\le &\frac{256\ell^2}{\pi^2}\|g\|_{C^{\alpha}_b(S^+_0;\C)}\int_{\R^2}\frac{|s\eta|h^2}{(s^2+16\eta^2+h^2)^2(s^2+16\eta^2)}dsd\eta\\
\le & ch^2\|g\|_{C^{\alpha}_b(S^+_0;\C)}\int_0^{+\infty}\frac{\rho^{\alpha+1}}{(\rho^2+h^2)^2}d\rho= ch^{\alpha}\|g\|_{C^{\alpha}_b(S^+_0;\C)}
\end{align*}
and
\begin{align*}
\|\Psi'_{1,h}\|_{\infty}\le & \frac{128\ell^2}{\pi^2}\|g\|_{C^{\alpha}_b(S_0^+;\C)}\int_{\R^2}
\frac{(s^2+48\eta^2+h^2)|\eta|^{\alpha}}{(s^2+16\eta^2+h^2)^3}dsd\eta\\
\le &c\|g\|_{C^{\alpha}_b(S_0^+;\C)}\int_0^{+\infty}\frac{\rho^{2+\alpha}}{(\rho^2+h^2)^2}d\rho\le c\|g\|_{C^{\alpha}_b(S_0^+;\C)}h^{\alpha-1}.
\end{align*}
Replacing these inequalities into \eqref{swear} and taking $h=|y_2-y_1|$, we conclude that
\begin{eqnarray*}
|\Psi_1(y_2)-\Psi_1(y_1)|\le c\|g\|_{C^{\alpha}_b(S^+_0;\C)}|y_2-y_1|^{\alpha}.
\end{eqnarray*}
Therefore, $\phi_4'\in C^{\alpha}_b(\R;\C)$ and $\|\phi_4'\|_{C^{\alpha}_b(\R;\C)}\le c\|g\|_{C^{\alpha}_b(S^+_0;\C)}$.
Putting everything together it follows that $\zeta\in C^{2+\alpha}_b(\R;\C)$ and $\|\zeta\|_{C^{2+\alpha}_b(\R;\C)}\le c\|g\|_{C^{\alpha}_b(S^+_0;\C)}$.

Finally, we consider the function ${\mathscr S}_{\lambda}g-\zeta=:v\in C_b([0,+\infty)\times\R;\C)\cap\bigcap_{p<+\infty} W^{2,p}_{\rm loc}((0,+\infty)\times\R;\C)$. Since ${\rm Le}^{-1}\Delta v+v_x\in C^{\alpha}_b((0,+\infty)\times\R;\C)$ and by construction $v(0,\cdot)=0$, by classical results (see e.g., \cite{krylov})
$v\in C^{2+\alpha}_b((0,+\infty)\times\R)$ and
\begin{align*}
\|v\|_{C^{2+\alpha}_b(H^+_0;\C)}\le &c(\|v\|_{\infty}+
\|{\rm Le}^{-1}\Delta v+v_x\|_{C_b^{\alpha}(H^+_0)}+\|\zeta\|_{C^{2+\alpha}_b(\R;\C)})\\
\le & c(\|v\|_{\infty}+\|\lambda {\mathscr S}_{\lambda}g-g\|_{C^{\alpha}_b(H^+_0;\C)}
+\|\zeta\|_{C^{2+\alpha}_b(\R;\C)})\le c\|g\|_{C^{\alpha}_b(H^+_0;\C)}.
\end{align*}
Formula \eqref{cotto-11} follows as once.
\end{proof}

\begin{lemm}
\label{finiremo}
For each $\lambda\in \C$ such that ${\rm Re}\lambda>-({\rm Im}\lambda)^2$, $f\in C_b(\overline{S};\C)$ and $g\in C_b(\overline{S_0^+};\C)$, we denote by ${\mathscr T}_{\lambda}^{\pm}f:\overline{H_R^{\mp}}\to\C$ and ${\mathscr U}_{\lambda}g:\overline{H_0^+}\to\C$, respectively, the functions defined by
\begin{align*}
&({\mathscr T}_{\lambda}^{\pm}f)(x,y)= e^{-\frac{{\rm Le}}{2}(x-R)}\sum_{k\in\Z}e^{\pm \frac{1}{2}Y_k(x-R)}\mathcal F_{k,1}f(R)e_k(y),\qquad\;\,(x,y)\in \overline{H_R^{\mp}},\\[1mm]
&({\mathscr U}_{\lambda}g)(x,y)=e^{-\frac{\rm Le}{2}x}\sum_{k\in\Z}e^{-\frac{Y_k}{2}x}({\mathcal G}_kg)(0)e_k(y),\qquad\;\,(x,y)\in \overline{H_0^+}.
\end{align*}
Then, the following properties are satisfied.
\begin{enumerate}[\rm (i)]
\item
${\mathscr T}_{\lambda}^{\pm}f$ belongs to $C^1_b(\overline{H_R^{\mp}};\C)\cap W^{2,p}_{\rm loc}(H_R^{\mp};\C)$, for each $p<+\infty$,
solves the equation $\lambda u-{\rm Le}^{-1}\Delta u-u_x=0$ in $H_R^{\mp}$ and, for each $M>0$, there exists a constant $c_M>0$ such that
\begin{align}
\label{pasta}
|\lambda|\|{\mathscr T}_{\lambda}^{\pm}f\|_{\infty}
+\sqrt{|\lambda|}\|\nabla{\mathscr T}_{\lambda}^{\pm}f\|_{\infty}
\leq c_M\|f\|_{\infty},\qquad\;\,|\lambda|\ge M;
\end{align}
\item
$\displaystyle\lim_{x\to +\infty}({\mathscr T}_{\lambda}^{-}f)(x,y)=\displaystyle\lim_{x\to -\infty}({\mathscr T}_{\lambda}^+f)(x,y)=0$ for every $y\in\R$ and $f\in C_b(\overline{S};\C)$;
\item
if $f\in C^{\alpha}_b(S;\C)$ and $f(\cdot,-\ell/2)=f(\cdot,-\ell/2)$, then the function ${\mathscr T}_{\lambda}^{\pm}f$ belongs to $C_b^{2+\alpha}(\overline{H^{\mp}_R};\C)$ and
$\|\mathscr T^{\pm}_{\lambda}f\|_{C_b^{2+\alpha}(H_R^{\mp})}\leq c\|f\|_{C^{\alpha}_b(S;\C)}$ for each ${\rm Re}\lambda>-({\rm Im}\lambda)^2$;
\item
${\mathscr U}_{\lambda}g$ belongs to $C^1_b(\overline{H_0^+};\C)\cap W^{2,p}_{\rm loc}(H_0^+;\C)$, for each $p<+\infty$, solves the equation $\lambda u-{\rm Le}^{-1}\Delta u-D_xu=0$ and, for every $M>0$ there exists a positive constant $c_M'$ such that
\begin{align*}
|\lambda|\|{\mathscr U}_{\lambda}g\|_{\infty}
+\sqrt{|\lambda|}\|\nabla{\mathscr U}_{\lambda}g\|_{\infty}\leq c_M'\|g\|_{\infty},\qquad\;\,|\lambda|\ge M;
\end{align*}
\item
$\displaystyle\lim_{x\to +\infty}({\mathscr U}_{\lambda}^{-}g)(x,y)=0$ for each $y\in\R$ and $g\in C_b(\overline{S_0^+};\C)$;
\item
if $g\in C^{\alpha}_b(S_0^+;\C)$ is such that $g(\cdot,-\ell/2)=g(\cdot,-\ell/2)$, then the function ${\mathscr U}_{\lambda}g$ belongs to $C_b^{2+\alpha}(H^+_0;\C)$ and
$\|\mathscr U_\lambda g\|_{C_b^{2+\alpha}(H^+_0;\C)}\leq c\|g\|_{C^{\alpha}_b(S_0^+;\C)}$
for each $\lambda\in\C$ such that ${\rm Re}\lambda>-({\rm Im}\lambda)^2$.
\end{enumerate}
\end{lemm}

\begin{proof}
(i) The arguments as in the proof of Lemma \ref{finito} show that the function ${\mathscr T}_{\lambda}^{\pm}f$ is continuous in $\overline{H_R^{\mp}}$ and smooth in its interior, where it solves the equation $\lambda u-{\rm Le}^{-1}\Delta u-u_x=0$. Further, the function $v^{\pm}={\mathscr T}_{\lambda}^{\pm}f-{\mathscr R}_{\lambda,1}f$ is bounded, vanishes on $\{R\}\times\R$ and
$\lambda v^{\pm}-{\rm Le}^{-1}\Delta v^{\pm}-v^{\pm}_x=h$ in $H_R^{\mp}$,
where $L^{\infty}(H_R^{\mp};\C)\ni h=-{\rm Le}^{-1}f^{\sharp}+({\rm Le}^{-1}-1)(\lambda{\mathscr R}_{\lambda,1} f-D_x{\mathscr R}_{\lambda,1}f)$.

By classical results, the realization of the operator ${\rm Le}^{-1}\Delta+D_x$ in $L^{\infty}(H_R^{\mp};\C)$ with homogeneous Dirichlet boundary conditions generates an analytic semigroup with domain $\{u\in C_b(\overline{H_R^{\mp}};\C)\cap\bigcap_{p<+\infty}W^{2,p}_{\rm loc}(H_R^{\mp};\C): {\rm Le}^{-1}\Delta u+D_x u\in L^{\infty}(H_R^{\mp};\C)\}$. In particular, for every $\lambda\in\Sigma_0$ it holds that $|\lambda|\|v^{\pm}\|_{\infty}+\sqrt{|\lambda|}\|\nabla v^{\pm}\|_{\infty}\le c\|h\|_{\infty}$. From the definition of $v^{\pm}$ and taking \eqref{braccio} into account, estimate \eqref{pasta} follows immediately.

(ii) The proof of this property is immediate since the series defining ${\mathscr T}_{\lambda}^+f$ (resp. ${\mathscr T}_{\lambda}^-f$) converges uniformly in $\overline{H_R^-}$ (resp. in $\overline{H_R^+}$) and each of its terms vanishes as $x\to -\infty$ (resp. $x\to +\infty$), uniformly with respect to $y\in\R$.

(iii) Fix $\lambda\in\C$ with ${\rm Re}\lambda>-({\rm Im}\lambda)^2$.
Since $f\in C^\alpha_b(S)$ and $f(\cdot,-\ell/2)=f(\cdot,\ell/2)$, thanks to Lemma \ref{finito}$(iii)$ we can infer that the function $h\in C_b^{\alpha}(\R^2;\C)$. Hence, ${\rm Le}^{-1}\Delta v^{\pm}+v^{\pm}_x\in C_b^\alpha(H_R^{\mp};\C)$ and by classical results it follows that $v\in C_b^{2+\alpha}(H_R^{\mp};\C)$ and
\begin{align*}
\|v^{\pm}\|_{C_b^{2+\alpha}(H_R^{\mp};\C)}\leq c(\|v^{\pm}\|_{\infty}+\|{\rm Le}^{-1}\Delta v^{\pm}+v^{\pm}_x\|_{C_b^\alpha(H_R^{\mp};\C)}).
\end{align*}
From the definition of $v^{\pm}$ and the above estimate, the assertion follows at once.

(iv)-(vi) The proof of these three properties follows applying the procedure of the first part of the proof, with ${\mathscr R}_{\lambda,1}f $ being replaced by the function
${\mathscr S}_{\lambda}g$. The details are left to the reader.
\end{proof}

\subsection{Analytic semigroups and interpolation spaces}
To state the main result of this subsection, for each $k\in\N\cup\{0\}$ we introduce the functions (the so-called dispersion relations)
\begin{align*}
&{\mathcal D}_k(\lambda)=({\rm Le}-Y_k)\bigg [\exp\bigg (\frac{R}{2}({\rm Le}-1-X_k-Y_k)\bigg )-1+\theta_iRX_k\bigg ],\\[1mm]
&\widetilde {\mathcal D}_k(\lambda)=({\rm Le}-Y_k)\bigg [\exp\bigg (\frac{R}{2}({\rm Le}-1-X_k-Y_k)\bigg )-\exp\bigg (\frac{R}{2}({\rm Le}-1+X_k-Y_k)\bigg )+\theta_i R X_k\bigg ],
\end{align*}
where $X_k=X_k(\lambda)=\sqrt{1+4\lambda+4\lambda_k}$, $Y_k=Y_k(\lambda)=\sqrt{{\rm Le}^2+4\lambda {\rm Le}+4\lambda_k}$, $\lambda_k=4\pi^2k^2\ell^{-2}$, and the sets
\begin{align}
&\Omega_k=\{\lambda\in\C: {\rm Re}\lambda\ge -({\rm Im}\lambda)^2-\lambda_k\,{\rm and}\,{\mathcal D}_k(\lambda)=0\},
\label{omega-k}
\\
&\Omega_k'=\{\lambda\in\C: -{\rm Le}^{-1}(({\rm Im}\lambda)^2+\lambda_k)<{\rm Re}\lambda<-({\rm Im}\lambda)^2-\lambda_k,\,{\rm and}\,\widetilde {\mathcal D}_k(\lambda)=0\}.\notag
\end{align}

\begin{thm}\label{banca}
The realization $L$ of the operator ${\mathscr L}$ in $\boldsymbol{\mathcal X}$, with domain
\begin{align*}
D(L)\!=\!\bigg\{&{\bf u}\in \boldsymbol{\mathcal X}: u_j(\cdot,-\ell/2)=u_j(\cdot,\ell/2), j=1,2,\, u_1^{\sharp}\in C^1_b(\overline{H_0^-};\C)\cap \bigcap_{p<+\infty}W^{2,p}_{\rm loc}(H_0^{-};\C)\\
&~u_1^{\sharp},u_2^{\sharp}\!\in\! C^1([0,R]\times\R;\C)\!\cap\! C^1_b(\overline{H_R^+};\C)\!\cap\!\!\bigcap_{p<+\infty}\!\!W^{2,p}_{\rm loc}((\R_+\!\setminus\!\{R\})\!\times\!\R;\C),\;{\mathscr L}{\bf u}\!\in\!\boldsymbol{\mathcal X},\,{\mathscr B}\uu={\bf 0}\bigg\},
\end{align*}
$($see \eqref{operatore-L} and \eqref{boundary-B}$)$ generates an analytic semigroup in $\boldsymbol{\mathcal X}$. Moreover,
\begin{enumerate}[\rm (i)]
\item
the spectrum $\sigma(L)$ of $L$ is the set $\{\lambda\in\C:{\rm Re}\lambda\le -\Le^{-1}({\rm Im}\lambda)^2\}\cup\bigcup_{k\in\N\cup\{0\}}(\Omega_k\cup\Omega_k')$;
\item
there exist two positive constants $M$ and $c$ such that
$\sqrt{|\lambda|}\|\nabla R(\lambda,L)\f\|_{\infty}\le c\|\f\|_{\infty}$
for $\lambda\in\C$ with ${\rm Re}\lambda\ge M$ and $\f\in\boldsymbol{\mathcal X}$;
\item
if $\f\in \boldsymbol{\mathcal X}_{\alpha}$ for some $\alpha\in (0,1)$, then for each $\lambda\in\rho(L)$ the function $R(\lambda,L){\bf f}$ belongs to
$\boldsymbol{\mathcal X}_{2+\alpha}$ and
$\|R(\lambda,L){\bf f}\|_{2+\alpha}\leq c_{\lambda}\|{\bf f}\|_{\alpha}$.
\end{enumerate}
\end{thm}

\begin{proof}
Since it is rather long, we split the proof into four steps. In Steps 1 and 2, we characterize $\sigma(L)$, whereas in Step 3 we prove that $L$
generates an analytic semigroup in $\boldsymbol{\mathcal X}$ as well as the estimate for the spatial gradient for the resolvent operator. Finally, in Step 4, we
prove property (iii).

{\em Step 1.} Fix ${\bf f}\in\boldsymbol{\mathcal X}$ and $\lambda\in\C$ such that ${\rm Re}\lambda>-({\rm Im}\lambda)^2$ and $\lambda\not\in\Omega_k$ for each $k\in\N\cup\{0\}$,
and assume that the equation $\lambda\uu-{\mathscr L}\uu=\f$ admits a solution $\uu=(u_1,u_2)$ in $D(L)$.
The arguments in the proof of Lemma \ref{finito}(i) show that for every $k\in\Z$ the functions $\hat u_{1,k}$ and $\hat u_{2,k}$ (see Subsection \ref{sub-notation}), solve, respectively, the differential equations $\lambda \hat u_k-\hat u_k'-\hat u''_k+\lambda_k\hat u_k=\hat f_{1,k}$ in $\R\setminus\{0,R\}$
and $\lambda \hat v_k-\hat v'_k-{\rm Le}^{-1}\hat v_k''+{\rm Le}^{-1}\lambda_k\hat v_k=\hat f_{2,k}$ in $\R_+\setminus\{R\}$. Moreover, they belong to $C^1_b(\R;\C)$ and
$C^1_b([0,+\infty);\C)$, respectively. Thus,
\begin{align*}
\hat u_{1,k}(x)=({\mathcal F}_{k,1}f_1)(x)+c_{1,k}e^{\nu^{+}_kx}\chi_{(-\infty,0]}(x)+
(c_{2,k}e^{\nu^{-}_kx}+c_{3,k}e^{\nu^{+}_kx})\chi_{(0,R)}(x)+
c_{4,k}e^{\nu^{-}_kx}\chi_{[R,+\infty)}(x)
\end{align*}
for every $x\in\R$, $k\in\Z$ and
\begin{align}
\hat u_{2,k}(x)=({\mathcal G}_kf_2)(x)+(d_{1,k}e^{\mu^{-}_kx}+d_{2,k}e^{\mu^{+}_kx})\chi_{[0,R)}(x)+d_{3,k}e^{\mu^{-}_kx}\chi_{[R,+\infty)}(x)
\label{u2k}
\end{align}
for every $x\ge 0$ and $k\in\Z$, where $\mu^{\pm}_k=-\frac{{\rm Le}}2\pm\frac{1}{2}Y_k$, $\nu^{\pm}_k = -\frac{1}{2}\pm\frac{1}{2}X_k$ and $c_{1,k},c_{2,k},c_{3,k},c_{4,k}$, $d_{1,k},d_{2,k},d_{3,k}$ ($k\in\Z$) are complex constants determined imposing the conditions ${\mathscr B}(\hat u_{1,k},\hat u_{2,k})={\bf 0}$ that the infinitely many functions $\hat u_{1,k}$ and $\hat u_{2,k}$ have to satisfy.
It turns out that the above constants are uniquely determined if and only if $\mathcal D_k(\lambda)\neq 0$, as we are assuming, and as a byproduct,
\begin{align}
&u_1={\mathscr R}_{\lambda,1} f_1+\frac{1}{\ell}\sum_{k\in\Z}p_{1,k}({\mathcal F}_{k,1}f_1)(R)e_k+\frac{2}{{\rm Le}\, \ell}\sum_{k\in\Z}p_{2,k}({\mathcal G}_kf_2)(0)e_k,\qquad
{\rm in}~\overline{H_0^-},
\label{piove}\\[2mm]
&u_1={\mathscr R}_{\lambda,1} f_1+\frac{1}{\ell}\sum_{k\in\Z}p_{3,k}({\mathcal F}_{k,1}f_1)(R)e_k
+\frac{2}{{\rm Le}\,\ell}\sum_{k\in\Z}p_{4,k}({\mathcal G}_kf_2)(0)e_k,\label{piove-1} \\[2mm]
&u_2={\mathscr S}_{\lambda}f_2-\frac{{\rm Le}}{2\theta_iR\ell}{\mathscr T}_{\lambda}^{+}f_1+\frac{1}{\ell}{\mathscr U}_{\lambda}^-f_2
+\frac{{\rm Le}}{2\ell}\sum_{k\in\Z}q_{1,k}({\mathcal F}_{k,1}f_1)(R)e_k+\frac{1}{\ell}\sum_{k\in\Z}q_{2,k}({\mathcal G}_kf_2)(0)e_k,\label{vento}
\end{align}
in $(0,R)\times\R$,
and
\begin{align}
&u_1={\mathscr R}_{\lambda,1} f_1+\frac{1}{\ell}\sum_{k\in\Z}p_{5,k}({\mathcal F}_{k,1}f_1)(R)e_k+\frac{2\theta_iR}{{\rm Le}\,\ell}\sum_{k\in\Z}p_{6,k}({\mathcal G}_kf_2)(0)e_k,\label{sole-0}\\[2mm]
&u_2={\mathscr S}_{\lambda}f_2+\frac{{\rm Le}}{2\theta_iR\ell}{\mathscr T}_{\lambda}^{-}f_1+\frac{{\rm Le}}{2\ell}\sum_{k\in\Z}q_{3,k}({\mathcal F}_{k,1}f_1)(R)e_k+\frac{1}{\ell}\sum_{k\in\Z}q_{4,k}({\mathcal G}_kf_2)(0)e_k, \label{sole}
\end{align}
in $\overline{H_R^+}$, where ${\mathscr R}_{\lambda,1}f$ and ${\mathscr S}_{\lambda}g$ have been introduced in Lemmata \ref{finito} and \ref{finire},
\begin{align*}
&p_{1,k}(x)=\frac{e^{-\nu_k^{+}R}-e^{-\mu_k^{+}R}}{W_k}e^{\nu_k^+x},\qquad\;\;\;\, p_{2,k}(x)=\frac{(\theta_iRX_k-1+e^{-X_kR})Y_k}{X_kW_k(Y_k-{\rm Le})}e^{\nu_k^{+}x}\\
&p_{3,k}(x)=\frac{e^{\nu_k^{+}(x-R)}-e^{\nu_k^{-}x-\mu_k^{+}R}}{W_k},\qquad\;\,p_{4,k}(x)=\frac{Y_k[(\theta_iRX_k-1)e^{\nu^{-}_kx}+e^{-X_kR}e^{\nu^{+}_kx}]}{X_kW_k(Y_k-{\rm Le})},\\
&p_{5,k}(x)=\frac{e^{-\nu_k^{-}R}-e^{-\mu_k^{+}R}}{W_k}e^{\nu^{-}_kx},\qquad\;\;\;\, p_{6,k}(x)=\frac{Y_k}{(Y_k-{\rm Le})W_k}e^{\nu^{-}_kx},\\
&q_{1,k}(x)=\frac{e^{(\nu_k^{-}-\mu_k^+)R}-1}{\theta_iRW_k}e^{\mu_k^{+}(x-R)}-\frac{X_ke^{-\mu_k^{+}R}}{Y_kW_k}[({\rm Le}+Y_k)e^{\mu_k^{-}x}-{\rm Le}\,e^{\mu_k^{+}x}],\\
&q_{2,k}(x)=\frac{2{\rm Le}(\theta_iRX_k-1)}{W_k(Y_k-{\rm Le})}e^{\mu^{-}_kx}-\frac{e^{(\nu_k^{-}-\mu_k^+)R}}{W_k}(e^{\mu^{-}_kx}+e^{\mu^{+}_kx}),\\
&q_{3,k}(x)=\frac{1-e^{(\nu_k^{-}-\mu_k^{+})R}}{\theta_iRW_k}e^{\mu^{-}_k(x-R)}-\frac{{\rm Le}(e^{\mu^{-}_kx-\mu_k^{+}R}-e^{\mu_k^{-}(x-R)})X_k}{Y_kW_k}
-\frac{X_ke^{\mu^{-}_kx-\mu_k^{+}R}}{W_k},\\
&q_{4,k}(x)=\frac{Y_k(\theta_iRX_k-1)e^{\mu^{-}_kx}-{\rm Le}\, e^{\mu^{-}_kx+(\nu_k^{-}-\mu_k^{+})R}+(Y_k+{\rm Le})e^{\nu_k^{-}R+\mu_k^{-}(x-R)}}{W_k(Y_k-{\rm Le})}
\end{align*}
and $W_k=W_k(\lambda)=\theta_iRX_k-1+e^{(\nu^{-}_k-\mu_k^{+})R}$.

In view of Lemmata \ref{finito} to \ref{finiremo}, to prove that the pair $\uu$ defined by \eqref{piove}-\eqref{sole} belongs to $D(L)$ and $\lambda {\bf u}-\mathscr L{ \bf u}={\bf f}$ we just need to consider the series in the above formulae, which we denote, respectively by ${\mathscr P}_{\lambda,2k+j}f_j$ ($j=1,2$, $k=0,1,2$), ${\mathscr Q}_{\lambda,2h+j}f_j$ ($j=1,2$, $h=0,1$).
To begin with, we observe that, by $(i)$ in the proof of Lemma \ref{finito}, we already know that ${\rm Re}(X_k)+{\rm Re}(Y_k)\ge c_{\lambda}|k|$ and $|X_k|+|Y_k|\le c_{\lambda}(|k|+1)$ for each $k\in\Z$. As a byproduct, taking also \eqref{spugna} into account, we can infer that
$|({\mathcal F}_{k,1}f_1)(R)|\le c_{\lambda}(1+k^2)^{-1}\|f_1\|_{\infty}$ and $|({\mathcal G}_kf_2)(0)|\le c_{\lambda}(1+k^2)^{-1}\|f_2\|_{\infty}$ for each $k\in\Z$.
Moreover, we can also estimate
\begin{align}
|W_k|&\ge \theta_iR|X_k|-1-e^{{\rm Re}(\nu_k^{-}-\mu_k^{+})R}\ge \theta_iR|X_k|-1-e^{\frac{{\rm Le}-1}{2}R}\ge c_{\lambda}|k|,\qquad\;\,k\in\Z\setminus\{0\}.
\label{aaa}
\end{align}
Putting everything together, we conclude that $\|p_{1,k}\|_{C^h_b((-\infty,0];\C)}\le c_1e^{-c_2k}$ for each $h\in\N$ and
\begin{align*}
&\|p_{2,k}\|_{C_b^j((-\infty,0];\C)}+\sum_{i=1}^2[\|(p_{i+2,k},q_{i,k})\|_{C^j([0,R];\C^2)}+\|(p_{4+i,k},q_{i+2,k})\|_{C_b^j([R,+\infty);\C^2)}]
\le c_{\lambda}|k|^{j-1}
\end{align*}
for each $k\in\Z\setminus\{0\}$. Using these estimates, it is easy to check that
${\mathscr P}_{\lambda,1}f_1\in C^{\beta}_b(\overline{H^{-}_0};\C)$ for $\beta>0$ and ${\mathscr P}_{\lambda,2}f_2\in C^{\infty}(H^{-}_0;\C)\cap C^1_b(\overline{H^{-}_0};\C)$. Moreover, they solve the equation $\lambda w-\Delta_xw-D_xw=0$ in $H^{-}_0$. Since the series which define ${\mathscr P}_{\lambda,1}f_1$ and ${\mathscr P}_{\lambda,2}f_2$ converge uniformly in $H^{-}_0$ and each term vanishes as $x\to -\infty$,
 uniformly with respect to $y\in\R$, we immediately infer that $\lim_{x\to -\infty}({\mathscr P}_{\lambda,1}f_1)(x,y)=\lim_{x\to -\infty}({\mathscr P}_{\lambda,1}f_2)(x,y)=0$ for each $y\in\R$.
 On the other hand, the functions ${\mathscr P}_{\lambda,3}f_1$, ${\mathscr P}_{\lambda,4}f_2$
and ${\mathscr Q}_{\lambda,1}f_1$, ${\mathscr Q}_{\lambda,2}f_2$ belong to $C^{\infty}((0,R)\times\R;\C)\cap C^1_b([0,R]\times\R;\C)$ and solve, in $(0,R)\times\R$, the equations $\lambda w_1-\Delta_xw_1-D_xw_1=0$ and $\lambda w_2-{\rm Le}^{-1}\Delta_xw_2-D_xw_2=0$, respectively.
Finally, the functions ${\mathscr P}_{\lambda,5}f_1$, ${\mathscr P}_{\lambda,6}f_2$ and ${\mathscr Q}_{\lambda,3}f_1$, ${\mathscr Q}_{\lambda,4}f_2$ belong to $C^{\infty}(H^+_R;\C)\cap C^1_b(\overline{H^+_R};\C)$ solves, in $H^+_R$, the equations $\lambda w_1-\Delta_xw_1-D_xw_1=0$ and $\lambda w_2-{\rm Le}^{-1}\Delta_xw_2-D_xw_2=0$, respectively, and vanish as $x$ tends to $+\infty$ for each $y\in\R$. Therefore, the function $\uu$ defined by \eqref{piove}-\eqref{sole} belongs to $\bigcap_{p<+\infty}W^{2,p}_{\rm loc}((\R\setminus\{0,R\})\times\R)\times\bigcap_{p<+\infty}W^{2,p}_{\rm loc}((\R_+\setminus\{R\})\times\R)$, solve the equation $\lambda\uu-{\mathscr L}\uu=\f$ and $\lim_{x\to \pm\infty}u_1(x,y)=\lim_{x\to +\infty}u_2(x,y)=0$ for each $y\in\R$.
Moreover, $\|\nabla{\bf u}\|_{\infty}\le c\|\f\|_{\infty}$.
To conclude that $\uu\in D(L)$, we have to check that ${\mathscr B}\uu={\bf 0}$, but this is an easy task taking into account that all the series appearing in the definition of $\uu$ may be differentiated term by term and ${\mathscr B}(\hat u_{1,k},\hat u_{2,k})={\bf 0}$ for every $k\in\Z$.
We have so proved that $\uu\in D(L)$ and that
\begin{eqnarray*}
\bigcup_{k\in\N\cup\{0\}}\Omega_k\subset\sigma(L)\subset \{\lambda\in\C:{\rm Re}\lambda\le -({\rm Im}\lambda)^2\}\cup\bigcup_{k\in\N\cup\{0\}}\Omega_k.
\end{eqnarray*}

{\em Step 2.} To complete the characterization of $\sigma(L)$, let us check that $\sigma(L)\supset\{\lambda\in\C:{\rm Re}\lambda\le -{\rm Le}^{-1}({\rm Im}\lambda)^2\}\cup\bigcup_{k\in\N\cup\{0\}}\Omega_k'$.
Clearly, each $\lambda\in\C$ such that ${\rm Re}\lambda\le -{\rm Le}^{-1}({\rm Im}\lambda)^2$ belongs to $\sigma(L)$, since in this case $\nu_0^{\pm}$ and $\mu_0^{\pm}$ have nonpositive real parts so that the more general solution to the equation $\lambda{\bf u}-{\mathcal L}{\bf u}={\bf 0}$, which belongs to
$\boldsymbol{\mathcal X}$ and is independent of $y$, is determined up to $8$ arbitrary complex constants and we have just $7$ boundary condition.
Thus, the previous equation admits infinitely many solutions in $\boldsymbol{\mathcal X}$.
Similarly, if $\lambda\in\Omega_k'$ for some $k\in\N\cup\{0\}$, then the pair $\uu=(\hat u_{1,k}e_k,\hat u_{2,k}e_k)$, where
\begin{align*}
\hat u_{1,k}(x)=({\mathcal F}_{k,1}f_1)(x)\!+\!(c_{1,k}e^{\nu^{-}_kx}\!+\!c_{2,k}e^{\nu^{+}_kx})\chi_{(0,R)}(x)\!+\!
(c_{3,k}e^{\nu^{-}_kx}\!+\!c_{4,k}e^{\nu^{+}_kx}\chi_{[R,+\infty)}(x),\qquad\;\,x\in\R,
\end{align*}
and $\hat u_{2,k}$ is still given by \eqref{u2k}, is smooth, belongs to $\boldsymbol{\mathcal X}$ and solves the differential equation $\lambda\uu-{\mathscr L}\uu={\bf f}$ for
every choice of $c_{i,k}$, $d_{j,k}$ ($i=1,\ldots,4$, $j=1,2,3$). Imposing the condition ${\mathscr B}\uu={\bf 0}$, we get to a linear
system of $7$ equations in $7$ unknowns whose determinant is $\widetilde {\mathcal D}_k(\lambda)$. Since $\lambda\in\Omega_k'$, the above equation
admits infinitely many solutions in $D(L)$.

{\em Step 3.} Since the roots of the dispersion relation have bounded from above real part (see also the forthcoming computations),
Step 1 shows that the resolvent set $\rho(L)$ contains a right-halfplane. Hence, to prove that $L$ generates an analytic semigroup it remains to prove that $|\lambda| \|R(\lambda,L)\f\|_{\infty}\leq c\|{\bf f}\|_{\infty}$ for each $\lambda$ in a suitable right-halfplane. Again, in view of Lemmata \ref{finito}, \ref{finire} and \ref{finiremo}, we can limit ourselves to dealing with the functions ${\mathscr P}_{\lambda,2k+j}f_j$ and ${\mathscr Q}_{\lambda,2h+j}f_j$.

For each $\lambda\in\C$ with positive real part, we can refine the estimate for ${\rm Re}(X_k)$ and ${\rm Re}(Y_k)$; it turns out that
\begin{align*}
&|X_k|\ge {\rm Re}(X_k)=\sqrt{\frac{|1\!+\!4\lambda\!+\!4\lambda_k|\!+\!{\rm Re}(1+4\lambda+4\lambda_k)}{2}}\ge \sqrt{2|\lambda|}\vee 1\vee 2\sqrt{\lambda_k}\ge \frac{\sqrt{3}}{3}\sqrt{|\lambda|+1+\lambda_k},\\[1mm]
&|X_k|\le 2\sqrt{1+|\lambda|+\lambda_k}
\end{align*}
and, similarly, $c_1\sqrt{1+|\lambda|+\lambda_k}\le {\rm Re}(Y_k)\le|Y_k|\le c_2\sqrt{1+|\lambda|+\lambda_k}$ for each $k\in\Z$ and $\lambda\in\C$ with positive real
part. As a byproduct, we get $|(\mathcal F_{k,1}f_1)(R)|+|({\mathcal G}_kf_2)(0)|\leq c_R(|\lambda|+1+k^2)^{-1}\|\bf f\|_{\boldsymbol{\mathcal X}}$ for each $k$ and $\lambda$ as above. Moreover, using \eqref{aaa} we can also estimate
$|W_k|\ge c_R\sqrt{1+|\lambda|+\lambda_k}$
for each $k\in\Z$ and $\lambda\in\Sigma_M:=\{\lambda\in\C: {\rm Re}\lambda\ge M\}$ with $M$ large enough. Finally,
\begin{align*}
|Y_k-{\rm Le}|=&\bigg |\frac{4{\rm Le}\lambda+4\lambda_k}{\sqrt{{\rm Le}^2+4{\rm Le}\lambda+\lambda_k}+{\rm Le}}\bigg |
\ge \frac{{\rm Le}|\lambda|+\lambda_k}{\sqrt{{\rm Le}^2+{\rm Le}|\lambda|+\lambda_k}}
\ge \frac{{\rm Le}}{2}\sqrt{1+|\lambda|+\lambda_k}
\end{align*}
for each $\lambda\in\Sigma_1$ and $k\in\Z$, since ${\rm Le}\in (0,1)$. Hence, up to replacing $M$ with $M\vee 1$, if necessary, we can estimate
\begin{align*}
&\sum_{j=1}^2(|p_{j,k}(x)|+|q_{j,k}(x')|)+\sum_{j=3}^4(|p_{j,k}(x')|+|q_{j,k}(x'')|)+\sum_{j=5}^6|p_{j,k}(x'')|
\le c_R(1+|\lambda|+\lambda_k)^{-\frac{1}{2}}
\end{align*}
for each $k\in\Z$, $\lambda\in\Sigma_M$, $x<0$, $x'\in (0,R)$ and $x''>R$.
We are almost done. Indeed, taking the above estimates and the fact that
\begin{align*}
\sum_{k\in\Z}\frac{\|{\bf f}\|_{\infty}}{(|\lambda|+1+k^2)^{3/2}}
\leq & c\|{\bf f}\|_{\infty}\int_0^{+\infty}(|\lambda|+1+r^2)^{-\frac{3}{2}}dr
\leq  \frac{c}{1+|\lambda|}\|{\bf f}\|_{\infty}
\end{align*}
into account, we easily conclude that
\begin{align*}
&\sum_{k=1}^3(\|\mathscr P_{\lambda,2k-1}f_1\|_{\infty}+\|\mathscr P_{\lambda,2k}f_2\|_{\infty})
+\sum_{k=1}^2(\|\mathscr Q_{\lambda,2k-1}f_1\|_{\infty}+\|\mathscr Q_{\lambda,2k}f_2\|_{\infty})\leq c_{M}|\lambda|^{-1}\|\bf f\|_{\boldsymbol{\mathcal X}}
\end{align*}
for every $\lambda\in\C$ with ${\rm Re}\lambda\ge M$ and some positive constant $c_M$ independent of $\lambda$. Similarly,
\begin{align*}
&k^{1-j}\{|D_x^{(j)}p_{i,k}(x)|+|D_{x'}^{(j)}p_{i+2,k}(x')|+|D_{x'}^{(j)}q_{i,k}(x')|+|D_{x''}^{(j)}q_{i,k}(x'')|+|D_{x''}^{(j)}p_{i+4,k}(x'')|\}
\le c_M,
\end{align*}
for each $x\le 0$, $x'\in [0,R]$, $x''\ge 0$, $i=1,2$, $j=0,1$ and
\begin{align*}
\sum_{k\in\Z}\frac{\|{\bf f}\|_{\infty}}{|\lambda|+1+k^2}
\leq & c\|{\bf f}\|_{\infty}\int_0^{+\infty}(|\lambda|+1+r^2)^{-1}dr
\leq  \frac{c}{\sqrt{1+|\lambda|}}\|{\bf f}\|_{\infty}.
\end{align*}
Thus, we deduce that
\begin{align*}
&\sum_{k=1}^3(\|\nabla\mathcal P_{\lambda,2k-1}f_1\|_{\infty}+\|\nabla\mathcal P_{\lambda,2k}f_2\|_{\infty})
+\sum_{k=1}^2(\|\nabla\mathscr Q_{\lambda,2k-1}f_1\|_{\infty}+\|\nabla\mathscr Q_{\lambda,2k}f_2\|_{\infty})\leq c_{M}|\lambda|^{-\frac{1}{2}}\|\bf f\|_{\infty}.
\end{align*}

{\em Step 4.} Finally, we show that if ${\bf f}\in \boldsymbol{\mathcal X}_{\alpha}$ then ${\bf u}\in \boldsymbol{\mathcal X}_{2+\alpha}$. Again, in view of Lemmata
\ref{finito}-\ref{finiremo} and the estimate $\|p_{1,k}\|_{C^h_b((-\infty,0];\C)}\le c_1e^{-c_2k}$ (for every $h\in\N$) in Step 3, which shows that the
function ${\mathscr P}_{\lambda,1}f_1$ belongs to $C^{\beta}_b(H_0^-)$ for every $\beta>0$ and $\|{\mathscr P}_{\lambda,1}f_1\|_{C^{\beta}_b(H_0^-)}\le c_{\lambda}\|f_1\|_{\infty}$,
it suffices to deal with the other functions ${\mathscr P}_{\lambda,2k+j}f_j$ and ${\mathscr Q}_{\lambda,2h+j}f_j$. We adapt the arguments in Step 2 of the proof of
Lemma \ref{finire}. To begin with, we consider the function
${\mathscr P}_{\lambda,2}f_2\in C^2_b(\overline{H^-_0})$ which solves the equation $\lambda {\mathscr P}_{\lambda,2}f_2-\Delta {\mathscr P}_{\lambda,2}f_2-D_x{\mathscr P}_{\lambda,2}f_2=0$ in $H^-_0$. To prove that it belongs to $C^{2+\alpha}_b(H^-_0)$, we check that $({\mathscr P}_{\lambda,2 }f_2)(0,\cdot)\in C^{2+\alpha}_b(\R)$. Note that $p_{2,k}(0)= \ell(\pi k)^{-1}+\widetilde p_{2,k}$ for every $k\in\Z\setminus\{0\}$, where $\widetilde p_{2,k}=O(k^{-2})$. Therefore, we can split
\begin{eqnarray*}
({\mathscr P}_{\lambda,2}f_2)(0,\cdot)=\frac{2}{{\rm Le}\,\pi}\sum_{k\in\Z}k^{-1}({\mathcal G}_kf_2)(0)e_k
+\frac{2}{{\rm Le}\, \ell}\sum_{k\in\Z}\widetilde p_{2,k}({\mathcal G}_kf_2)(0)e_k=\psi_1+\psi_2.
\end{eqnarray*}
Since $|\widetilde p_{2,k}({\mathcal G}_kf_2)(0)|\le c|k|^{-4}\|f_2\|_{\infty}$ for every $k\in\Z\setminus\{0\}$, it follows immediately that
$\psi_2\in C^{2+\alpha}_b(\R)$ and $\|\psi_2\|_{C^{2+\alpha}_b(\R)}\le c\|f_2\|_{\infty}$.
As far as $\psi_1$ is concerned, a straightforward computation reveals that
$\psi_1'=c({\mathscr S}_{\lambda}f_2)(0,\cdot)$ so that, by Lemma \ref{finire}, $\psi_1'\in C^{1+\alpha}_b(\R)$ and
$\|\psi_1'\|_{C^{1+\alpha}_b(\R)}\leq c\|\f\|_{\alpha}$. Thus, $({\mathscr P}_{\lambda,2}f_2)(0,\cdot)$ belongs to $C^{2+\alpha}_b(\R)$ and $\|({\mathscr P}_{\lambda,2}f_2)(0,\cdot)\|_{C^{2+\alpha}_b(\R)}\le c_{\lambda}\|\f\|_{\alpha}$.
By classical results for elliptic problems (see \cite{krylov}), ${\mathscr P}_{\lambda,2}f_2$ belongs to $C^{2+\alpha}_b(H_0^-)$  and $\|{\mathscr P}_{\lambda,2}f_2\|_{C^{2+\alpha}_b(H_0^-)}\le c_{\lambda}\|\f\|_{\alpha}$.

Next, we split ${\mathscr P}_{\lambda,3}f_1$ into the sum of the functions
\begin{align*}
{\mathscr P}_{3,\lambda,1}f_1=\frac{1}{\ell}\sum_{k\in\Z}\frac{e^{\nu_k^+(\cdot-R)}}{W_k}({\mathcal F}_{k,1}f_1)(R)e_k,\qquad\;\,
{\mathscr P}_{3,\lambda,2}f_1=\frac{1}{\ell}\sum_{k\in\Z}\frac{e^{-\mu_k^+R}}{W_k}({\mathcal F}_{k,1}f_1)(R)e^{\nu_k^-\cdot}e_k.
\end{align*}
The first function belongs to $C^2_b(\overline{H^{-}_R};\C)$ and
$\lambda {\mathscr P}_{3,\lambda,1}f_1-\Delta {\mathscr P}_{3,\lambda,1}f_1-D_x{\mathscr P}_{3,\lambda,1}f_1=0$ in $H^-_R$.
Since
\begin{eqnarray*}
(({\mathscr P}_{3,\lambda,1}f_1)(R,\cdot))'=c({\mathscr R}_{\lambda,1}f_1)(R,\cdot)+\frac{1}{\ell}\sum_{k\in\Z}\widetilde p_{3,k}(\mathcal F_{k,1}f_1)(R)e_k
\end{eqnarray*}
and $|\widetilde p_{3,k}|\le ck^{-1}$, the same arguments as above and Lemma \ref{finito} allow to show first that the function $(({\mathscr P}_{3,\lambda,1}f_1)(R,\cdot))'$ belongs to $C^{1+\alpha}_b(\R)$ and then to
conclude that ${\mathscr P}_{3,\lambda,1}f_1\in C^{2+\alpha}_b(H^-_R)$ and $\|{\mathscr P}_{3,\lambda,1}f_1\|_{C^{2+\alpha}_b(H^-_R)}\le c_{\lambda}\|\f\|_{\alpha}$.
The smoothness of the function ${\mathscr P}_{3,\lambda,2}f_1$ is easier to prove, due to the uniform (in $[0,R]$) exponential decay to zero of the terms of the series. It turns out that $\|{\mathscr P}_{3,\lambda,2}f_1\|_{C^{2+\alpha}_b(H^-_R)}\le c_{\lambda}\|f_1\|_{\infty}$.

Let us consider the function ${\mathscr P}_{4,\lambda}f_2$, which we split it into the sum of the functions ${\mathscr P}_{4,\lambda,1}f_2$ and ${\mathscr P}_{4,\lambda,2}f_2$ defined, respectively, by
\begin{align*}
{\mathscr P}_{4,\lambda,j}f_2=\frac{2}{{\rm Le} \ell}\sum_{k\in\N}p_{4,j,k}(\mathcal G_kf_2)(0)e_k,\qquad\;\,j=1,2,
\end{align*}
where
\begin{align*}
p_{4,1,k}(x)= \frac{\theta_iR}{W_k}e^{\nu_k^-x}, \qquad\;\,
p_{4,2,k}(x)= \frac{\theta_iR{\rm Le}\,e^{\nu_k^-x}}{W_k(Y_k-{\rm Le})}-\frac{Y_ke^{\nu_k^-x}}{X_kW_k(Y_k-{\rm Le})}+\frac{Y_ke^{-X_kR+\nu_k^+x}}{X_kW_k(Y_k-{\rm Le})}.
\end{align*}
Function ${\mathscr P}_{4,\lambda,1}f_2$ belongs to $C^2_b(\overline{H_0^+})$ and
$\lambda {\mathscr P}_{4,\lambda,1}f_2-\Delta {\mathscr P}_{4,\lambda,1}f_2-D_x{\mathscr P}_{4,\lambda,1}f_2=0$ in $H^+_0$. Moreover,
$({\mathscr P}_{4,\lambda,1}f_2)(0,\cdot)$ is an element of $C^{2+\alpha}_b(\R)$ and $\|({\mathscr P}_{4,\lambda,1}f_2)(0,\cdot)\|_{C^{2+\alpha}_b(\R)}\le c_{\lambda}\|\f\|_{\alpha}$, so that ${\mathscr P}_{4,\lambda,1}f_2\in C^{2+\alpha}_b(H^+_0)$ and $\|{\mathscr P}_{4,\lambda,1}f_2\|_{C^{2+\alpha}_b(H^+_0)}\le c_{\lambda}\|\f\|_{\alpha}$. On the other hand, the series, which defines
${\mathscr P}_{4,\lambda,2}f_2$ is easier to analyze since it converges in $C^{2+\alpha}_b(H^-_R)$ and $\|{\mathscr P}_{4,\lambda,2}f_2\|_{C^{2+\alpha}_b(H^-_R)}\leq c_{\lambda}\|\f\|_{\alpha}$.

All the remaining functions ${\mathscr P}_{\lambda,2k+j}f_j$ and ${\mathscr Q}_{\lambda,2h+j}f_j$ can be analyzed in the same way. The details are left to the reader.
\end{proof}

Now, we characterize the interpolation spaces $D_L(\alpha/2,\infty)$ and $D_L(1+\alpha/2,\infty)$.
To simplify the notation, we introduce the operator ${\mathscr B}_0$, defined by
\begin{equation}
{\mathscr B}_0{\bf u}=(u_1(0^+,\cdot)-u_1(0^-,\cdot),u_1(R^+,\cdot)-u_1(R^-,\cdot)),\qquad\;\,{\bf u}\in\boldsymbol{\mathcal X},
\label{operatore-B0}
\end{equation}
and the sets $\boldsymbol{\mathcal X}_{\alpha,{\mathscr B}_0}=\{{\bf u}\in \boldsymbol{\mathcal X}_{\alpha}: {\mathscr B}_0{\bf u}={\bf 0}\}$ $(\alpha\in (0,1])$ and $\boldsymbol{\mathcal X}_{2+\alpha,{\mathscr B}}=\{{\bf u}\in \boldsymbol{\mathcal X}_{2+\alpha}: {\mathscr B}{\bf u}={\bf 0}, {\mathscr B}_0L{\bf u}={\bf 0},\, \lim_{x\to \pm\infty}(L\uu)_1(x,y)=
\lim_{x\to +\infty}(L\uu)_2(x,y)=0\}$ ($\alpha\in (0,1)$), equipped with the norm of $\boldsymbol{\mathcal X}_{\alpha}$ and $\boldsymbol{\mathcal X}_{2+\alpha}$, respectively.

\begin{prop}
\label{perdita}
For each $\alpha\in (0,1)$ the following characterizations hold:
\begin{align}
(i)~D_L(\alpha/2,\infty)=\boldsymbol{\mathcal X}_{\alpha,{\mathscr B}_0},\qquad\;\, (ii)~D_L(1+\alpha/2,\infty)
= & \boldsymbol{\mathcal X}_{2+\alpha,\mathscr B},
\label{penna}
\end{align}
with equivalence of respective norms. Moreover,
\begin{align}
\label{doccia}
\boldsymbol{\mathcal X}_{1,\mathscr B_0}\hookrightarrow D_L(1/2,\infty).
\end{align}
\end{prop}

\begin{proof}
Throughout the proof, we assume that $\alpha$ is arbitrarily fixed in $(0,1)$.

{\em Step 1: proof of \eqref{penna}$(i)$ and \eqref{doccia}}. Given ${\bf f}=(f_1,f_2)\in\boldsymbol{\mathcal X}_{\alpha,{\mathscr B}_0}$
and $t>0$, we introduce the functions $f_{t,1}$ and $f_{t,2}$ defined by
\begin{align*}
&f_{t,1}(x,y)=\int_{H_0^+}f_1^{\sharp}(x+tx',y+ty')\varphi(x',y')dx'dy',\qquad\;\,(x,y)\in\R^2,\\
&f_{t,2}(x,y)=\int_{H_0^+}\widetilde{f_2}(x-tx',y+ty')\varphi(x',y')dx'dy',\qquad\;\,(x,y)\in [0,R]\times\R,\\
&f_{t,2}(x,y)=\int_{H_0^+}f_2^{\sharp}(x+tx',y+ty')\varphi(x',y')dx'dy',\qquad\;\,(x,y)\in H_R^+,
\end{align*}
where $\varphi$ is a positive smooth function with compact support in $(0,R)\times\R$, $\|\varphi\|_{L^1(\R^2)}=1$
and $\widetilde f_2:\overline{H^-_R}\to\C$ equals the function $f_2^{\sharp}\vartheta$ in $[0,R]\times\R$, whereas
$\widetilde f_2(x,\cdot)=f_2^{\sharp}(0,\cdot)\vartheta(x)$ if $x<0$ and $\vartheta$ is a smooth function compactly supported in $(-1,+\infty)$ and equal to $1$ in $(-1/2,+\infty)$. Since ${\mathscr B}_0{\bf f}={\bf 0}$ and $f_j(\cdot,-\ell/2)=f_j(\cdot,\ell/2)$ for $j=1,2$, the function
$f_1^{\sharp}$ belongs to $C^{\alpha}_b(\R^2;\C)$. On the other hand, the functions $\widetilde f_2$
and $f_2^{\sharp}$ belong to $C^\alpha(H^-_R;\C)$ and to $C_b^{\alpha}(H^+_R;\C)$, respectively. Moreover, $f_1^{\sharp}(\cdot,y)$ and $f_2^{\sharp}(\cdot,y)$ vanish at $\pm\infty$ and at $+\infty$, respectively, for every $y\in\R$. Hence, if we set ${\bf f}_t=(f_{t,1},f_{t,2})$, then $\f_t$ belongs to $\boldsymbol{\mathcal X}$, $\|{\bf f-f}_t\|_{\infty}\leq ct^\alpha\|{\bf f}\|_{\alpha}$ and $\|{\bf f}_t\|_{\infty}\le c\|{\bf f}\|_{\alpha}$ for every $t>0$.
Similarly, since $\int_{H_0^+}D^{\gamma}\varphi(x',y')dx'dy'=0$ for every multi-index $\gamma$, we can write
\begin{eqnarray*}
D^{\gamma}f_{t,1}(x,y)=t^{-|\gamma|}\int_{H_0^+}(f_1^{\sharp}(x+tx',y+ty')-f_1^{\sharp}(x,y))D^{\gamma}\varphi(x',y')dx'dy',\qquad\;\,(x,y)\in\R^2,
\end{eqnarray*}
so that $|D^{\gamma}f_{t,1}(x,y)|\le ct^{\alpha-|\gamma|}\|{\bf f}\|_{\alpha}$
for every $(x,y)\in\R^2$. In the same way we can estimate the derivatives of the function $f_{t,2}$, and conclude that
$t^{|\gamma|-\alpha}\|D^{\gamma}{\bf f}_t\|_{\infty}\leq c_T\|{\bf f}\|_{\alpha}$ for every $t\in (0,T]$ and $T>0$.

Now, we split $\omega^{\alpha/2}LR(\omega, L){\bf f}=\omega^{\alpha/2}LR(\omega, L)({\bf f}-{\bf f}_{\omega^{-1/2}})+\omega^{\alpha/2}LR(\omega, L){\bf f}_{\omega^{-1/2}}$ for each $\omega\in\rho(L)\cap\R$. By Theorem \ref{banca}, for $\omega\in\R$ sufficiently large (let us say $\omega\ge\omega_0>0$), it holds that $\|LR(\omega,L)\|_{L(\boldsymbol{\mathcal X})}\leq M$ for some positive constant $M$, independent of $\omega$. Hence, using the above estimates, we can infer that
$\|\omega^{\alpha/2}LR(\omega, L)({\bf f-f}_{\omega^{-1/2}})\|_{\infty}\le c\|\f\|_{\alpha}$. Next, we consider the term ${\bf v}=LR(\omega, L){\bf f}_{\omega^{-1/2}}$, which belongs to $D(L)$ and solves the equation  $\omega {\bf v}-L{\bf v}=L{\bf f}_{\omega^{-1/2}}$. If $\omega\ge\omega_0$, then we can estimate
$\omega\|{\bf v}\|_{\infty}\leq c\|L{\bf f}_{\omega^{-1/2}}\|_{\infty}\leq c\omega^{1-\alpha/2}\|{\bf f}\|_{\alpha}$.
Putting everything together, we conclude that $\omega^{\alpha/2}\|LR(\omega, L){\bf f}\|_{\infty}\leq c\|{\bf f}\|_{\alpha}$ for each $\omega\ge\omega_0$ and this
shows that ${\bf f}\in D_L(\alpha/2,\infty)$ and $\|{\bf f}\|_{D_L(\alpha/2,\infty)}\leq c\|{\bf f}\|_{\alpha}$.

Formula \eqref{doccia} can be proved just in the same way observing that, if ${\bf f}$ belongs to $\boldsymbol{\mathcal X}_{1,{\mathscr B}_0}$, then
the functions $f_1^{\sharp}$, $\widetilde f_2$ and $f_2^{\sharp}$ are bounded and Lipschitz continuous in $\R^2$ in $[0,R]\times\R$ and in $H_R^+$, respectively.

The embedding ``$\hookrightarrow$'' in \eqref{penna}(i) is a straightforward consequence of two properties:
\begin{align*}
(a)~\|\nabla{\bf u}\|_{\infty}\leq c\|{\bf u}\|_{\infty}^{\frac{1}{2}}\|{\bf u}\|_{D(L)}^{\frac{1}{2}},\quad\;\,{\bf u}\in D(L),\qquad\;\,(b)~(\boldsymbol{\mathcal X},\boldsymbol{\mathcal X}_{1,{\mathscr B}_0})_{\alpha,\infty}\hookrightarrow \boldsymbol{\mathcal X}_{\alpha,{\mathscr B}_0}.
\end{align*}
Indeed, property (a) shows that $\boldsymbol{\mathcal X}_{1,{\mathscr B}_0}$ belongs to the class $J_{1/2}$ between $\boldsymbol{\mathcal X}$ and $D(L)$, so that applying the reiteration theorem, we get
$D_L(\alpha/2,\infty)=(\boldsymbol{\mathcal X},D(L))_{\alpha/2,\infty}\subset (\boldsymbol{\mathcal X},\boldsymbol{\mathcal X}_{1,{\mathscr B}_0})_{\alpha,\infty}$ and conclude using (b).

{\em Proof of $(a)$}. It is an almost straightforward consequence of the estimate $\|\nabla R(\lambda,L)\f\|_{\infty}\le c|\lambda|^{-1/2}\|\f\|_{\infty}$ for each $\lambda\in\C$ with ${\rm Re}\lambda\ge M$, contained in Theorem \ref{banca}.
Indeed, let $L_M=L-MI$. As it is easily seen, $\rho(L_M)\supset (0,+\infty)$ and
$R(\lambda,L_M)=R(\lambda+M,L)$ for each $\lambda>0$. It thus follows that
$\sqrt{\lambda}\|\nabla R(\lambda,L_M)\|_{L(\boldsymbol{\mathcal X};\boldsymbol{\mathcal X}\times{\boldsymbol{\mathcal X}})}\le c$ for each $\lambda>0$,
so that we can estimate
\begin{eqnarray*}
\|\nabla {\bf u}\|_{\infty}=\|\nabla R(\lambda,L_M)(\lambda {\bf u}-L_M{\bf u})\|_{\infty}
\le c\lambda^{-\frac{1}{2}}\|\lambda {\bf u}-L_M{\bf u}\|_{\infty}\le c(\lambda^{\frac{1}{2}}\|{\bf u}\|_{\infty}+
\lambda^{-\frac{1}{2}}\|L_M{\bf u}\|_{\infty})
\end{eqnarray*}
for each ${\bf u}\in D(L)=D(L_M)$ and $\lambda>0$. Minimizing with respect to $\lambda>0$, we conclude the proof.

{\em Proof of $(b)$}. Let us fix a nontrivial ${\bf f}\in (\boldsymbol{\mathcal X},\boldsymbol{\mathcal X}_{1,{\mathscr B}_0})_{\alpha,\infty}$.
Since $(\boldsymbol{\mathcal X},\boldsymbol{\mathcal X}_{1,{\mathscr B}_0})_{\alpha,\infty}\hookrightarrow (\boldsymbol{\mathcal X},\boldsymbol{\mathcal X}_{1,{\mathscr B}_0})_{\alpha/2,1}$ and $\boldsymbol{\mathcal X}_{1,{\mathscr B}_0}$ is dense in $(\boldsymbol{\mathcal X},\boldsymbol{\mathcal X}_{1,{\mathscr B}_0})_{\alpha/2,1}$, we immediately deduce that $f_j(\cdot,-\ell/2)=f_j(\cdot,\ell/2)$ for $j=1,2$ and ${\mathscr B}_0{\bf f}={\bf 0}$.
Next, we recall that
\begin{eqnarray*}
\inf\left\{\|{\bf g}\|_{\infty}+t\|{\bf h}\|_{1}: {\bf f}={\bf g}+{\bf h},\,
{\bf g}\in\boldsymbol{\mathcal X}, {\bf h}\in\boldsymbol{\mathcal X}_{1,{\mathscr B}_0}\right\}
\le t^{\alpha}\|{\bf f}\|_{(\boldsymbol{\mathcal X},\boldsymbol{\mathcal X}_{1,{\mathscr B}_0})_{\alpha,\infty}},\qquad\;\,t>0.
\end{eqnarray*}
Fix $(x_j,y_j)\in S^-_0$, $j=1,2$, and take $t=\sqrt{|x_2-x_1|^2+|y_2-y_1|^2}$. Then, we can determine
${\bf g}\in\boldsymbol{\mathcal X}$ and ${\bf h}\in\boldsymbol{\mathcal X}_{1,{\mathscr B}_0}$ such that
${\bf f}={\bf g}+{\bf h}$ and
\begin{eqnarray*}
\|{\bf g}\|_{\infty}+\sqrt{|x_2-x_1|^2+|y_2-y_1|^2}\|{\bf h}\|_1\le
2\|{\bf f}\|_{(\boldsymbol{\mathcal X},\boldsymbol{\mathcal X}_{1,{\mathscr B}_0})_{\alpha,\infty}}(|x_2-x_1|^2+|y_2-y_1|^2)^{\frac{\alpha}{2}}.
\end{eqnarray*}
From this estimate we can infer that
\begin{align*}
&|f_1(x_2,y_2)-f_1(x_1,y_1)|\\
\le & |g_1(x_2,y_2)-g_1(x_1,y_1)|+|h_1(x_2,y_2)-h_1(x_1,y_1)|
\le 2\|{\bf g}\|_{\infty}+\sqrt{|x_2-x_1|^2+|y_2-y_1|^2}\|{\bf h}\|_1\\
\le & 4\|{\bf f}\|_{(\boldsymbol{\mathcal X},\boldsymbol{\mathcal X}_{1,{\mathscr B}_0})_{\alpha,\infty}}(|x_2-x_1|^2+|y_2-y_1|^2)^{\frac{\alpha}{2}}.
\end{align*}
Hence, $f_1\in C_b^{\alpha}(S^-_0)$ and
$\|f_1\|_{C_b^{\alpha}(S^-_0)}\le 5\|{\bf f}\|_{(\boldsymbol{\mathcal X},\boldsymbol{\mathcal X}_{1,{\mathscr B}_0})_{\alpha,\infty}}$.

The same argument can be used to prove that $f_1, f_2\in C^{\alpha}((0,R)\times (-\ell/2,\ell/2);\C)\cap C^{\alpha}_b(S_R^+;\C)$ and
\begin{align*}
&\|f_1\|_{C^{\alpha}_b(S_R^+;\C)}+\sum_{j=1}^2\|f_j\|_{C^{\alpha}((0,R)\times (-\ell/2,\ell/2);\C)}+\|f_2\|_{C^{\alpha}_b(S_R^+;\C)}\le c
\|{\bf f}\|_{(\boldsymbol{\mathcal X},\boldsymbol{\mathcal X}_{1,{\mathscr B}_0})_{\alpha,\infty}}.
\end{align*}

The proof of (b) is now complete.

{\em Step 2: proof of \eqref{penna}$(ii)$.} The embedding ``$\hookleftarrow$'' easily follows from the first part of the proof.
The other embedding follows from Theorem \ref{banca}. Indeed, fix ${\bf u}\in D_L(1+\alpha/2,\infty)$ and $\lambda\in\rho(L)$. Then, the function
${\bf f}:=\lambda {\bf u}-L{\bf u}$ belongs to $D_L(\alpha/2,\infty)=\boldsymbol{\mathcal X}_{\alpha,{\mathscr B}_0}$ and
$\|{\bf f}\|_{\alpha}\le c\|{\bf u}\|_{D_L(1+\alpha/2,\infty)}$ for some positive constant $c$, independent of ${\bf u}$.
Since ${\bf u}=R(\lambda,L){\bf f}$, Theorem \ref{banca} implies that ${\bf u}\in\boldsymbol{\mathcal X}_{2+\alpha}$ and
$\|{\bf u}\|_{2+\alpha}\le c\|{\bf f}\|_{\alpha}\le c\|{\bf u}\|_{D_L(1+\alpha/2,\infty)}$. Clearly, ${\mathscr B}{\bf u}={\bf 0}$, ${\mathscr B}_0L{\bf u}={\bf 0}$
and $\lim_{x\to\pm\infty}(L\uu)_1(x,y)=\lim_{x\to +\infty}(L\uu)_2(x,y)=0$, and this completes the proof.
\end{proof}

\begin{rmk}
\label{rmk-5.6}
{\rm From the classical theory of analytic semigroups (see e.g., \cite{lunardi}) and Proposition \ref{perdita} it follows that the part $L_{\alpha}$ of $L$ in
${\boldsymbol {\mathcal X}}_{\alpha,{\mathscr B}_0}$,
i.e., the restriction of $L$ to ${\boldsymbol {\mathcal X}}_{2+\alpha,{\mathscr B}_0}$,
generates an analytic semigroup for each $\alpha\in (0,1)$.}
\end{rmk}

\subsection{The lifting operators}
\label{subsection-lifting}
In this subsection we introduce some lifting operators which are used in the proof of the Main Theorem and Theorem \ref{soliera}.

To begin with, we consider the operator $\mathscr M$ defined by $\mathscr M\boldsymbol\psi=(0,M\psi_1+\frac{1}{2}(M\psi_2)(\cdot-R,\cdot))$ on functions $\boldsymbol\psi\in C([-\ell/2,\ell/2];\R^2)$ such that $\boldsymbol{\psi}(-\ell/2)=\boldsymbol{\psi}(\ell/2)$, where
\begin{align*}
(M\psi_1)(x,y) & :=|x|\eta (x)\int_{\R}\varphi(\xi)\psi_1^{\sharp}(y+\xi x)d\xi,\qquad\;\,(x,y)\in\R^2.
\end{align*}
Here, $\eta$ and $\varphi$ are smooth functions such that $\chi_{(-R/4,R/4)}\le\eta\le\chi_{(-R/2,R/2)}$,  $\varphi$ is an even nonnegative function compactly supported in $(-1,1)$ with $\|\varphi\|_{L^1(\R)}=1$. As it is easily seen, ${\mathscr M}\boldsymbol\psi\in\boldsymbol{\mathcal X}_{2+\alpha}$,
$\mathscr B\mathscr M\boldsymbol\psi=(0,0,0,\psi_1,0,0,\psi_2)$ for each $\boldsymbol\psi$ as above.

Next, we introduce the operator ${\mathscr N}$ defined by
\begin{align*}
(\mathscr N\boldsymbol h)_1=&N_1h_1+(N_1h_2)\circ\tau_R+\frac{N_2h_3}{2{\rm Le}}+\frac{(N_2h_5)\circ\tau_R}{2\theta_i R}+N_3\bigg (h_8-\frac1{{\rm Le}}h_3-h_1''\bigg )\\
&+\bigg [N_3\bigg (h_9-\frac{1}{\theta_i R}h_5-h_2''\bigg )\bigg ]\circ\tau_R,\\
({\mathscr N}\boldsymbol h)_2=& \bigg [N_1\bigg (h_6-\frac{\rm Le}{\theta_i R}h_5\bigg )\bigg ]\circ\tau_R+N_2h_4+\frac{1}{2}\bigg [N_3\bigg (h_7-{\rm Le} h_6+\frac{{\rm Le}^2}{\theta_i R}h_5\bigg )\bigg ]\circ\tau_R
\end{align*}
on smooth enough functions $\boldsymbol h:[-\ell/2,\ell/2]\to\R^9$, where $\tau_R(x,y)=(x-R,y)$,
\begin{align*}
(N_1\zeta)(x,y)=&\frac{1}{2}\eta(x)[\chi_{[0,+\infty)}(x)-\chi_{(-\infty,0]}(x)]\int_{\R}\varphi(\sigma)\zeta^\sharp(y+\sigma x^3)d\sigma,\\[1mm]
(N_2\zeta)(x,y)=&|x|\eta (x)\int_{\R}\varphi(\sigma)\zeta^{\sharp}(y+\sigma |x|)d\sigma,
\qquad  (N_3\zeta)(x,y)=\frac{x}{4}(N_2\zeta)(x,y),
\end{align*}
for each $(x,y)\in\R^2$. Moreover, we set ${\mathscr B}_*{\bf v}=(\widetilde{\mathscr B}{\bf v},{\mathscr B}_0{\mathscr L}{\bf v})$ for each ${\bf v}\in\boldsymbol{\mathcal X}_{2+\alpha}$, where $\widetilde{\mathscr B}$ is the operator in Remark \ref{rmk-opB} and the operator ${\mathscr B}_0$ is defined in \eqref{operatore-B0}.

In the next lemma, we deal with real valued spaces. In particular, by $\widetilde D(L_{\alpha})$ we denote the subset of $D(L_{\alpha})$ of real valued functions.

\begin{lemm}
\label{lemma-5.7}
The following properties are satisfied.
\begin{enumerate}[\rm (i)]
\item
The operator ${\mathscr N}$ is bounded from the set $(\boldsymbol{\mathcal X}_{2+\alpha})^2\times (\boldsymbol{\mathcal X}_{1+\alpha})^5\times (\boldsymbol{\mathcal X}_{\alpha})^2$ into $\boldsymbol{\mathcal X}_{2+\alpha}$. Moreover, the operator $P=I-{\mathscr N}{\mathscr B}_*:\boldsymbol{\mathcal X}_{2+\alpha}\to \boldsymbol{\mathcal X}_{2+\alpha}$ is a projection onto the kernel of $B_*$ which coincides with $\widetilde D(L_\alpha)$\hspace{3pt}\footnote{see Remark \ref{rmk-5.6}}.
\item
Let $\boldsymbol{\mathcal I}$ denote  the set of all functions $\uu\in {\boldsymbol {\mathcal X}}_{2+\alpha}$ such that ${\mathscr B}{\bf u}=\mathscr H({\bf u})$, $\mathscr B_0(\mathcal L{\bf u}+\mathscr F({\bf u}))={\bf 0}$.
Then, there exist $r_0, r_1>0$ such that $\boldsymbol{\mathcal I}\cap  B(0,r_0)$ is the graph of a smooth function $\Upsilon:B(0,r_1)\subset\widetilde D(L_{\alpha})\to (I-P)(\boldsymbol{\mathcal X}_{2+\alpha})$ such that $\Upsilon({\bf 0})={\bf 0}$.
\end{enumerate}
\end{lemm}

\begin{proof}
(i) Showing that ${\mathscr N}$ is a bounded operator is an easy task.
Some long but straightforward computations reveal that ${\mathscr B}_*{\mathscr N}=I$ on $(\boldsymbol{\mathcal X}_{2+\alpha})^2\times (\boldsymbol{\mathcal X}_{1+\alpha})^5\times (\boldsymbol{\mathcal X}_{\alpha})^2$ and allow to prove that $P$ is a projection onto
of ${\rm Ker}(B_*)=\widetilde D(L_{\alpha})$. In particular, we can split $\boldsymbol{\mathcal X}_{2+\alpha}=\widetilde D(L_{\alpha})\oplus (I-P)(\boldsymbol{\mathcal X}_{2+\alpha})$.
The details are left to the reader.

(ii) Let ${\mathscr K}:B(0,r_0)\subset\widetilde D(L_\alpha)\oplus (I-P)(\boldsymbol{\mathcal X}_{2+\alpha})\to (\boldsymbol{\mathcal X}_{2+\alpha})^2\times (\boldsymbol{\mathcal X}_{1+\alpha})^5\times (\boldsymbol{\mathcal X}_{\alpha})^2$ be the operator defined by
${\mathscr K}({\bf u},{\bf v})=(\mathscr B({\bf u}+{\bf v})-\mathscr H({\bf u}+{\bf v}), \mathscr B_0(\mathcal L{\bf u}+\mathcal L{\bf v}+\mathscr F({\bf u}+{\bf v})))$ for each $(\uu,{\bf v})\in B(0,r_0)$, with $r_0$ small enough to guarantee that ${\mathscr K}$ is well defined. Since the functions $\mathscr F$ and $\mathscr H$ are quadratic at ${\bf 0}$, it follows that ${\mathscr K}({\bf 0},{\bf 0})={\bf 0}$ and ${\mathscr K}$ is Fr\'echet differentiable at $({\bf 0},{\bf 0})$, with ${\mathscr K}_{\bf u}({\bf 0},{\bf 0})={\mathscr B}_*$. In view of (i), ${\mathscr B}_*$ is an isomorphism from $(I-P)(\boldsymbol{\mathcal X}_{2+\alpha})$ to $(\boldsymbol{\mathcal X}_{2+\alpha})^2\times (\boldsymbol{\mathcal X}_{1+\alpha})^5\times (\boldsymbol{\mathcal X}_{\alpha})^2$. Thus, we can invoke the implicit function theorem to complete the proof.
\end{proof}

\section{Solving the nonlinear problem \eqref{asta}}
\label{sect-5}

Now we are able to solve the nonlinear Cauchy problem
\begin{align}
\left\{
\begin{array}{lll}
D_t{\bf u}={\mathscr L}{\bf u}+\mathscr F({\bf u}),\\[1mm]
\mathscr B{\bf u}=\mathscr H({\bf u}), \\[1mm]
D^{\gamma_1}_xD^{\gamma_2}_y{\bf u}(\cdot,\cdot,-\ell/2)=D^{\gamma_1}_xD^{\gamma_2}_y{\bf u}(\cdot,\cdot,\ell/2), &&\gamma_1+\gamma_2\le 2,
\end{array}
\right.
\label{asta1}
\end{align}
for the unknown ${\bf u}=(u,w)$. Also in this section we assume that the function spaces that we deal with are real valued ones.
\begin{thm}\label{soliera}
Fix $\alpha\in(0,1)$ and $T>0$. Then, there exists $r_0=r_0(T)>0$ such that, for each ${\bf u}_0\in B(0,r_0)\subset \boldsymbol{\mathcal X}_{2+\alpha}$ satisfying the compatibility conditions $\mathscr B{\bf u}_0=\mathscr H({\bf u}_0)$, $\mathscr B_0({\mathscr L}{\bf u}_0+\mathscr F({\bf u_0}))={\bf 0}$ and $D^{\gamma}{\bf u}_0(\cdot,-\ell/2)=D^{\gamma}{\bf u_0}(\cdot,\ell/2)$ for each multi-index $\gamma$ with length at most two,
Problem \eqref{asta1} admits a unique solution ${\bf u}\in\boldsymbol{\mathcal Y}_{2+\alpha}$ with ${\bf u}(0,\cdot)={\bf u}_0$. Moreover,
$\|{\bf u}\|_{\boldsymbol{\mathcal Y}_{2+\alpha}}\leq c\|{\bf u}_0\|_{2+\alpha}$.
\end{thm}

\begin{proof}
The proof can be obtained arguing as in the proof of \cite[Theorem 4.1]{lorenzi-1}. For this purpose, we just sketch the main points.

We first need to prove optimal regularity results for the linear version of Problem \eqref{asta1}, i.e., with the problem
\begin{equation}
\label{wallace}
\left\{
\begin{array}{lll}
D_t{\bf u}(t,\cdot,\cdot)={\mathscr L}{\bf u}(t,\cdot,\cdot)+{\bf f}(t,\cdot,\cdot), &t\in [0,T],\\[1mm]
\mathscr B({\bf u}(t,\cdot,\cdot))={\bf h}(t,\cdot), &t\in [0,T],\\[1mm]
D^{\gamma_1}_xD^{\gamma_2}_y{\bf u}(t,\cdot,-\ell/2)=D^{\gamma_1}_xD^{\gamma_2}_y{\bf u}(t,\cdot,\ell/2), &t\in [0,T], &\gamma_1+\gamma_2\le 2,\\[1mm]
{\bf u}(0,\cdot)={\bf u_0},
\end{array}
\right.
\end{equation}
when ${\bf f}\in \boldsymbol{\mathcal Y}_{\alpha}$, ${\bf u_0}\in\boldsymbol{\mathcal X}_{2+\alpha}$,
when $h_j\equiv 0$ if $j\neq 4,7$, $h_4=\psi_1$, $h_7=\psi_2$, $\boldsymbol{\psi}=(\psi_1,\psi_2)\in C^{(1+\alpha)/2,1+\alpha}((0,T)\times (-\ell/2,\ell/2);\R^2)$ satisfy the compatibility conditions
\begin{itemize}
\item
$\mathscr B{\bf u_0}={\bf h}(0,\cdot)$, $\mathscr B_0({\mathscr L}{\bf u_0}(0,\cdot)+{\bf f}(0,\cdot))={\bf 0}$;
\item
${\bf f}(0,\cdot,-\ell/2)={\bf f}(0,\cdot,\ell/2)$, $D^{\gamma}{\bf u_0}(\cdot,-\ell/2)=D^{\gamma}{\bf u_0}(\cdot,\ell/2)$ and
 $D_y^{(j)}\boldsymbol\psi(\cdot,-\ell/2)=D_y^{(j)}\boldsymbol\psi(\cdot,\ell/2)$ for every multi-index $\gamma$ with length at most two and $j=0,1$.
\end{itemize}
We also need to show the estimate
\begin{align}
\|{\bf u}\|_{\boldsymbol{\mathcal Y}_{2+\alpha}}\leq c_0(\|{\bf f}\|_{\boldsymbol{\mathcal Y}_{\alpha}}
+\|{\bf u_0}\|_{2+\alpha}+\|\boldsymbol{\psi}\|_{C^{(1+\alpha)/2,1+\alpha}((0,T)\times (-\ell/2,\ell/2);\R^2)}).
\label{argine}
\end{align}
for its unique solution ${\bf u}\in\boldsymbol{\mathcal Y}_{2+\alpha}$. This is the content of Steps 1 to 3.

{\em Step 1.} To begin with, we note that $Mh\in C^{(1+\alpha)/2,2+\alpha}_b((0,T)\times H_0^-)\cap C^{(1+\alpha)/2,2+\alpha}_b((0,T)\times H_0^+)$
for all $h\in C^{(1+\alpha)/2,1+\alpha}((0,T)\times (-\ell/2,\ell/2))$ such that $D_y^{(j)}h(\cdot,-\ell/2)=D_y^{(j)}h(\cdot,\ell/2)$ ($j=0,1$), and
\begin{align*}
\|Mh\|_{C^{(1+\alpha)/2,2+\alpha}_b((0,T)\times H_0^-)}
+\|Mh\|_{C^{(1+\alpha)/2,2+\alpha}_b((0,T)\times H_0^+)}\le c\|h\|_{C^{(1+\alpha)/2,1+\alpha}((0,T)\times (-\ell/2,\ell/2))}.
\end{align*}
Thus, the function ${\bf f}+{\mathscr L}{\mathscr M}\boldsymbol\psi$ belongs to $C^{\alpha/2}([0,T];\boldsymbol{\mathcal X})$.
Moreover, by Proposition \ref{perdita} and the compatibility conditions on ${\bf u}_0$ it follows that
${\bf u}_0-{\mathscr M}(\boldsymbol{\psi}(0,\cdot))\in D(L)$,  ${\mathscr L}{\bf u}_0+{\bf f}(0,\cdot)\in D_L(\alpha/2,\infty)$. The theory of analytic semigroups (see e.g., \cite[Chapter 4]{lunardi}),
Theorem \ref{banca} and Proposition \ref{perdita} show that there exists a unique function ${\bf v}\in C^{1+\alpha/2}([0,T];\boldsymbol{\mathcal X})\cap C([0,T];D(L))$ which solves the equation
$D_t{\bf v}={\mathscr L}{\bf v}+{\bf f}+{\mathscr L}\mathscr M\boldsymbol{\psi}$ and satisfies the condition ${\bf v}(0,\cdot)={\bf u}_0$. In addition $D_t{\bf v}$ is bounded with values in
$\boldsymbol{\mathcal X}_{\alpha,{\mathscr B}_0}$ (which implies that $D_t{\bf v}\in \boldsymbol{\mathcal Y}_{\alpha}$) and ${\mathscr L}{\bf v}\in C^{\alpha}([0,T];\boldsymbol{\mathcal X})$.
By difference, ${\mathscr L}{\bf v}=L{\bf v}$ is bounded in $[0,T]$ with values in $\boldsymbol{\mathcal X}_{\alpha}$ and, in view of Theorem \ref{banca},
${\bf v}$ is bounded in $[0,T]$ with values in $\boldsymbol{\mathcal X}_{2+\alpha}$. In particular, ${\mathscr B}{\bf v}={\bf 0}$ and $D^{\gamma}{\bf v}(\cdot,\cdot,-\ell/2)=D^{\gamma}{\bf v}(\cdot,\cdot,\ell/2)$ for every $|\gamma|\le 2$. Further, $\|{\bf v}(t,\cdot)\|_{2+\alpha}+\|D_t{\bf v}\|_{\boldsymbol{\mathcal Y}_{\alpha}}\le
c(\|{\bf f}\|_{\boldsymbol{\mathcal Y}_{\alpha}}
+\|{\bf u_0}\|_{2+\alpha}+\|\boldsymbol{\psi}\|_{C^{(1+\alpha)/2,1+\alpha}((0,T)\times (-\ell/2,\ell/2);\R^2)})$ for every $t\in [0,T]$.

{\em Step 2.} Let $\widetilde{\bf w}$ be the function defined by
$\widetilde{\bf w}(t,\cdot,\cdot)=\int_0^te^{(t-s)L}({\mathscr M}\boldsymbol\psi(s,\cdot)-{\mathscr M}\boldsymbol\psi(0,\cdot))ds$ for every $t\in [0,T]$,
where $\{e^{tL}\}$ is the analytic semigroup generated by $L$ in $\boldsymbol{\mathcal X}$.
Taking \eqref{doccia} into account, it follows that $\|\mathscr M\boldsymbol{\psi}\|_{C^{(1+\alpha)/2}([0,T], D_L(1/2,\infty))}
\leq c\|\boldsymbol\psi\|_{C^{(1+\alpha)/2,1+\alpha}((0,T)\times (-\ell/2,\ell/2);\R^2)}$. Again by the theory of analytic semigroups we infer that
$\widetilde{\bf w}\in C([0,T];D(L))$, $D_t\widetilde{\bf w}$ is bounded in $[0,T]$ with values in $\boldsymbol{\mathcal X}_{2+\alpha,{\mathscr B}}$, $L\widetilde{\bf w}\in C^{1+\alpha/2}([0,T];\boldsymbol{\mathcal X})$,
$D_t\widetilde{\bf w}=L\widetilde{\bf w}+{\mathscr M}\boldsymbol\psi-{\mathscr M}\boldsymbol\psi(0,\cdot)$ and
$\|D_t\widetilde{\bf w}\|_{C^{1+\alpha/2}([0,T];D(L))}+\sup_{t\in [0,T]}\|LD_t\widetilde{\bf w}(t,\cdot)\|_{\boldsymbol {\mathcal X}_{\alpha,{\mathscr B}_0}}\le c\|\boldsymbol{\psi}\|_{C^{(1+\alpha)/2,1+\alpha}((0,T)\times (-\ell/2,\ell/2);\R^2)}$.
From these properties and using the same arguments as above, it can be easily checked that the function ${\bf w}=-L\widetilde{\bf w}+{\mathscr M}(\boldsymbol{\psi}(0,\cdot))$ is as smooth as ${\bf v}$ is. Moreover, $D_t{\bf w}={\mathscr L}{\bf w}-{\mathscr L}{\mathscr M}\boldsymbol\psi$, ${\mathscr B}{\bf w}=(0,0,0,\psi_1,0,0,\psi_2)$,
${\bf w}(0,\cdot,\cdot)=\mathscr M(\boldsymbol{\psi}(0,\cdot))$ and $D^{\gamma}{\bf w}(\cdot,-\ell/2)=D^{\gamma}{\bf w}(\cdot,\ell/2)$ for every $|\gamma|\le 2$.

{\em Step 3.} Clearly, the function ${\bf u}={\bf v}+{\bf w}\in \boldsymbol{\mathcal Y}_{2+\alpha}$ solves the Cauchy problem \eqref{wallace}, satisfies \eqref{argine} and it is the unique solution to the above problem in $C^1([0,T];\boldsymbol{\mathcal X})\cap C([0,T];D(L))$. Moreover,
\begin{equation}
\sup_{t\in [0,T]}\|{\bf u}(t,\cdot)\|_{2+\alpha}+\|D_t{\bf u}\|_{\boldsymbol{\mathcal Y}_{\alpha}}\le
c(\|{\bf f}\|_{\boldsymbol{\mathcal Y}_{\alpha}}
+\|{\bf u_0}\|_{2+\alpha}+\|\boldsymbol{\psi}\|_{C^{(1+\alpha)/2,1+\alpha}((0,T)\times (-\ell/2,\ell/2);\R^2)}).
\label{negramaro}
\end{equation}

To conclude that ${\bf u}\in \boldsymbol{\mathcal Y}_{2+\alpha}$ and it satisfies estimate \eqref{argine}, we use an interpolation argument. It is well known that
$\|\cdot\|_{C^2_b(H^+_R;\R^2)}\le c\|\cdot\|_{C^{\alpha}_b(H_R^+;\R^2)}^{\alpha/2}\|\cdot\|_{C^{2+\alpha}_b(H^+_R;\R^2)}^{1-\alpha/2}$. Using this estimate and the formula
\begin{eqnarray*}
\uu(t,x,y)-\uu(s,x,y)=\int_s^tD_t\uu(r,x,y)dr,\qquad\;\,s,t\in [0,T],\;\,(x,y)\in H_R^+,
\end{eqnarray*}
to show that $\|\uu(t,\cdot)-\uu(s,\cdot)\|_{C^{\alpha}_b(H_R^+;\R^2)}\le\|D_t\uu\|_{\boldsymbol{\mathcal Y}_{\alpha}}|t-s|$,
we get $\|\uu(t,\cdot)-\uu(s,\cdot)\|_{C^2_b(H^+_R;\R^2)}\le c\|D_t\uu\|_{\boldsymbol{\mathcal Y}_{\alpha}}^{\alpha/2}\sup_{r\in [0,T]}\|\uu(r,\cdot)\|_{2+\alpha}^{1-\frac{\alpha}{2}}|t-s|^{\alpha/2}$.

In the same way, we can show that $u_1\in C^{\alpha/2}([0,T];C^2_b(\overline{H_R^-};\R^2))$,
$\uu\in C^{\alpha/2}([0,T];C^2_b([0,R]\times [-\ell/2,\ell/2];\R^2))$ and
\begin{eqnarray*}
\|\uu\|_{C^{\alpha/2}([0,T];C^2_b([0,R]\times [-\ell/2,\ell/2];\R^2))}+\|u_1\|_{C^{\alpha/2}([0,T];C^2_b(\overline{H_R^-};\R^2))}
\le c\bigg (\|D_t\uu\|_{\boldsymbol{\mathcal Y}_{\alpha}}+\sup_{t\in [0,T]}\|\uu(t,\cdot)\|_{2+\alpha}\bigg ).
\end{eqnarray*}

Taking \eqref{negramaro} into account we complete this step of the proof. In particular, from all the above results it follows that
\begin{equation}
\uu(t,\cdot,\cdot)=e^{tL}\uu_0+\int_0^te^{(t-s)L}[{\bf f}(s,\cdot,\cdot)+{\mathscr L}{\mathscr M}(\boldsymbol\psi(s,\cdot))]ds
-L\int_0^te^{(t-s)L}{\mathscr M}(\boldsymbol\psi(s,\cdot))ds,\qquad\;\,t\in [0,T].
\label{repr-formula}
\end{equation}

{\em Step 4.} Let $r>0$ and $\boldsymbol{\mathcal C}_r$ be the space of all ${\bf u}\in\boldsymbol{\mathcal Y}_{2+\alpha}$ such that $D^{\gamma_1}_xD^{\gamma_2}_y{\bf u}(\cdot,\cdot,-\ell/2)=D^{\gamma_1}_xD^{\gamma_2}_y{\bf u}(\cdot,\cdot,\ell/2)$, for every $0\le\gamma_1, \gamma_2$ such that $\gamma_1+\gamma_2\le 2$,
$\|{\bf u}\|_{\boldsymbol{\mathcal Y}_{2+\alpha}}\le r$ and ${\bf u}(0,\cdot,\cdot)={\bf u}_0$.

In view of Steps 1-3, for every ${\bf u}_0\in B(0,r_0)\subset \boldsymbol{\mathcal Y}_{2+\alpha}$ satisfying the compatibility conditions in the statement of the theorem, we can define the operator $\Gamma$, which to every ${\bf u}\in \boldsymbol{\mathcal C}_r$ (with $r$ sufficiently small norm to guarantee that the nonlinear terms ${\mathscr F}({\bf u}(t,\cdot,\cdot))$ and ${\mathscr H}({\bf u}(t,\cdot,\cdot))$ are well defined for every $t\in [0,T]$) associates the unique solution ${\bf v}$ of the Cauchy problem \eqref{wallace} with ${\bf f}={\mathscr F}({\bf u})$ and $\boldsymbol{\psi}={\mathscr H}(\uu)$.
Since the maps $\uu\mapsto {\mathscr F}(\uu)$, $\uu\mapsto {\mathscr H}(\uu)$ are smooth in $\boldsymbol{\mathcal C}_r$ and quadratic at $\uu={\bf 0}$, we can estimate
\begin{eqnarray}
\begin{array}{l}
\|{\mathscr F}(\uu)\|_{\boldsymbol{\mathcal Y}_{\alpha}}\le c\|\uu\|_{\boldsymbol{\mathcal Y}_{2+\alpha}}^2,\qquad\;\, \|{\mathscr F}(\uu)-{\mathscr F}({\bf v})\|_{\boldsymbol{\mathcal Y}_{\alpha}}\le cr\|\uu-{\bf v}\|_{\boldsymbol{\mathcal Y}_{\alpha}},\\[1mm]
\|{\mathscr H}(\uu)\|_{C^{(1+\alpha)/2,1+\alpha}((0,T)\times (-\ell/2,\ell/2);\R^2)}\le c\|\uu\|_{\boldsymbol{\mathcal Y}_{2+\alpha}}^2,\\[1mm]
\|{\mathscr H}(\uu)\!-\!{\mathscr H}({\bf v})\|_{C^{(1+\alpha)/2,1+\alpha}((0,T)\times (-\ell/2,\ell/2);\R^2)}\le cr\|\uu-{\bf v}\|_{\boldsymbol{\mathcal Y}_{\alpha}}.
\end{array}
\label{vita}
\end{eqnarray}
These estimates combined with \eqref{argine} show that $r$ and $r_0$ can be determined small enough such that $\Gamma$ is a contraction in
$\boldsymbol{\mathcal C}_{\rho}$.

Uniqueness of the solution ${\bf u}$ to \eqref{asta1} follows from standard arguments, which we briefly sketch here. At first, for every $t_0\in[0,T]$, $R,\delta>0$ and ${\bf u}_1\in\boldsymbol {\mathcal X}_{2+\alpha}$, which satisfies the compatibility conditions ${\mathscr B}({\bf u}_1)={\mathscr H}({\bf u}_1)$, ${\mathscr B}_0({\mathscr L}{\bf u}_1+{\mathscr F}({\bf u}_1))=0$ and $D^{\gamma}{\bf u}_1(\cdot,-\ell/2)=D^{\gamma}{\bf u}_1(\cdot,\ell/2)$ for each multi-index $\gamma$ with length at most two, we set
\begin{align*}
\boldsymbol{\mathcal Z}_{\delta,R}^{t_0}({\bf u}_1):=\{{\bf u}\in \boldsymbol{\mathcal Y}_{2+\alpha}(t_0,t_0+\delta):{\bf u}(t_0,\cdot)={\bf u}_1, \ \|{\bf u}-{\bf u}_1\|_{\boldsymbol{\mathcal Y}_{2+\alpha}(t_0,t_0+\delta)}\leq R\}.
\end{align*}
Given $R>0$ we can determine $r_1>0$ and $\delta>0$ (independent of $t_0$) with $\delta^{\alpha/2}R$ sufficiently small such that, if ${\bf u}_1$ belongs to
$B(0,r_1)\subset\boldsymbol{\mathcal X}_{2+\alpha}$, then the Cauchy problem
\begin{align}
\left\{
\begin{array}{lll}
D_t{\bf w}={\mathscr L}{\bf w}+\mathscr F({\bf w}),\\[1mm]
\mathscr B{\bf w}=\mathscr H({\bf w}), \\[1mm]
D^{\gamma_1}_xD^{\gamma_2}_y{\bf w}(\cdot,\cdot,-\ell/2)=D^{\gamma_1}_xD^{\gamma_2}_y{\bf w}(\cdot,\cdot,\ell/2), &&\gamma_1+\gamma_2\le 2 \\ [1mm]
{\bf w}(t_0,\cdot)={\bf u}_1,
\end{array}
\right.
\label{pitenji}
\end{align}
admits a unique solution ${\bf w}\in \mathscr Z_{\delta,R}^{t_0}({\bf u}_1)$. We are almost done. Indeed, let ${\bf u}\in\boldsymbol{\mathcal C}_r$ be the unique fixed point of $\Gamma$, and take $r_0$
small enough such that ${\bf u}\in B({\boldsymbol 0},\rho_1)\subset \boldsymbol{\mathcal Y}_{2+\alpha}$. Assume that ${\bf v}\in\boldsymbol{\mathcal Y}_{2+\alpha}$
is another solution to \eqref{asta1}, and let $t_0>0$ denote the supremum of the set $\{\tau\in[0,T]:{\bf u}(t,\cdot)={\bf v}(t,\cdot), t\in[0,\tau]\}$.
Suppose by contradiction that $t_0<T$. Then, both ${\bf u}$ and ${\bf v}$ are solutions in $\boldsymbol{\mathcal Y}_{2+\alpha}(t_0,t_0+\delta)$ to the Cauchy problem \eqref{pitenji}, with ${\bf u_1}:={\bf u}(t_0,\cdot)={\bf v}(t_0,\cdot)$. Taking $R\geq 2\max\{\|{\bf u}\|_{\boldsymbol{\mathcal Y}_{2+\alpha}}, \|{\bf v}\|_{\boldsymbol{\mathcal Y}_{2+\alpha}}\}$ large enough and $\delta>0$ small enough, it follows that $\uu$ and ${\bf v}$
 both belong to $\boldsymbol{\mathcal Z}_{\delta,R}^{t_0}({\bf u}_1)$, so that they do coincide, leading us to a contradiction.
\end{proof}

\begin{rmk}
\label{remark-nonlinear}
{\rm Since Problem \eqref{asta1} is autonomous, under the same assumptions as in Theorem \ref{soliera}, for each $a>0$ and $T>0$ there exists a unique solution $\uu\in\boldsymbol{\mathcal Y}_{2+\alpha}(a,a+T)$ such that $\uu(a,\cdot)=\uu_0$.}
\end{rmk}

\section{Proof of the main result}
\label{sect-6}

\subsection{Study of the dispersion relation and the point spectrum}
Since, we are interested in the instability of the travelling wave solution $(\Theta^{(0)},\Phi^{(0)})$ to Problem \eqref{system-1}-\eqref{infty},
we need to determine a range of Lewis numbers $\Le$ which lead to eigenvalues of $L_{\alpha}$ (see Remark \ref{rmk-5.6}) with positive real part.
In view of Theorem \ref{banca}, such eigenvalues will lie in $\Omega_k$ (see \eqref{omega-k}). For simplicity, we will look for positive real elements of $\Omega_k$.
Note that
$\Le-Y_k$ vanishes for no $\lambda$'s with positive real part. To determine elements of $\sigma(L_{\alpha})$ with positive real part we need to analyze the reduced dispersion relation
\begin{eqnarray*}
{\mathcal D}_{0,k}(\lambda, \Le)=\exp\bigg (\frac{R}{2}({\rm Le}-1-X_k(\lambda)-Y_k(\lambda,\Le))\bigg )-1+\theta_iRX_k(\lambda).
\end{eqnarray*}
We recall that $X_k(\lambda)=\sqrt{1+4\lambda+4\lambda_k}$, $Y_k(\lambda,\Le)=Y_k(\lambda)=\sqrt{{\rm Le}^2+4\lambda {\rm Le}+4\lambda_k}$ and $\lambda_k=4\pi^2k^2\ell^{-2}$
for each $k\in\N\cup\{0\}$. Throughout this subsection we assume $\theta_i$ is fixed in $(0,1)$; so is $R>0$ via \eqref{theta-i-R}.

\begin{lemm}
\label{leading_mode}
There exists ${\ell}_0(\theta_i)$ such that, for all $\ell >{\ell}_0(\theta_i)$, ${\mathcal D}_{0,1}(0, \Le)=0$ has a unique root $\Le_c=\Le_c(1) \in (0,1)$.
Moreover, for each $\ell$ fixed as above, there exists a maximal integer $K \geq 1$ such that, for
$k\in\{1,\ldots,K\}$, ${\mathcal D}_{0,k}(0, \Le)=0$ has a unique root $\Le_c(k) \in(0,1)$. Finally, it holds:
\begin{align}\label{Le_decreasing}
0< \Le_c(K) \leq \ldots \leq \Le_c(2) \leq \Le_c(1).
\end{align}\end{lemm}

\begin{proof}
An easy but formal computation shows that, if $\Le_c(k)$ is a root of ${\mathcal D}_{0,k}(0,\cdot)$, then
\begin{align*}
\Le_c(k) = 1 + X_k(0) +Y_k(0, \Le_c(k)) + 2R^{-1}\ln(1-\theta_i R X_k(0)),
\end{align*}
or equivalently:
\begin{equation}\label{critic1}
\Le_c(k) = \displaystyle\frac{(1+X_k(0))[R^2 +2R \log(1-\theta_i R X_k(0))] +2|\log(1-\theta_i R X_k(0))|^2}{R^2(1+ X_k(0)) + 2R \log(1-\theta_i R X_k(0))}.
\end{equation}
However, Formula \eqref{critic1} makes sense only if $1-\theta_i R X_k(0) >0$. Hence, for each fixed $\ell>0$ there exists $K\in\N$ such that
$1-\theta_i R X_k(0) >0$ if and only if $k\le K$.

Further, $\Le_c(k)$ is required to meet the physical requirement that $0<\Le_c(k)<1$. In this respect, $\ell$ should be large enough:
\begin{align*}
\Le_c(1) = \Le_0 + \frac{16 \pi^2 \theta_i e^R (1 - \theta_i e^R)}{(2 \theta_i e^R -1) \ell^2}  + o(\ell^{-2})
\end{align*}
as $\ell\to +\infty$, where $\Le_0 = R(2e^R-R-2)^{-1}$ belongs to $(0,1)$ see \cite[Formula (43), p. 2083]{BGKS15}.
Thus, there exists $\ell_0(\theta_i)>0$ such that, if $\ell>\ell_0(\theta_i)$, then $\Le_c(1)\in (0,1)$.

Finally, it remains to prove property \eqref{Le_decreasing}, for a fixed $\ell>\ell_0(\theta_i)$ which in turn defines the integer $K \geq 1$. The latter property follows from the following estimate (see \cite[Proposition 3.1]{BGZ}):
\begin{align*}
\left (\sqrt{(\Le_c (k))^2+4\lambda_k}-\Le_c (k)\right )\frac{d\gap\Le_c(k)}{d\lambda_k} < 4\bigg(1-\frac{\theta_i}{1-\theta_i R}\bigg ).
\end{align*}
Obviously,
$ \displaystyle \frac{\theta_i}{1-\theta_i R}= \frac{1-e^{-R}}{Re^{-R}}=\frac{e^{R}-1}{R}  >1$,
which implies that $\displaystyle\frac{d\gap\Le_c(k)}{d\lambda_k}<0$.
\end{proof}

In view of Lemma \ref{leading_mode} our focus will be on the case when $\Le\in (0,\Le_c(1))$. Hereafter
we will simply denote the critical value $\Le_c(1)$ by $\Le_c$, keeping in mind that $\Le_c$ at fixed $0<\theta_i<1$ depends on $\ell > {\ell}_0(\theta_i)$.

\begin{lemm}
\label{lem-pasqua}
The function ${\mathcal D}_{0,1}$ is smooth in $[0,\sqrt{\lambda_1}]\times [0,\Le_c]$. Moreover, $\displaystyle\frac{\partial {\mathcal D}_{0,1}}{\partial \lambda}$ is positive
in $[0,\sqrt{\lambda_1}]\times [0,\Le_c]$, $\displaystyle\frac{\partial {\mathcal D}_{0,1}}{\partial \Le}$ is positive in $[0,\sqrt{\lambda_1})\times [0,\Le_c]$ and vanishes on
$\sqrt{\lambda_1}\times [0,\Le_c]$.
\end{lemm}

\begin{proof}
The proof of the positivity of $\displaystyle\frac{\partial {\mathcal D}_{0,1}}{\partial \Le}$ is straightforward and based on the observation that $Y_1(\lambda, \Le) - \rm Le - 2 \lambda >0$ for $\lambda> \sqrt{\lambda_1}$ and $Y_1(\sqrt{\lambda_1}), \Le)=\Le +2\sqrt{\lambda_1}$.
On the other hand, we observe that
\begin{align*}
\frac{\partial {\mathcal D}_{0,1}}{\partial \lambda}(\lambda, \Le) &= -R\exp\bigg (\frac{R}{2}({\rm Le}-1-X_1(\lambda)-Y_1(\lambda, \Le)) \bigg)(X_1^{-1}
+\Le\, Y_1^{-1})+2(1-e^{-R})X_1^{-1}.
\end{align*}
Since $\Le\, Y_1^{-1}<(X_1)^{-1}$ and ${\rm Le}-1-X_1(\lambda)-Y_1(\lambda, \Le)<-2$, we can estimate
\begin{align*}
\frac{\partial {\mathcal D}_{0,1}}{\partial \lambda}(\lambda, \Le) > 2(1-(R+1)e^{-R})X_1^{-1},\qquad\;\,(\lambda,\Le_c)\in [0,\sqrt{\lambda_1}]\times [0,\Le_c],
\end{align*}
and the positivity of $\displaystyle\frac{\partial {\mathcal D}_{0,1}}{\partial \lambda}$ in $[0,\sqrt{\lambda_1}]\times [0,\Le_c]$ follows immediately.
\end{proof}

We can now prove the following result.
\begin{prop}
\label{thm-spectrum}
Under the assumptions of Lemma $\ref{leading_mode}$, there exist $\lambda_*\in (0,\sqrt{\lambda_1})$ and  a decreasing,
continuously differentiable function $\widetilde{\varphi}:(0,\Le_c) \to (0,\lambda_*)$ such that ${\mathcal D}_{0,1}(\widetilde{\varphi}(\Le), \Le)=0$ for all $\Le\in (0,\Le_c)$.
\end{prop}

\begin{proof}
From Lemma \ref{lem-pasqua} it follows that $\frac{\partial {\mathcal D}_{0,1}}{\partial \Le}(0, \Le_c)>0$. We can thus apply the implicit function theorem, which shows that
there exist $\delta_1, \delta_2>0$ and a unique function $\varphi\in C^1([\Le_c-\delta_1,\Le_c+\delta_1])$ such that, if $(\lambda,\Le)\in [\Le_c-\delta_1, \Le_c+\delta_1]\times [-\delta_2,\delta_2]$ is a root of
${\mathcal D}_{0,1}$, then $\lambda=\varphi(\Le)$. In view of the previous lemma, $\varphi$ is a decreasing function.
As a byproduct, taking the restriction of $\varphi$ to $[\Le_c-\delta_1, \Le_c]$, we have constructed a (small) branch $\lambda=\varphi (\Le)$ of positive
roots of ${\mathcal D}_{0,1}(\lambda, \Le)=0$. We may reiterate the implicit function theorem and continue this branch up to a left endpoint $(\lambda_*,\Le_*) \in [0,\sqrt{\lambda_1}]\times [0,{\rm Le}_c]$. By continuity, ${\mathcal D}_{0,1}(\lambda_*,\Le_*)=0$. This maximal extension of $\varphi$ is a non-increasing, $C^1$-function $\widetilde{\varphi}:(\Le_*,\Le_c]\to [0,\lambda_*)$.

To complete the proof, we need to show that $\Le_*=0$. If $\Le_*>0$ then $\lambda_*=\sqrt{\lambda_1}$ otherwise, applying the implicit function theorem again,
we could extend $\widetilde\varphi$ in a left-neighborhood of $\Le_*$, contradicting the maximality of $\widetilde\varphi$.
On the other hand, it is not difficult to check that ${\mathcal D}_{0,1}(\sqrt{\lambda_1},\Le_*)\neq 0$ whenever $R>0$.
Indeed, using condition \eqref{theta-i-R} we can easily show that
\begin{eqnarray*}
{\mathcal D}_{0,1}(\sqrt{\lambda_1},\Le_*)=e^{-R(1+2\sqrt{\lambda_1})}-1+\theta_iR(1+2\sqrt{\lambda_1})=
e^{-R(1+2\sqrt{\lambda_1})}+2\sqrt{\lambda_1}-e^{-R}(1+2\sqrt{\lambda_1})
\end{eqnarray*}
and the function $x\mapsto f(x)=e^{-x(1+2\sqrt{\lambda_1})}+2\sqrt{\lambda_1}-e^{-x}(1+2\sqrt{\lambda_1})$ vanishes at $x=0$ and its derivative is positive in $(0,+\infty)$.
\end{proof}

\begin{figure}[ht]
\centering
\includegraphics[height=5.5cm,width=8cm]{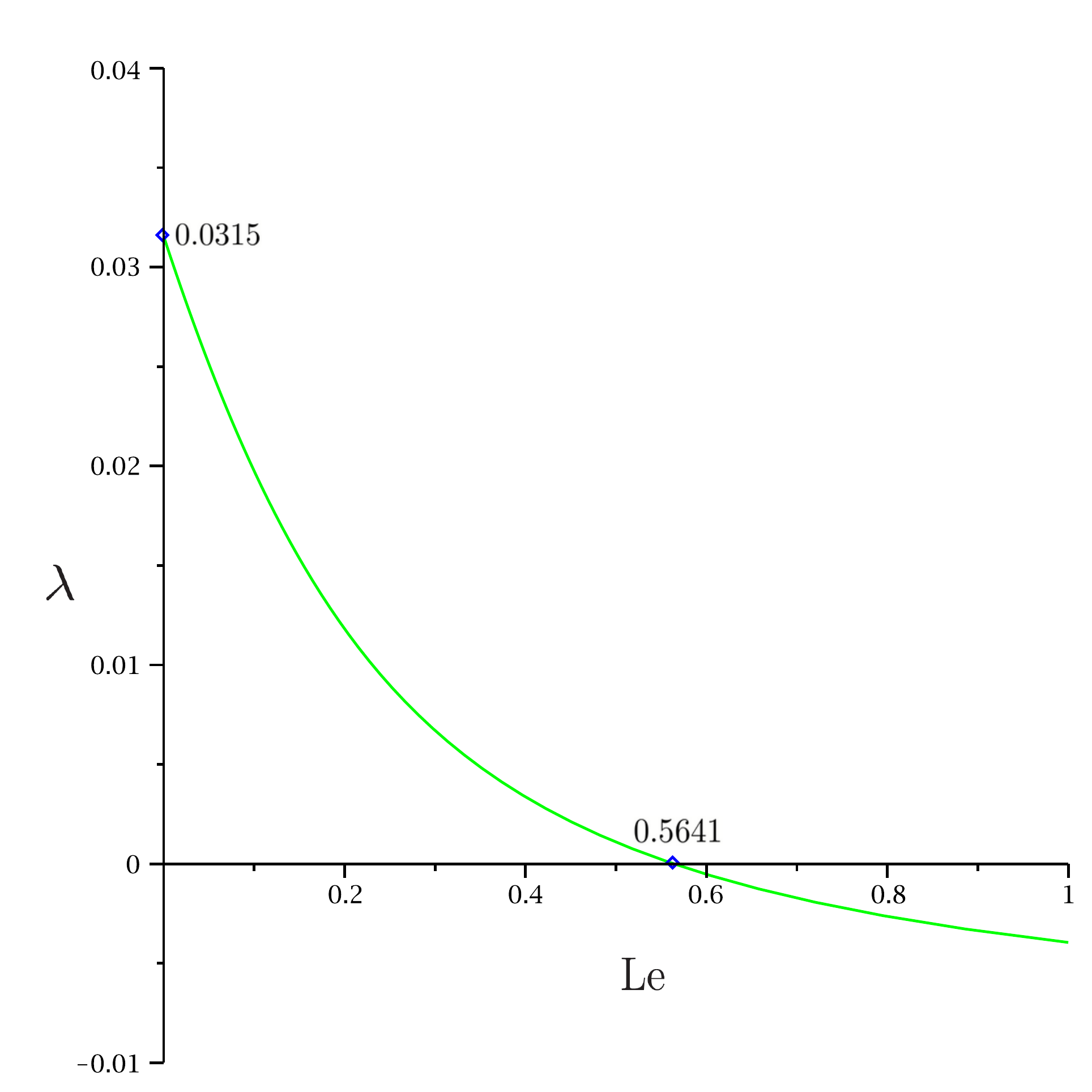}
\caption{Numerical computation of the implicit curve $\lambda = \widetilde{\varphi}(\Le)$ for $\Le \in (0,\Le_c)$, extended beyond $\Le_c$. Here $\theta_i=0.75$, $\ell=100$, $\Le_c\simeq 0.5641$,
$\lambda_* \simeq 0.0315$. Note that
$\sqrt{\lambda_1}=\pi/50 \simeq 0.0628$.} \label{Implicit_curve}
\end{figure}

\begin{coro}
\label{coro-spettro}
The spectrum of the operator $L$ contains elements with positive real parts. Moreover, the part of $\sigma(L)$ in the right halfplane $\{\lambda\in\C: {\rm Re}\lambda\ge 0\}$ consist of $0$ and
a finite number of eigenvalues.
\end{coro}

\begin{proof}
By Theorem \ref{banca} we know that if $\lambda\neq 0$ is an element in the spectrum of $L$ with nonnegative real part, then it is an eigenvalue and it belongs to $\Omega_k$ for some $k\in\N\cup\{0\}$. Hence,
$D_k(\lambda)=0$ or, equivalently, ${\mathcal D}_{0,k}(\lambda)=0$ for some $k\in\N\cup\{0\}$. As it is immediately seen, each function $\lambda\mapsto {\mathcal D}_{0,k}(\lambda)$ is
holomorphic in the halfplane $\Pi=\{\lambda\in\C:{\rm Re}\lambda\ge 0\}$ and it does not identically vanish in it. Therefore, its zeroes in $\Pi$ are at most finitely many.
Moreover, for each $\lambda\in\Pi$ and $k\in\N\cup\{0\}$ we can estimate
\begin{eqnarray*}
{\rm Re}X_k(\lambda)\ge \sqrt{\frac{1}{2}+2\lambda_k},\qquad\;\,{\rm Re}Y_k(\lambda,{\rm Le})\ge \sqrt{\frac{Le^2}{2}+2\lambda_k},
\end{eqnarray*}
so that the real part of ${\mathcal D}_{0,k}(\cdot,{\rm Le})$ diverges to $+\infty$, as $k\to +\infty$, uniformly with respect to $\lambda\in\Pi$. As a byproduct, we deduce that there exists $k_0\in\N$ such that
the nontrivial eigenvalues $\lambda\in\Pi$ lie in $\bigcup_{k=0}^{k_0}\Omega_k$ and this completes the proof.
\end{proof}

To prove the main result of this section, we also need the following result which is a variant of \cite[Theorem 5.1.5]{H80} and \cite[Theorem 4.3]{BL00}.

\begin{lemm}
Let $X$ be a complex Banach space, $r>0$ and $T_n:B(0,r)\subset X\to X$ $(n\in\N)$ be a bounded operator such that
$T_n(x)=Mx+O(\|x\|^p)$ as $\|x\|\to 0$, for some $p>1$ and some bounded linear operator $M$ on $X$ with spectral radius $\rho>1$.
Further, assume that there exists an eigenvector $u$ of $M$ with eigenvalue $\lambda\in\C$ such that $|\lambda|^p>\rho$ and that there exists $x'\in X'$ such that $x'(u)\neq 0$. Then, there exist $c>0$ and, for any $\delta>0$, $x_0\in B(0,\delta)$ and $n_0\in\N$ $($depending on $\delta)$ such that the sequence $x_0,\ldots,x_{n_0}$, where $x_n=T_n(x_{n-1})$ for any $n=1,\ldots,n_0$, is well defined and $|x'(x_{n_0})|\ge c|x'(u)|$.
\label{pisapia}
\end{lemm}

\begin{proof}
Without loss of generality, we assume that $\|u\|= 1$ and $\|x'\|\leq 1$. Moreover, we choose $a,b>0$ such that
$\|T_n(x)-Mx\|\leq b\|x\|^p$ for each $x\in B(0,a)\subset X$ and $n\in\N$.
Since $|\lambda|^p>\rho$, we can fix $\eta>0$ such that $|\lambda|^p>\rho+\eta$ and, from the definition of the spectral radius of a bounded operator, we can also determine a positive constant $K$ such that $\|M^n\|_{L(X)}\leq K(\rho+\eta)^n$ for any $n\in\N$. Finally, we fix $\delta>0$, choose $n_0\in\N$ be such that $|\lambda|^{-n_0}<\delta$, and set
$x_0:=\sigma u |\lambda|^{-n_0}$, where $\sigma\in (0,1)$ is chosen so as to satisfy the conditions
\begin{align}
\sigma\leq \frac{a}{2}, \qquad\;\, \frac{2^p b K}{|\lambda|^p-\rho-\eta}\sigma^{p-1}\leq \frac{1}{2}|x'(u)|.
\label{brasile-mex}
\end{align}

To begin with, we prove that the sequence $x_0,\ldots,x_{n_0}$ is well defined. For this purpose, in view of the condition in \eqref{brasile-mex} it suffices to check that, if $x_k$ is well defined, then  $\|x_k\|\le 2 \sigma |\lambda|^{k-n_0}$. We prove by recurrence. Clearly, $x_0$ satisfies this property. Suppose that the claim is true for $k=0,\ldots,n-1$. Then, $x_n$ is well defined and it easy to check that
\begin{equation}
x_n=M^nx_0+\sum_{k=0}^{n-1}M^{n-1-k}(x_{k+1}-Mx_k)=M^nx_0+\sum_{k=0}^{n-1}M^{n-1-k}(T_{k+1}(x_k)-Mx_k).
\label{form-1}
\end{equation}
Thus, we can estimate
\begin{align}
\label{riverdale}
\|x_n\|\le |\lambda|^n\|x_0\|+Kb\sum_{k=0}^{n-1}(\rho+\eta)^{n-1-k}\|x_k\|^p.
\end{align}
Let us consider the second term in the right-hand side of \eqref{riverdale}, which we denote by $S_n$. Since, we are assuming that $\|x_k\|\le 2 \sigma |\lambda|^{k-n_0}$ for each $k=0,\ldots,n-1$,
we can write
\begin{align*}
S_n\le & 2^pKb\sigma^p |\lambda|^{p(n-n_0-1)}\sum_{k=0}^{n-1}\bigg (\frac{\rho+\eta}{|\lambda|^p}\bigg )^{n-1-k}\le \sigma |\lambda|^{n-n_0}\frac{2^pKb}{|\lambda|^p-\rho-\eta}\sigma^{p-1}
\end{align*}
and, using the second condition in \eqref{brasile-mex} and the fact that $x'$ has norm which does not exceed $1$, we conclude that
\begin{align}
\label{suits}
S_n \le \frac{1}{2}\sigma|\lambda|^{n-n_0}|x'(u)|\leq \frac{1}{2}\sigma |\lambda|^{n-n_0}.
\end{align}
Since $|\lambda|^n\|x_0\|\leq \sigma|\lambda|^{n-n_0}$, from \eqref{riverdale} and \eqref{suits} the claim follows at once.

To conclude the proof, it suffices to use \eqref{form-1} with $n=n_0$, as well as \eqref{riverdale} and \eqref{suits} again, to estimate
\begin{align*}
|x'(x_{n_0})|
\geq & |x'(M^{n_0}x_0)|-|x'(S_{n_0})|\geq  \sigma |x'(u)|-\frac{1}{2}\sigma|x'(u)|=\frac{1}{2}\sigma|x'(u)|.
\end{align*}
The assertion follows with $c=\sigma/2$.
\end{proof}

Now, we can state and prove the following theorem.

\begin{thm}
\label{thm-main-1}
Let $0<\theta_i<1$ be fixed, $\ell>\ell_0(\theta_i)$ as in Lemma $\ref{leading_mode}$, $\Le_c= \Le_c(1)$ defined by \eqref{critic1}.
Then, for each $\Le\in (0,\Le_c)$, the null solution $\uu$ of Problem \eqref{asta1} is poinwise unstable with respect to small perturbations in $\boldsymbol{\mathcal X}_{2+\alpha}$.
More precisely, there exists a positive constant $C$ such that for each $y_0\in\R$ and $\delta>0$ there exist $ \widetilde {\bf u}_0, {\bf u}_0^*\in B(0,\delta)\subset\boldsymbol{\mathcal X}_{2+\alpha}$ and $\widetilde n, n_*\in\N$ depending on $\delta$ such that $\min\{|\widetilde u_2(\widetilde n,0,y_0)|, |u_1^*(n_*,R,y_0)|\}\ge C$, where
$\widetilde{\bf u}=(\widetilde u_1,\widetilde u_2)$ and ${\bf u}_*=(u_1^*,u_2^*)$ denote the solution to the Cauchy problem with initial datum $\widetilde {\bf u}_0$ and ${\bf u}_0^*$, respectively.
\end{thm}

\begin{proof}
We split the proof into two steps. The first one is devoted to prove an estimate which will allow us to apply Lemma \ref{pisapia}. Then, in Step 2, we prove the pointwise instability.

\vspace{1mm}
{\em Step 1}.
The smoothness of $\Upsilon$ implies that there exists $c>0$ such that $\|\Upsilon({\bf v}_0)\|_{2+\alpha}\leq c\|{\bf v}_0\|_{2+\alpha}$, for each ${\bf v}_0\in D(L_{\alpha})$ with sufficiently small norm. Here, $\Upsilon$ is defined in Lemma \ref{lemma-5.7}$(ii)$. Fix $r$ so small such that $\|{\bf v}_0+\Upsilon({\bf v}_0)\|_{2+\alpha}\le r_0$ for each ${\bf v}_0\in \overline{B(0,r)}\subset D(L_\alpha)$, where
$r_0=r_0(1)$ is defined in the statement of Theorem \ref{soliera}.

For $n\in\N$, let $R_n:B(0,\rho)\subset \widetilde D(L_\alpha)\to \widetilde D(L_\alpha)$ be the map defined by $R_n({\bf v}_0)=P({\bf u}_n(n,\cdot,\cdot))$ for each ${\bf v_0}\in B(0,\rho)$ (see Lemma \ref{lemma-5.7}), where ${\bf u}_n$ is the solution to problem \eqref{asta1} with initial condition ${\bf  u}_0={\bf v}_0+\Upsilon({\bf v}_0)$ at $t=n-1$. (Note that, by Lemma \ref{lemma-5.7}$(ii)$, ${\bf u}_0$ satisfies the compatibility conditions in Theorem \ref{soliera}. Further, by Remark \ref{remark-nonlinear}, ${\bf u}_n$ is well defined in the time domain $[n-1,n]$.)
We claim that
\begin{align}
\label{treno}
\|R_n({\bf v_0})-e^{L_\alpha}{\bf v}_0\|_{2+\alpha}\leq c\|{\bf v}_0\|^2_{2+\alpha},\qquad\;\,{\bf v}_0\in B(0,r).
\end{align}
Estimate \eqref{treno} follows from the integral representation of the solution of Problem \eqref{asta1} and estimates \eqref{vita}. Indeed, again by
Remark \ref{remark-nonlinear}, ${\bf u}_n(n,\cdot,\cdot)$ is the value at $t=1$ of the solution $\uu$ to Problem \eqref{asta1}, with $\uu(0,\cdot)=\uu_0$
and, by the proof of Theorem \ref{soliera} (see, in particular, formula \eqref{repr-formula}),
\begin{align*}
\uu(1,\cdot,\cdot)-e^L\uu_0=\int_0^1e^{(1-s)L}[{\mathscr F}(\uu(s,\cdot,\cdot))\!+\!{\mathscr L}{\mathscr M}({\mathscr H}(\uu(s,\cdot,\cdot))]ds\!-\!
L\int_0^1e^{(1-s)L}{\mathscr M}({\mathscr H}(\uu(s,\cdot,\cdot))ds.
\end{align*}
Since ${\mathscr F}$ and ${\mathscr H}$ are quadratic at $0$, it follows immediately that
$\|{\bf u}(1,\cdot,\cdot)-e^L{\bf u}_0\|_{2+\alpha}\leq c\|{\bf v}_0\|^2 _{2+\alpha}$. Noting that $P({\bf u}(1,\cdot,\cdot)-e^{L_{\alpha}}{\bf u}_0)=
R_n({\bf v}_0)-e^{L_{\alpha}}{\bf v}_0$, formula \eqref{treno} follows at once.

{\em Step 2.} Let us begin by proving that there exists $C>0$ such that for any $y_0\in\R$ and $\delta>0$ there exists ${\bf u}_0\in B(0,\delta)\subset \boldsymbol{\mathcal X}_{2+\alpha}$ and $n_0\in\N$ depending on $\delta$ such that $|u_2(n_0,0^+,y_0)|\geq C$, where ${\bf u}=(u_1,u_2)$ is the solution to \eqref{asta1} with initial datum ${\bf u}_0$ at time $t=0$. For this purpose, we want to apply Lemma \ref{pisapia} with $X=D(L_{\alpha})$ endowed with the norm of $\boldsymbol{\mathcal X}_{2+\alpha}$.
To begin with, we observe that, by Corollary \ref{coro-spettro}, there exists only a finite number of eigenvalues of $L$ (and hence of $L_{\alpha}$) with positive real part.
From the spectral mapping theorem for analytic semigroups it thus follows that the spectral radius $\rho$
of the operator $M=e^{L_{\alpha}}$ is larger than one and there exists an eigenvalue $\lambda$ such that $|\lambda|=\rho$. Let us fix $y_0\in\R$ and $\delta>0$. It is not difficult to show that a corresponding eigenfunction is the function
${\bf w}=(w_1 e_1(\cdot - 2\pi\ell^{-1}y_0)), w_2 e_1(\cdot - 2\pi\ell^{-1}y_0)))$, where
\begin{align*}
&w_1(x)=e^{\nu_1^+x}\chi_{(-\infty,0]}(x)+(c_1e^{\nu_1^- x}+c_2e^{\nu_1^+x})\chi_{(0,R)}(x)+c_3e^{\nu_1^-x}\chi_{[R,+\infty)}(x),\\
&w_2(x)=(d_1e^{\mu_1^- x}+d_2e^{\mu_1^+ x})\chi_{[0,R)}(x)+d_3 e^{\mu_1^-x}\chi_{[R,+\infty)}(x),
\end{align*}
for every $x\in\R$ and
\begin{eqnarray*}
\begin{array}{llll}
&\displaystyle c_1=\frac{e^{(X_1+\mu_1^+)R}(\theta_iRX_1-1)}{e^{\mu_1^+R}-e^{\nu_1^+R}},\qquad\;\,
&\displaystyle c_2=\frac{e^{\mu_1^+R}}{e^{\mu_1^+R}-e^{\nu_1^+R}},\qquad\;\, &\displaystyle c_3 = \frac{\theta_iRe^{(X_1+\mu_1^+)R}X_1}{e^{\mu_1^+R}-e^{\nu_1^+R}}\vspace{2mm} \\
&\displaystyle d_1=-\frac{{\rm Le}({\rm Le}+\mu_1^+)e^{\nu_1^+R}X_1}{(e^{\mu_1^+R}-e^{\nu_1^+R})Y_1},\qquad\;\, &\displaystyle d_2 = -(\Le+\mu_1^-)d_1,\qquad\;\,
&\displaystyle d_3 = (1-e^{Y_1R})d_1, \vspace{2mm} \\
& \displaystyle \nu_1^{\pm}=-\frac12\pm\sqrt{1+4\tilde\lambda+4\lambda_1},
& \displaystyle \mu_1^{\pm}=-\frac{\rm le}2\pm\sqrt{{\rm Le}^2+4{\rm Le}\tilde\lambda+4\lambda_1}.
\end{array}
\end{eqnarray*}
Note that
\begin{eqnarray*}
w_2(0^+,y_0)=(d_1+d_2)=\frac{{\rm Le}X_1e^{\nu_1^+R}X_1}{(e^{\mu_1^+R}-e^{\nu_1^+R})Y_1}\neq 0.
\end{eqnarray*}
Hence, if we set $x'({\bf f})=f_2(0^+,y_0)$ for any ${\bf f}\in D(L_\alpha)$, then $|x'({\bf w})|\neq 0$. As in Step 1, we fix $r>0$ such that
$\|{\bf v}_{0,j}+\Upsilon({\bf v}_{0,j})\|_{{2+\alpha}}\le r_0$ for $j=1,2$ for each ${\bf v}_0={\bf v}_{0,1}+i{\bf v}_{0,2}\in B(0,r)\subset D(L_{\alpha})$.
By Theorem \ref{soliera} both ${\bf u}(n,{\bf v}_{0,1}+\Upsilon({\bf v}_{0,1}),n-1)$ and ${\bf u}(n,{\bf v}_{0,2}+\Upsilon({\bf v}_{0,2}),n-1)$ are well defined for any $n\in\N$. We can thus introduce the operator
$T_n:B(0,r)\subset D(L_{\alpha})\to D(L_{\alpha})$ ($n\in\N$) by setting $T_n({\bf v}_0)=P{\bf u}(n,{\bf v}_{0,1}+\Upsilon({\bf v}_{0,1}),n-1)+iP{\bf u}(n,{\bf v}_{0,2}+\Upsilon({\bf v}_{0,2}),n-1)$ for
each ${\bf v}_0\in B(0,r)$, where $P$ is the projection in Lemma \ref{lemma-5.7}(i). By the arguments in Step 1 we deduce that
$\|T_n({\bf v}_0)-e^{L_\alpha}{\bf v}_0\|_X\leq C\|{\bf  v}_0\|_X^2$ for some positive constant $C$ and each ${\bf v}_0\in B(0, r)$.
We can thus apply Lemma \ref{pisapia} with $M=e^{L_\alpha}$, $p=2$ and conclude that there exist $c>0$ and, for each $\delta>0$, a function ${\bf v}_0={\bf v}_{0,1}+i{\bf v}_{0,2}\in B(0,\delta)\subset D(L_{\alpha})$ and $n_0\in\N$ such that ${\bf v}_n=T_n({\bf v}_{n-1})$ is well defined for each $n\in\{1,\ldots,n_0\}$ and $|x'({\bf v}_{n_0})|\ge c$.
Since ${\bf v}_{n_0}=P{\bf u}(n_0,{\bf v}_{0,1}+\Upsilon({\bf v}_{0,1}),0)+iP{\bf u}(n_0,{\bf v}_{0,2}+\Upsilon({\bf v}_{0,2}),0)$, where ${\bf u}(n_0,{\bf u}_0,0)$ denotes the value at $n_0$ of the unique solution to problem \eqref{asta} with initial datum ${\bf u}_0$
at time $t=0$, we have so proved that
\begin{align}
|(P{\bf u}(n_0,{\bf v}_{0,1}+\Upsilon({\bf v}_{0,1}),0))_2(0^+,y_0)|^2+|(P{\bf u}(n_0,{\bf v}_{0,2}+\Upsilon({\bf v}_{0,2}),0))_2(0^+,y_0)|^2\ge c^2.
\label{russia}
\end{align}
By definition, $P=I-{\mathscr N}{\mathcal B}_*$ (see Lemma \ref{lemma-5.7}(i)) and $({\mathscr N}{\bf u})_2(0^+,\cdot)=0$ for any function ${\bf u}$. Hence,
$(P{\bf u}(n_0,{\bf v}_{0,j}+\Upsilon({\bf v}_{0,j}),0))_2(0^+,y_0)=({\bf u}(n_0,{\bf v}_{0,j}+\Upsilon({\bf v}_{0,j}),0))_2(0^+,y_0)$ for $j=1,2$.
From \eqref{russia} it thus follows that there exists $\bar j$ such that
$|({\bf u}(n_0,{\bf v}_{0,\bar j}+\Upsilon({\bf v}_{0,\bar j}),0))_2(0^+,y_0)|\ge c/2$ and the thesis follows with $C=c/2$, $\widetilde n=n_0$ and $\widetilde {\bf u}={\bf u}(\widetilde n,{\bf v}_{0,\bar j}+\Upsilon({\bf v}_{0,\bar j}),0)$.

To prove the existence of ${\bf u}_*$ as in the statement of the theorem, it suffices to take as $x'$ the functional defined by $x'({\bf f})=f_1(R,y_0)$ for each ${\bf f}\in D(L_\alpha)$.
The missing easy details are left to the reader.
\end{proof}

\begin{rmk}
{\rm Clearly, Theorem \ref{thm-main-1} implies the $C^{2+\alpha}$-instability of the null solution $\uu$ to \eqref{asta1}.}
\end{rmk}
\begin{rmk}
{\rm The $C^{2+\alpha}$-instability of the null solution $\uu$ to \eqref{asta1} can be directly obtained by applying \cite[Theorem 5.1.5]{H80}, taking advantage of Step $1$ of Theorem \ref{thm-main-1} and arguing as in \cite[Corollary 4.5]{BL00}. Finally, it can also be proved in a slightly different way adapting the arguments in \cite[Theorem 3.4]{lorenzi-3}.}
\end{rmk}

From Theorem \ref{thm-main-1} we can now easily derive the proof of the main result of this paper.

\begin{proof}[Proof of Main Theorem]
Taking the changes of variables and unknown in Subsections \ref{fix-domain} and \ref{sect-3.2} into account, the result in Theorem \ref{thm-main-1} allows us to conclude easily that the normalized temperature $\Theta$ and the normalized concentration of deficient reactant in problem \eqref{system-1}-\eqref{infty} are unstable with respect to two dimensional $C^{2+\alpha}$ perturbations.
Similarly, using formulae \eqref{elim-f} and \eqref{elim-g} and again Theorem \ref{thm-main-1}, we can infer that there exist initial data $(\widetilde \Theta,\widetilde \Phi)$ and $(\Theta_*,\Phi_*)$ with $C^{2+\alpha}$-norm, arbitrarily close to the travelling wave solution \eqref{planar-front} such that the trailing front $G$ (resp. the ignition front $F$) to problem \eqref{system-1}-\eqref{infty} with initial datum $(\Theta(0,\cdot),\Phi(0,\cdot))=(\widetilde \Theta,\widetilde \Phi)$ (resp. $(\Theta(0,\cdot),\Phi(0,\cdot))=(\Theta_*,\Phi_*)$) is not arbitrarily close to $0$ (resp. $R$).
\end{proof}

\section{Numerical simulation}
\label{sect-7}
In this section, we are going to use some high resolution numerical methods, including Chebyshev collocation and Fourier spectral method (see, e.g., \cite{STW, WWQ, WZW}). We consider the problem \eqref{asta} in the finite domain $\Omega = \Omega_- \cup \gap\Omega_0 \cup \gap\Omega_+ = ([-A,0] \cup [0,R] \cup [R, B]) \times [-\ell/2,\ell/2]$, where $A>0$ and $B>0$ are large enough, see Figure \ref{domain}. The independent variables are $-A\le\xi\le B$, $-\ell/2 < \eta < \ell/2$.

\begin{figure}[h]
\centering
\includegraphics[height=3cm,width=7cm]{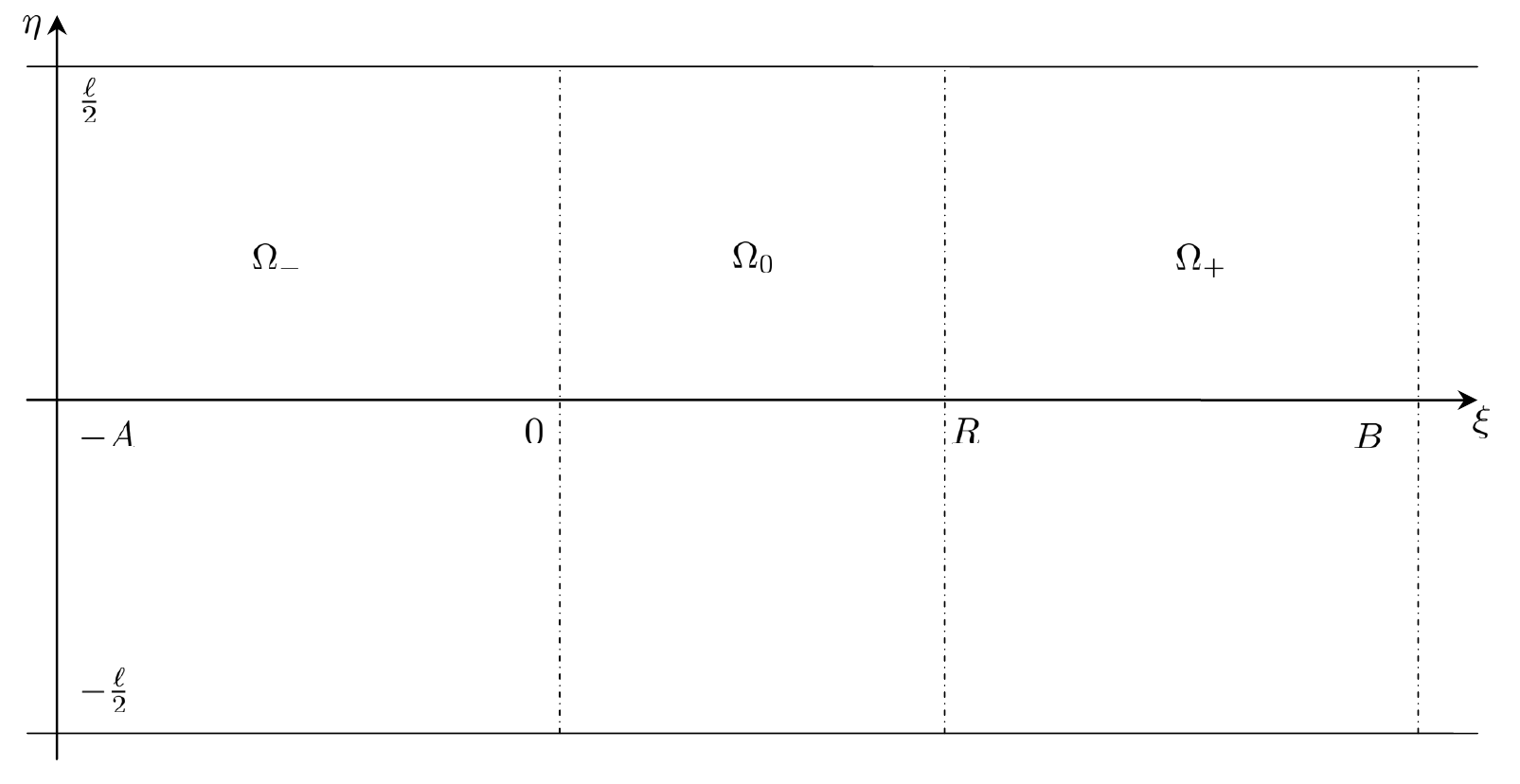}		
\caption{Computational domain.}
\label{domain}
\end{figure}

\subsection{The linear system}
The linearized system around the null solution of System \eqref{asta} reads:
\begin{align}
\left\{
\begin{array}{ll}
u_\tau = u_\xi + u_{\xi\xi} + u_{\eta\eta}, & \mbox{in} \quad \Omega,\\[1mm]
w=0, & \mbox{in} \quad \Omega_-, \\[1mm]
w_\tau =  w_\xi + \Le^{-1}(w_{\xi\xi} + w_{\eta\eta}), &\mbox{in} \quad  \Omega_0 \cup \Omega_+,\\[1mm]
\end{array}
\right.
\label{linear}
\end{align}
with
\begin{align}
{\mathscr B}(u,w)=0.
\label{IC-0R}
\end{align}

We map $\Omega_-, \Omega_0$ and $\Omega_+$ to $\mathbb{D} = [-1,1] \times [0, 2 \pi]$. Then, we consider in $\mathbb{D}$ the system for the three pairs of unknowns $(u_1, w_1)$, $(u_2, w_2)$ and $(u_3, w_3)$, corresponding respectively to $(u,w)$ in
$\Omega_-, \Omega_0$ and $\Omega_+$. The new independent variables are denoted by $x\in [-1,1]$ and $y \in [0, 2 \pi]$.

Therefore, System  \eqref{linear}-\eqref{IC-0R} is equivalent to:
\begin{align}
\left\{
\begin{array}{ll}
D_\tau u_{1} = \frac{2}{A} D_x u_{1} + \frac{4}{A^2} D_{xx} u_{1} + \frac{4\pi^2}{\ell^2} D_{yy} u_{1}, \; w_1 \equiv 0, & \\[1mm]
D_\tau u_{2} = \frac{2}{R} D_x u_{2} + \frac{4}{R^2} D_{xx} u_{2} + 4\frac{\pi^2}{\ell^2} D_{yy} u_{2},  	& \\[1mm]
D_\tau w_{2} = \frac{2}{R} D_x w_{2} + \frac{4}{\Le R^2} D_{xx} w_{2}+ \frac{4 \pi^2}{\Le \gap\ell^2} D_{yy} w_{2}, & \\[1mm]
D_\tau u_{3} = \frac{2}{B-R} D_x u_{3} + \frac{4}{(B-R)^2} D_{xx} u_{3} + \frac{4\pi^2}{\ell^2} D_{yy} u_{3}, & \\[1mm]
D_\tau w_{3} = \frac{2}{B-R} D_x w_{3} + \frac{4}{\Le\,(B-R)^2} D_{xx} w_{3}+ \frac{4 \pi^2}{\Le\,\ell^2} D_{yy} w_{3} ,&
\end{array}
\right. \label{systemin3domain}
\end{align}
together with the boundary conditions:
\begin{align}
\left\{
\begin{array}{ll}
u_1(-1) = u_3(1) = w_3(1) = 0, \quad u_1(1) = u_2(-1), \\[1mm]
w_2(-1) = \frac{2\Le}{A} D_x u_{1}(1) -\frac{2\Le}{R} D_x u_{2}(-1), \quad D_x w_{2}(-1)= -\frac{\Le R}{2} w_2(-1),\\[1mm]
D_x w_{2}(1) = \frac{R}{B-R} D_x w_{3}(-1)+\frac{\Le R}{2} ( w_3(-1)- w_2(1) ), \\[1mm]
D_x u_{3}(-1) = \frac{B-R}{R} D_x u_{2}(1)-\frac{B-R}{2\Le} ( w_3(-1)-w_2(1)), \\[1mm]
u_2(1) = -\frac{2\theta_i R}{B-R} D_x u_{3}(-1) + 2 \theta_i D_x u_{2}(1) , \quad u_3(-1) = u_2(1).
\end{array}
\right. \label{systemin3domain-IC}
\end{align}

Let us give a brief overview of the numerical method. Hereafter, we denote by $(u,w)$ any pair of unknowns $(u_i, w_i)$,
$1\leq i \leq 3$. We discretize System \eqref{systemin3domain}-\eqref{systemin3domain-IC} using a forward-Euler explicit scheme in time. Then, we use a discrete Fourier transform in the direction $y \in (0, 2 \pi)$, namely:
\begin{align*}
	u(x,y) = \sum_{k=-N_y/2}^{N_y/2} \hat{u}_k(x) e^{ik y},  \qquad\;\, w(x,y) = \sum_{k=-N_y/2}^{N_y/2} \hat{w}_k(x) e^{ik y},
\end{align*}
 and
\begin{align*}
D_{yy}u(x,y) = -\sum_{k=-N_y/2}^{N_y/2} k^2 \hat{u}_k(x) e^{ik y},  \qquad\;\, D_{yy}w(x,y) = - \sum_{k=-N_y/2}^{N_y/2} k^2 \hat{w}_k(x) e^{ik y}.
\end{align*}
Finally,  we use a Chebyshev collocation method in $x \in (-1, 1)$. Let $\{ l_j(x)\}_{j=0}^{N_x}$ be the Lagrange polynomials based on the Chebyshev-Gauss-Lobatto points $\{ x_j\}_{j=0}^{N_x} = \{ \cos(j\pi/N_x)\}_{j=0}^{N_x}$. We set:
\begin{align*}
\hat{u}_k(x) = \sum_{j=0}^{N_x}\hat{u}_{kj} l_j(x),  \qquad\;\, \hat{w}_k(x) = \sum_{j=0}^{N_x}\hat{w}_{kj} l_j(x).
\end{align*}
Denoting the differential matrix of order $m$ associated to $\{ x_j\}_{j=0}^{N_x} $ by $D^m = (d_{ij}^{(m)})_{i,j = 0,\cdots,N_x}$, where $d_{ij}^{(m)} = l_j^{(m)}(x_i)$ and $l_j(x_i) = \delta_{ij}$, we eventually obtain
\begin{align*}
& \hat{u}_k(x_i) = \sum_{j=0}^{N_x}\hat{u}_{kj} \delta_{ij}, \qquad\;\, D_x \hat{u}_k(x_i) = \sum_{j=0}^{N_x}\hat{u}_{kj} d_{ij}^{(1)}, \qquad\;\, D_{xx} \hat{u}_k(x_i) = \sum_{j=0}^{N_x}\hat{u}_{kj} d_{ij}^{(2)},\\
& \hat{w}_k(x_i) = \sum_{j=0}^{N_x}\hat{w}_{kj} \delta_{ij},  \qquad\;\, D_x\hat{w}_k(x_i) = \sum_{j=0}^{N_x}\hat{w}_{kj} d_{ij}^{(1)},  \qquad\;\, D_{xx}\hat{w}_k(x_i) = \sum_{j=0}^{N_x}\hat{w}_{kj} d_{ij}^{(2)}.
\end{align*}

As initial data, we take $w_2(0,\cdot)= \varepsilon \big(1+\sin^2(y)\big)$, $y \in [0, 2 \pi] $, which corresponds to $\xi = {R}/{2}$;  the other unknowns are taken as $0$. The following pictures are for $\varepsilon = 10^{-2}, A=B=10, \ell =100, \Delta t = 10^{-3}$. As expected, the two profiles blow up for Lewis number below critical.
\begin{figure}[h]
	\subfloat[Evolution of $u(t,0,y)$]{
		\begin{minipage}[t]{0.45\textwidth}
			\centering
			\includegraphics[height=4cm,width=5.3cm]{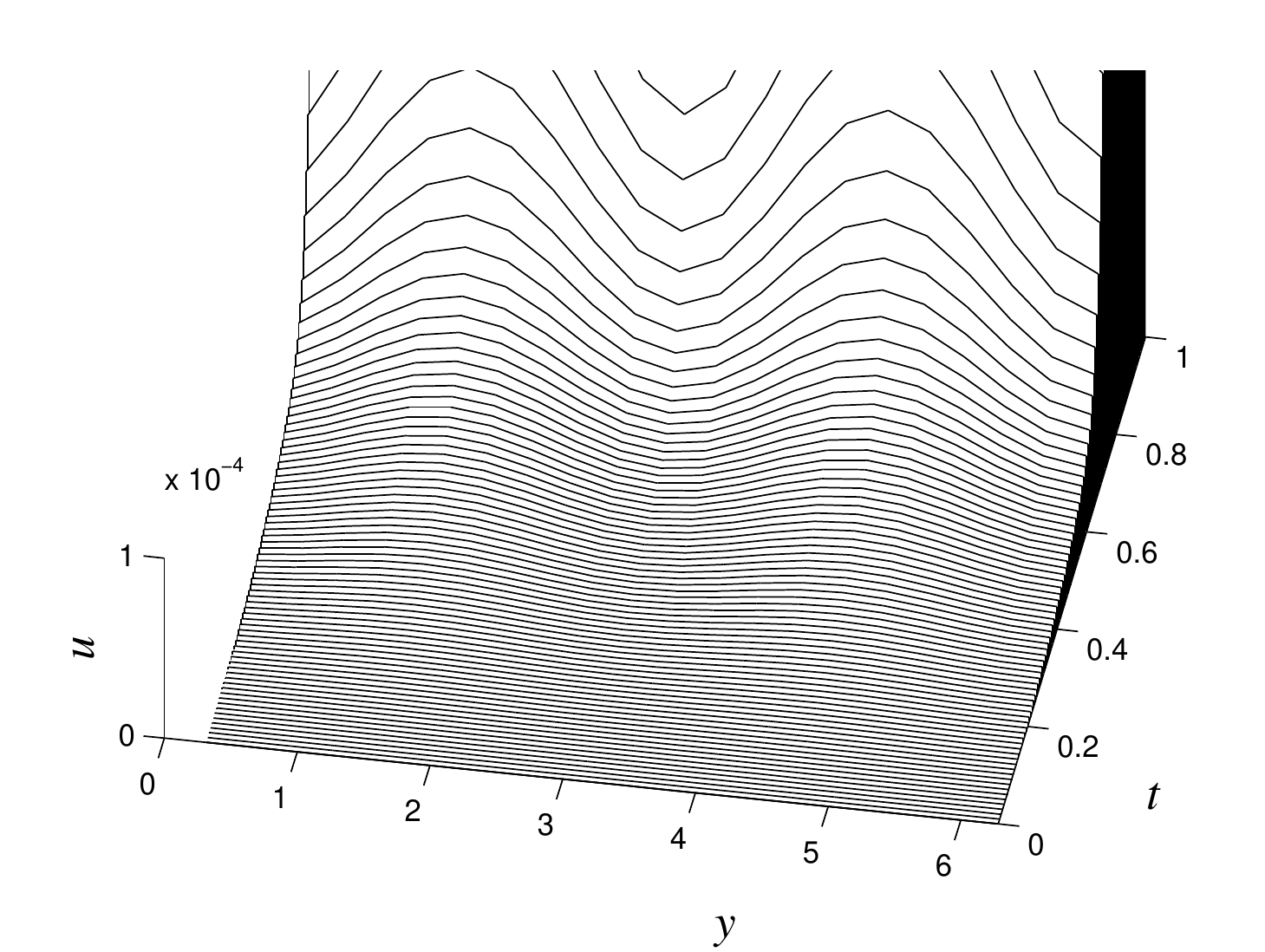}
		\end{minipage}}
	\subfloat[Evolution of $w(t,0,y)$] {
		\begin{minipage}[t]{0.45\textwidth}
			\centering
			\includegraphics[height=4cm,width=5.3cm]{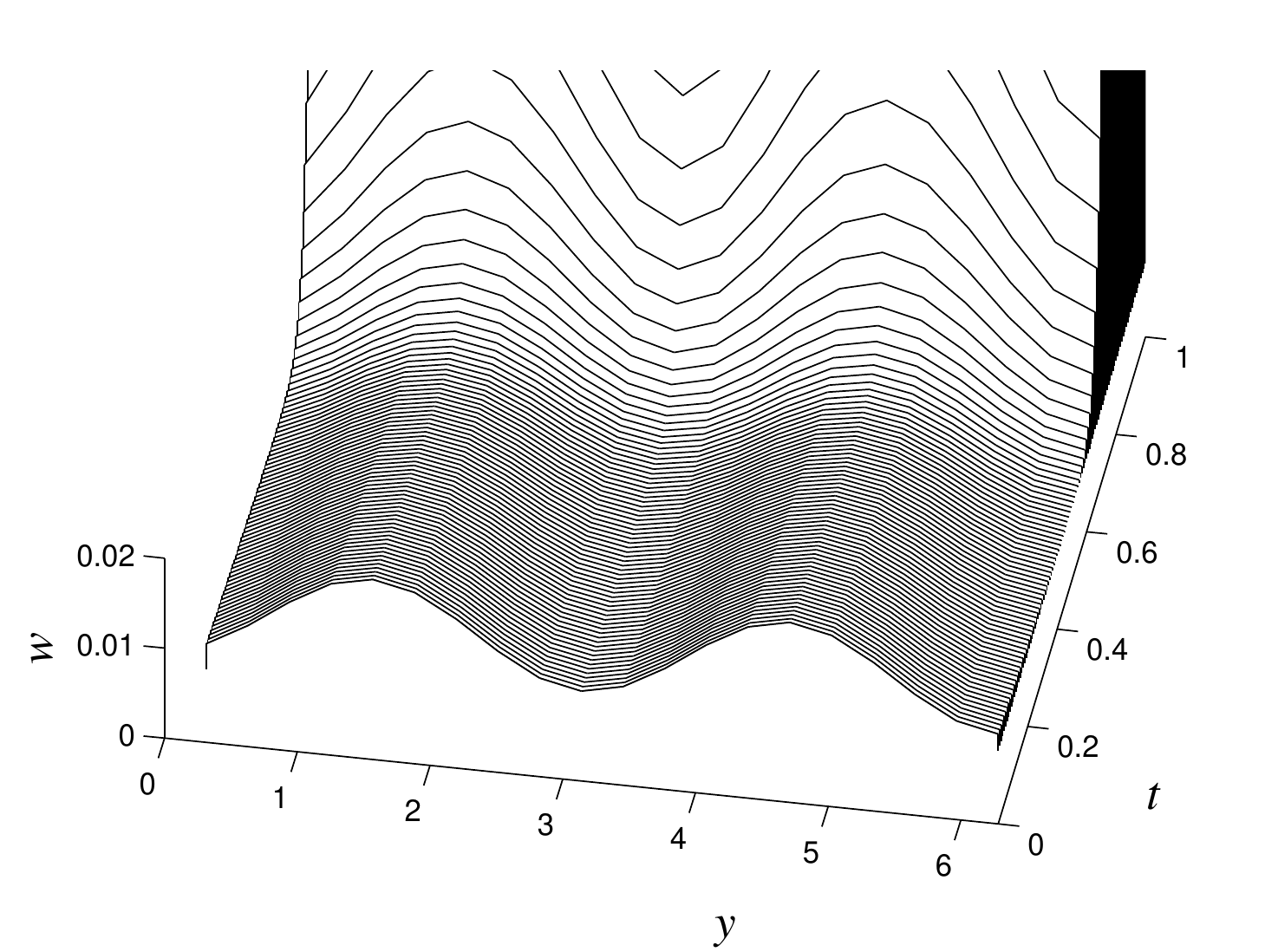}
		\end{minipage}}
	\caption{Evolution of ($u,w$) solution of the linear problem \eqref{linear}-\eqref{IC-0R} for $\theta_i = 0.75, \Le = 0.3 < \Le_c \simeq 0.56$.
		(A) $u(t,0,y)$ varies from $0$ to $10^{-4}$, $0<t<1, 0<y< 2\pi$,
		(B) $w(t,0,y)$ varies from $0$ to $2 . 10^{-2}$, $0<t<1, 0<y< 2\pi$.	 }
	\label{linear-unstable}
\end{figure}

\subsection{The fully nonlinear system}
By treating the nonlinearities explicitly, we can use the same algorithm as in the linear case. In the coordinates $(\xi,\eta)$, we approximate the mollifier $\beta(\xi)$ by the following trapezoid, see Figure \ref{fig-mollifiers}:
\begin{align*}
\beta(\xi) = \left\{
\begin{array}{ll}
2 + {\xi / \delta}, \quad & -2 \delta <\xi< -\delta, \\[1mm]
1, \quad & -\delta \le \xi \le \delta, \\[1mm]
2 - {\xi / \delta}, \quad  &\delta <\xi< 2 \delta, \\[1mm]
0, \quad &\text{elsewhere,}
\end{array}\right.
\end{align*}
\begin{figure}[h]
\centering
\includegraphics[height=3cm,width=10cm]{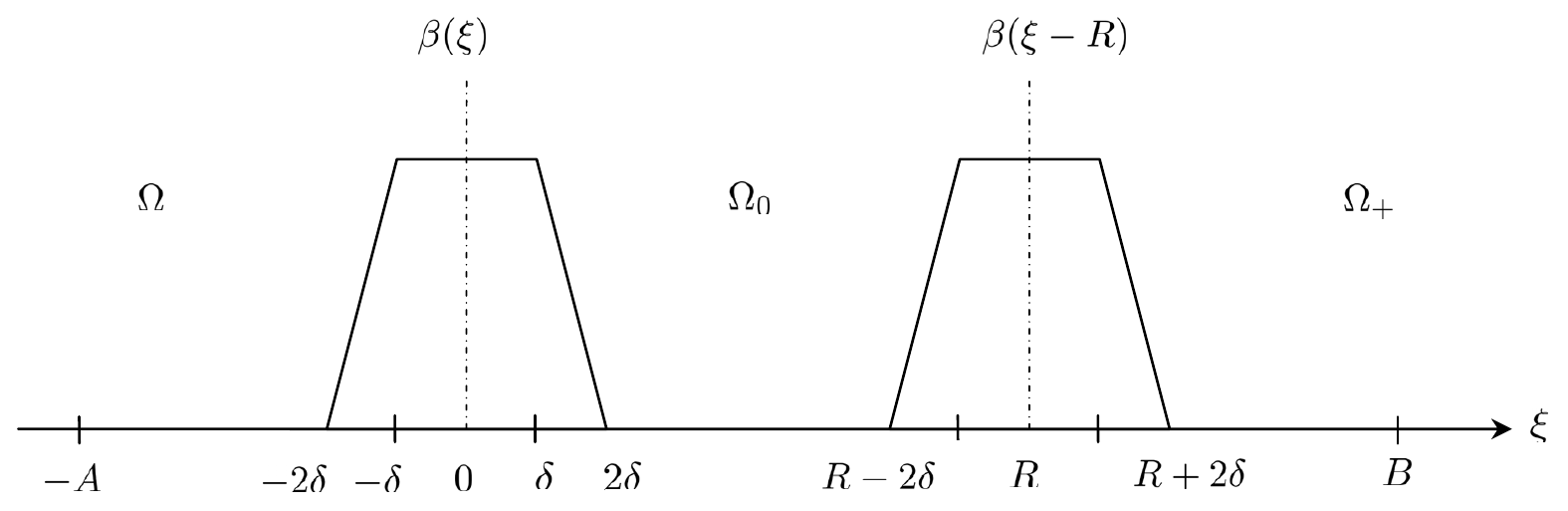}		
\caption{Approximation of the mollifier $\beta(\xi)$.}
\label{fig-mollifiers}
\end{figure}
Then, the fully nonlinear terms in System \eqref{asta}, namely ${\mathscr F}_1$, ${\mathscr F}_2$, ${\mathscr G}_1$, ${\mathscr G}_2$, ${\mathscr G}_3$, as well as $\vp_\tau, \ \vp_{\tau\xi}$,
have to be computed separately in eight intervals, as they are zero elsewhere: $[-2 \delta, - \delta] \cup \cdots \cup [R+\delta, R+2 \delta]$, see Figure \ref{fig-mollifiers}.  We refer to the Appendix for the formulas.

Hereafter, we present some typical numerical results for the fully nonlinear problem.
Simulations were performed using a standard pseudo-spectral method with small time step $\Delta t = 10^{-5}$ and small amplitude of initial perturbations (of order $10^{-4}$ to $10^{-3}$), to ensure sufficient accuracy.

We consider the situation when ignition temperature is fixed at $\theta_i=0.75$ and $\ell=100$, in such a case ${\Le}_c \simeq 0.5641 $. Three significant values of the Lewis number have been chosen in the interval $(0,\Le_c)$, namely $\Le=0.10$, $\Le=0.20$ and $\Le=0.50$.
Figures \ref{ignition-trailing-fronts-1050} and \ref{ignition-trailing-temperature-1050} represent the interface patterns and temperature levels. Numerically, we observe that, after a rapid transition period, a steady configuration consisting of ``two-cell'' patterns for the ignition and trailing interfaces is established. These simulations confirm the theoretical analysis, that is instability of the planar fronts for $\Le\in (0,\Le_c)$.

\begin{figure}[h]
\subfloat[${\rm Le} = 0.10$]{
	\begin{minipage}[t]{0.3\textwidth}
	\centering \label{igfront:7510}
	\includegraphics[height=3cm,width=4cm]{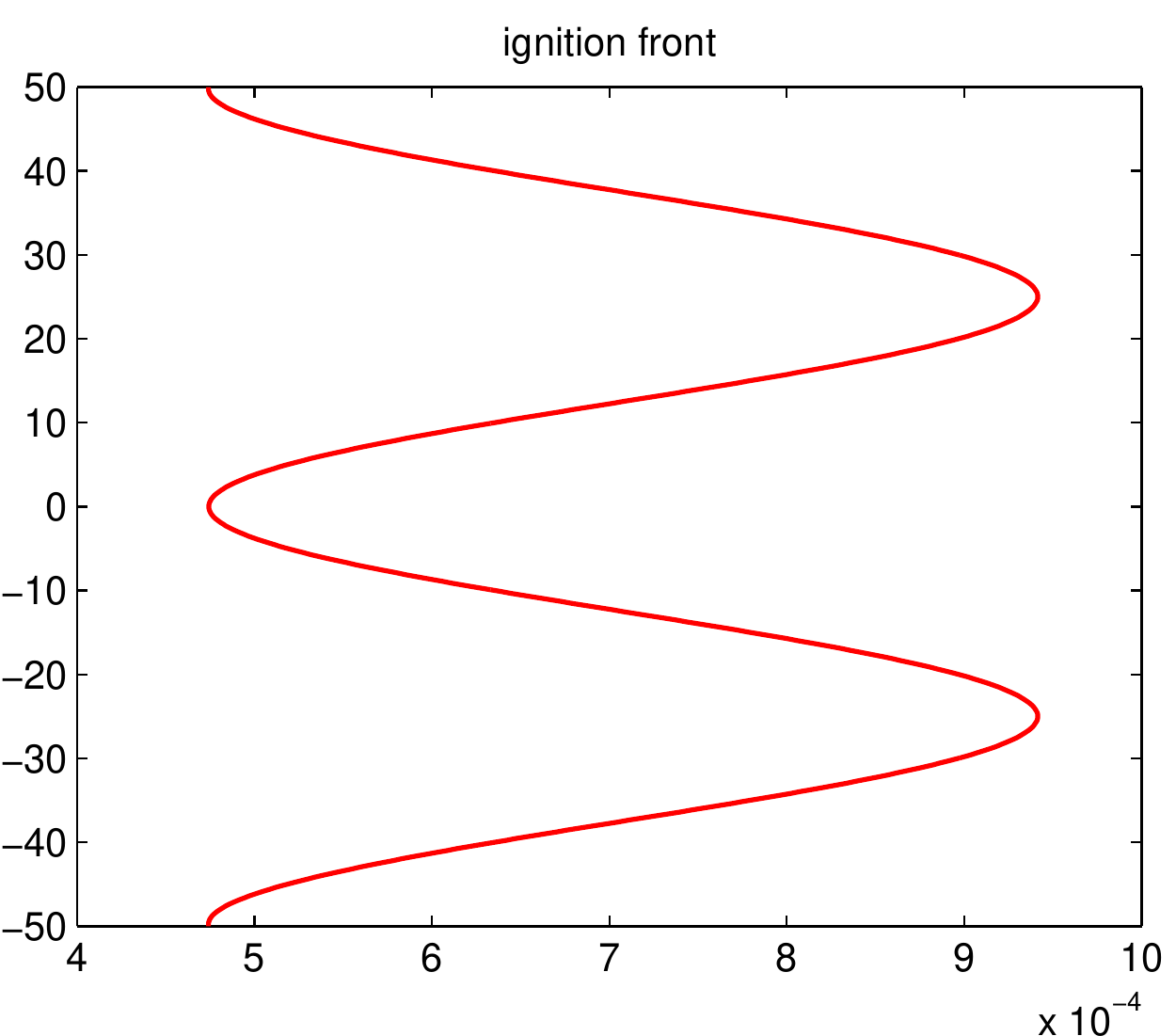}
	\end{minipage}}
\subfloat[${\rm Le} = 0.20$]{
	\begin{minipage}[t]{0.3\textwidth}
	\centering \label{igfront:7520}
	\includegraphics[height=3cm,width=4cm]{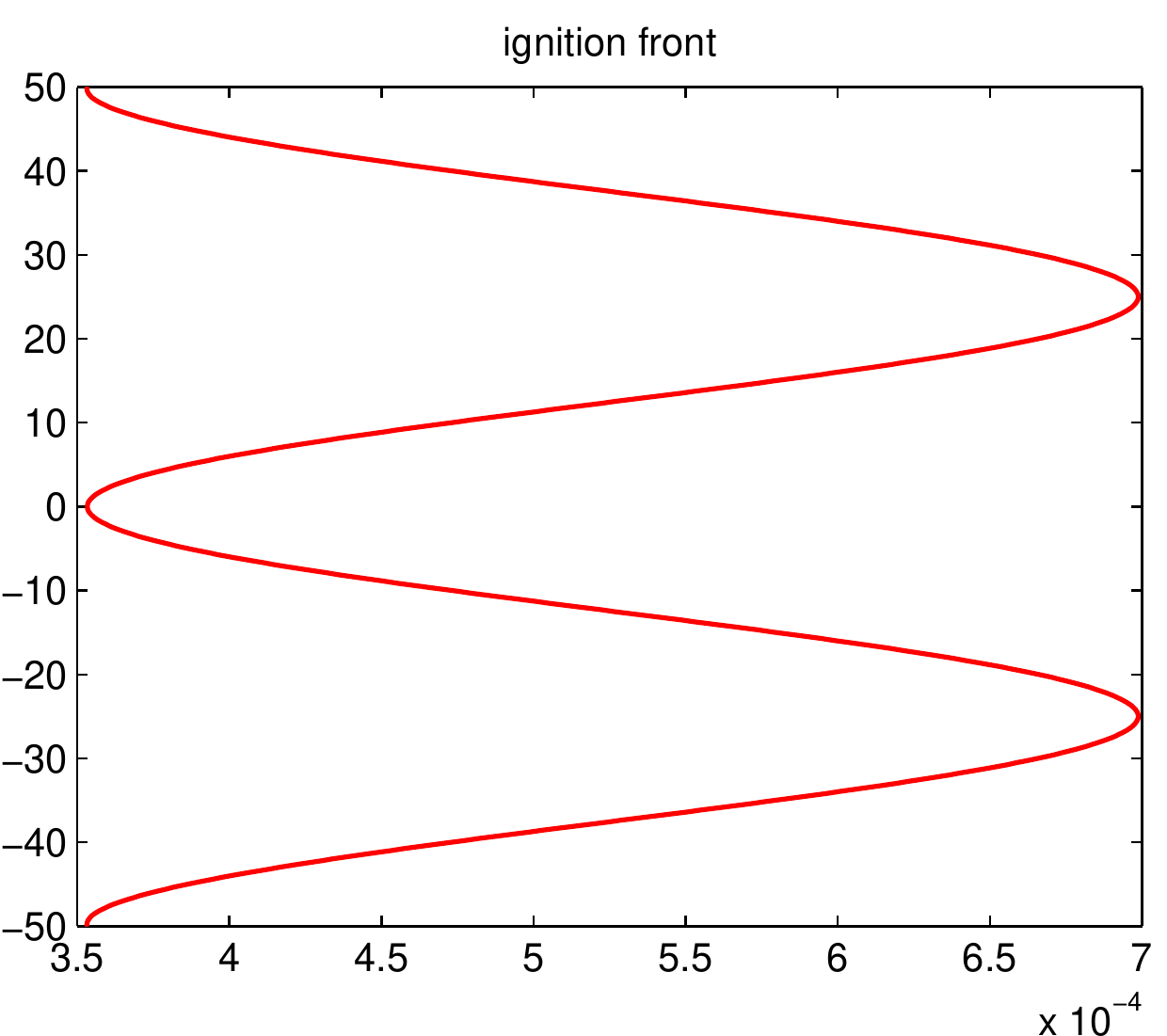}
	\end{minipage}}
\subfloat[${\rm Le} = 0.50$]{
	\begin{minipage}[t]{0.3\textwidth}
	\centering \label{igfront:7550}
	\includegraphics[height=3cm,width=4cm]{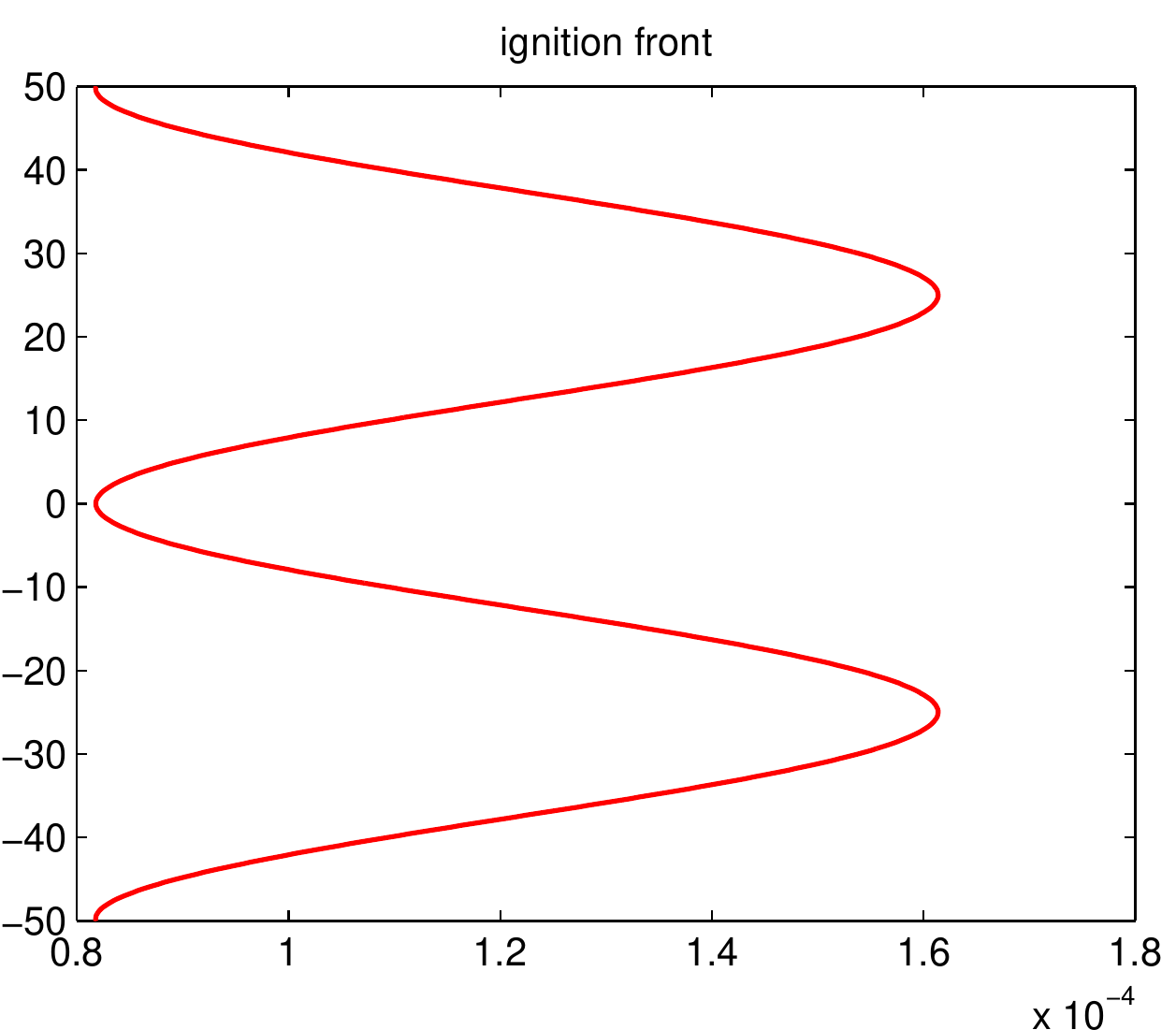}
	\end{minipage}} \\
\subfloat[${\rm Le} = 0.10$] {
	\begin{minipage}[t]{0.3\textwidth}
	\centering \label{trfront:7510}
	\includegraphics[height=3cm,width=4cm]{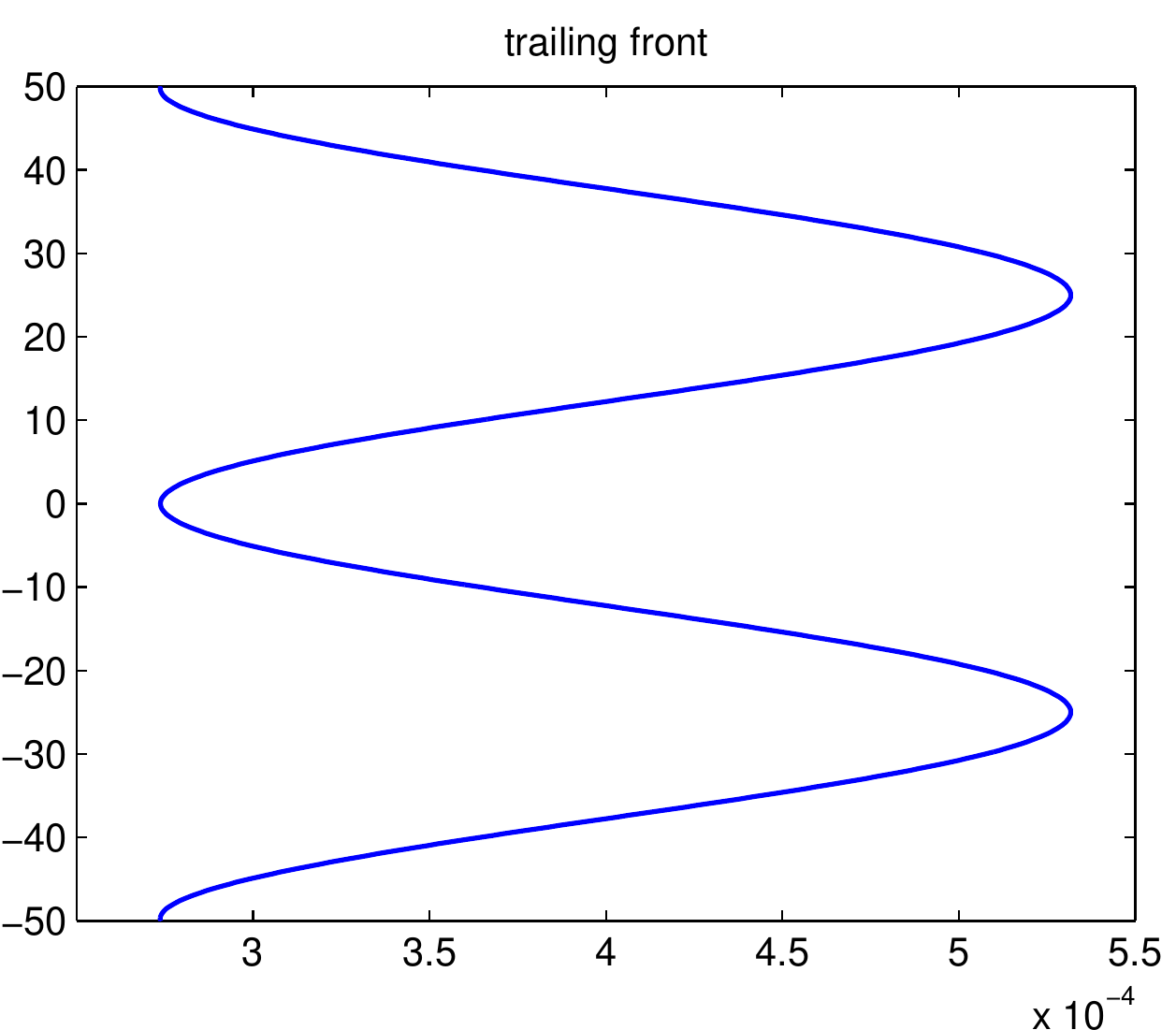}
	\end{minipage}}
\subfloat[${\rm Le} = 0.20$ ]{
	\begin{minipage}[t]{0.3\textwidth}
	\centering \label{trfront:7520}
	\includegraphics[height=3cm,width=4cm]{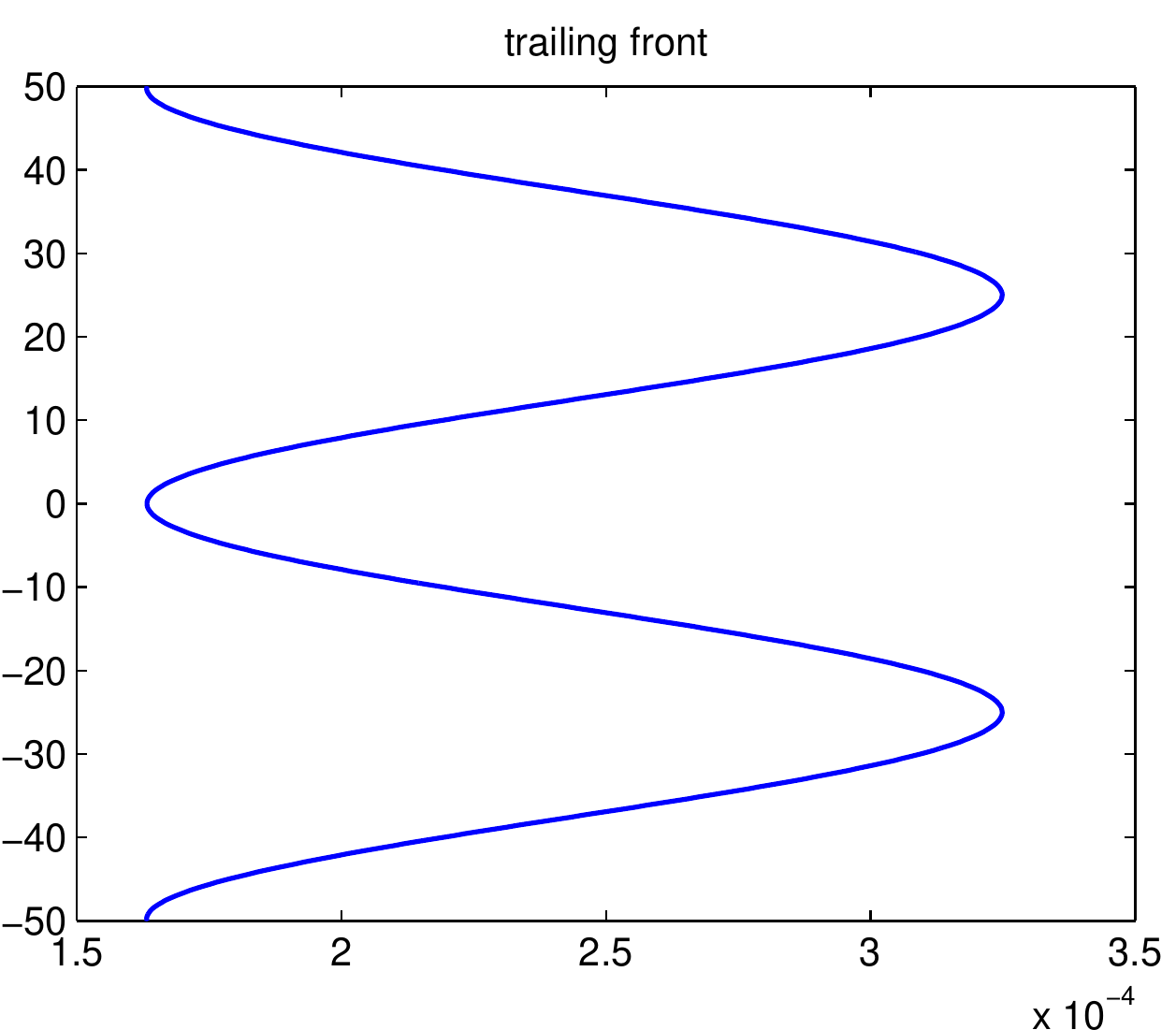}
	\end{minipage}}
\subfloat[${\rm Le} = 0.50$ ]{
	\begin{minipage}[t]{0.3\textwidth}
	\centering \label{trfront:7550}
	\includegraphics[height=3cm,width=4cm]{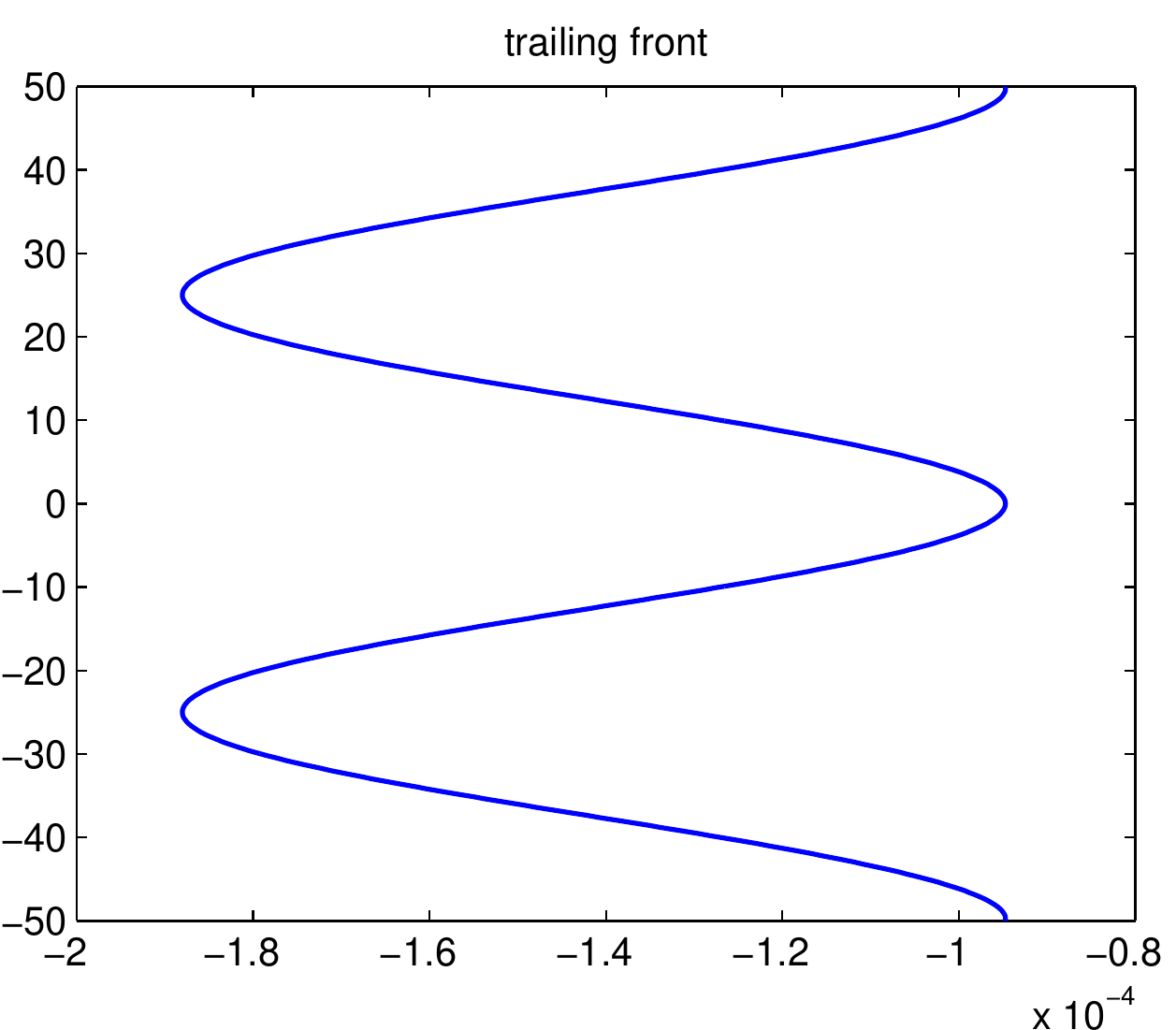}
	\end{minipage}}\\
\caption{Patterns of ignition (above) and trailing (below) interfaces.
	Here $\theta_i = 0.75$, $\ell=100$, ${\Le}_c \simeq 0.5641$.}
\label{ignition-trailing-fronts-1050}
\end{figure}

\begin{figure}[h]
\subfloat[${\rm Le} = 0.10$]{
	\begin{minipage}[t]{0.3\textwidth}
	\centering \label{igtemp:7510}
	\includegraphics[height=3cm,width=4cm]{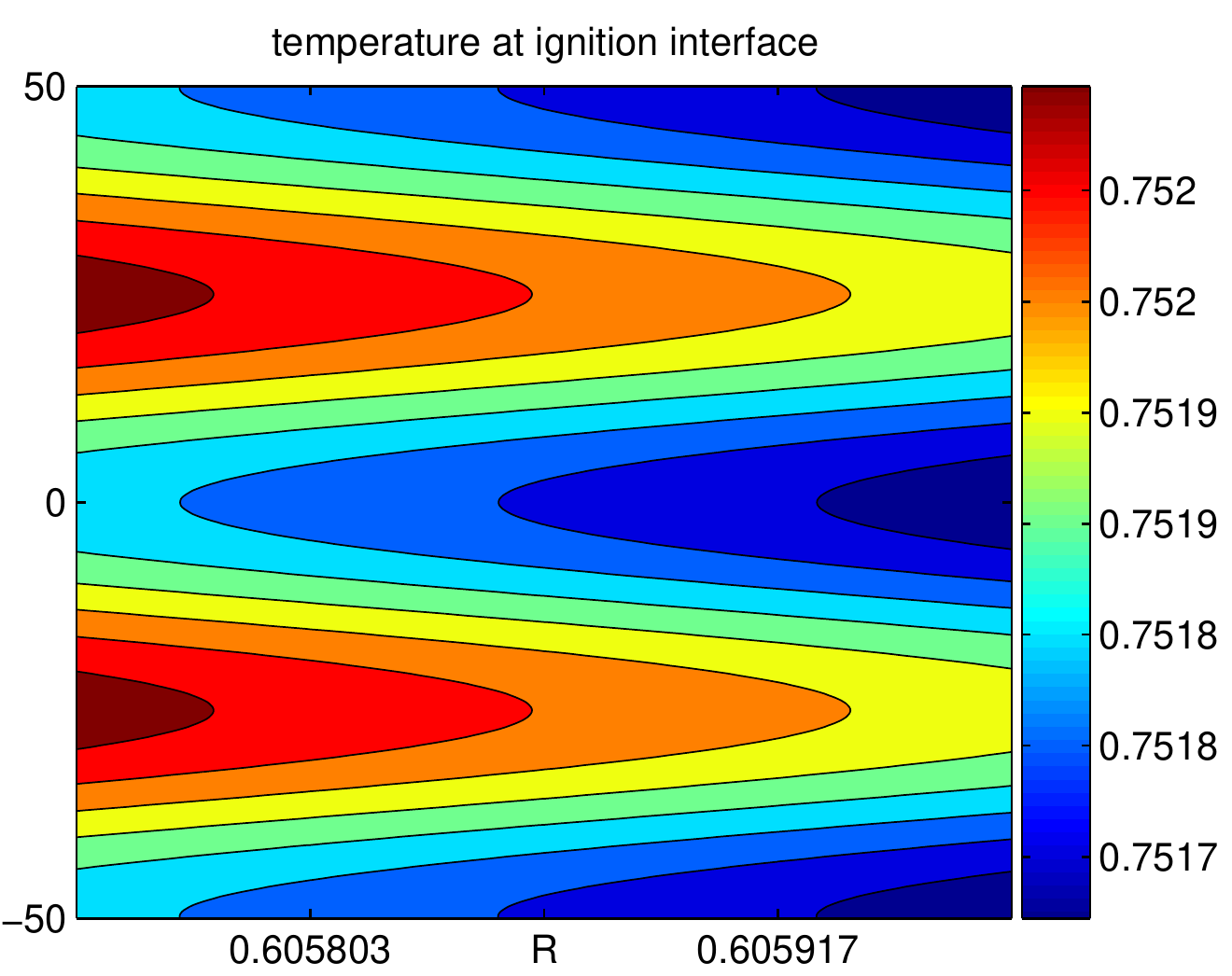}
	\end{minipage}}
\subfloat[${\rm Le} = 0.20$]{
	\begin{minipage}[t]{0.3\textwidth}
	\centering \label{igtemp:7520}
	\includegraphics[height=3cm,width=4cm]{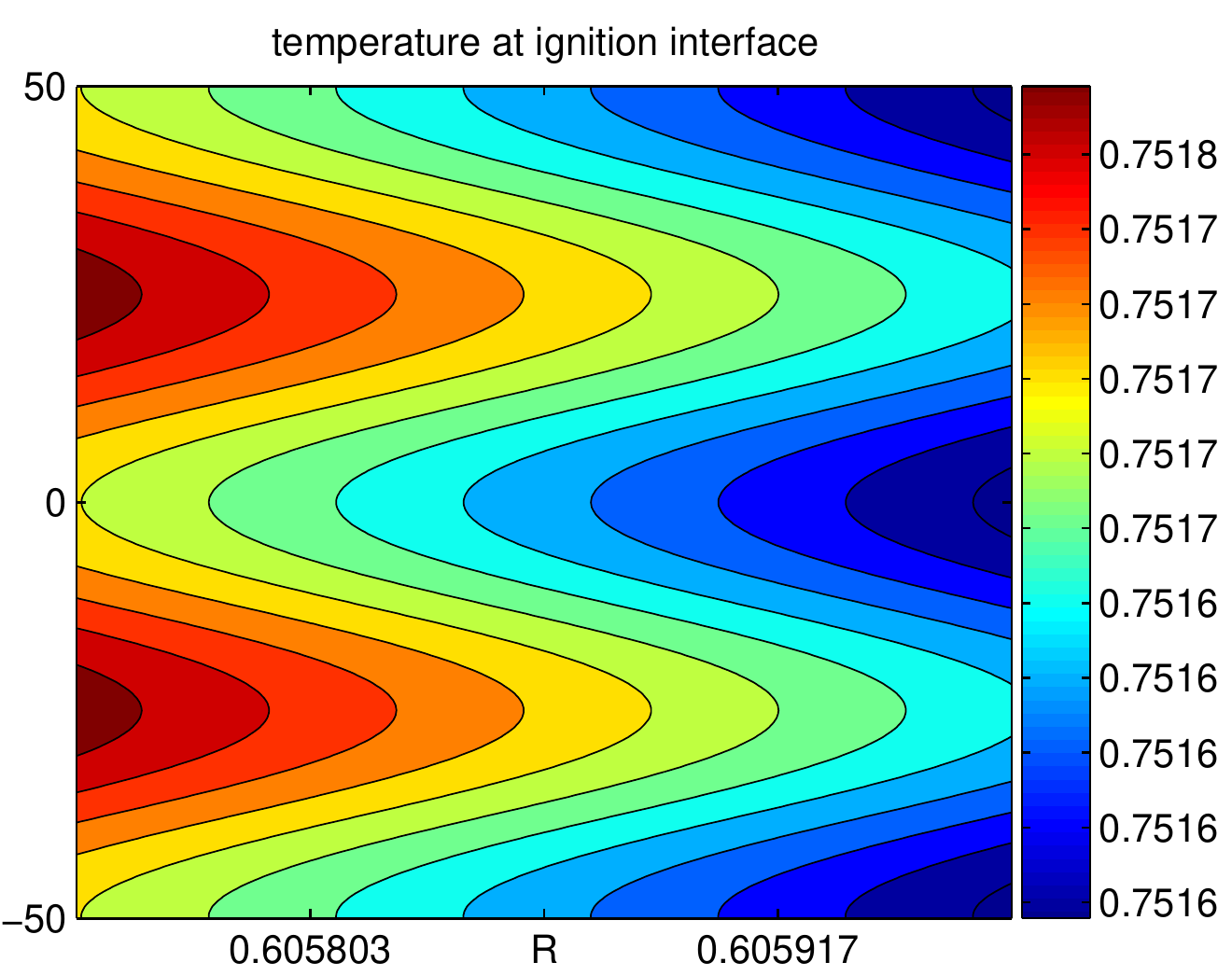}
	\end{minipage}}
\subfloat[${\rm Le} = 0.50$]{
	\begin{minipage}[t]{0.3\textwidth}
	\centering \label{igtemp:7550}
	\includegraphics[height=3cm,width=4cm]{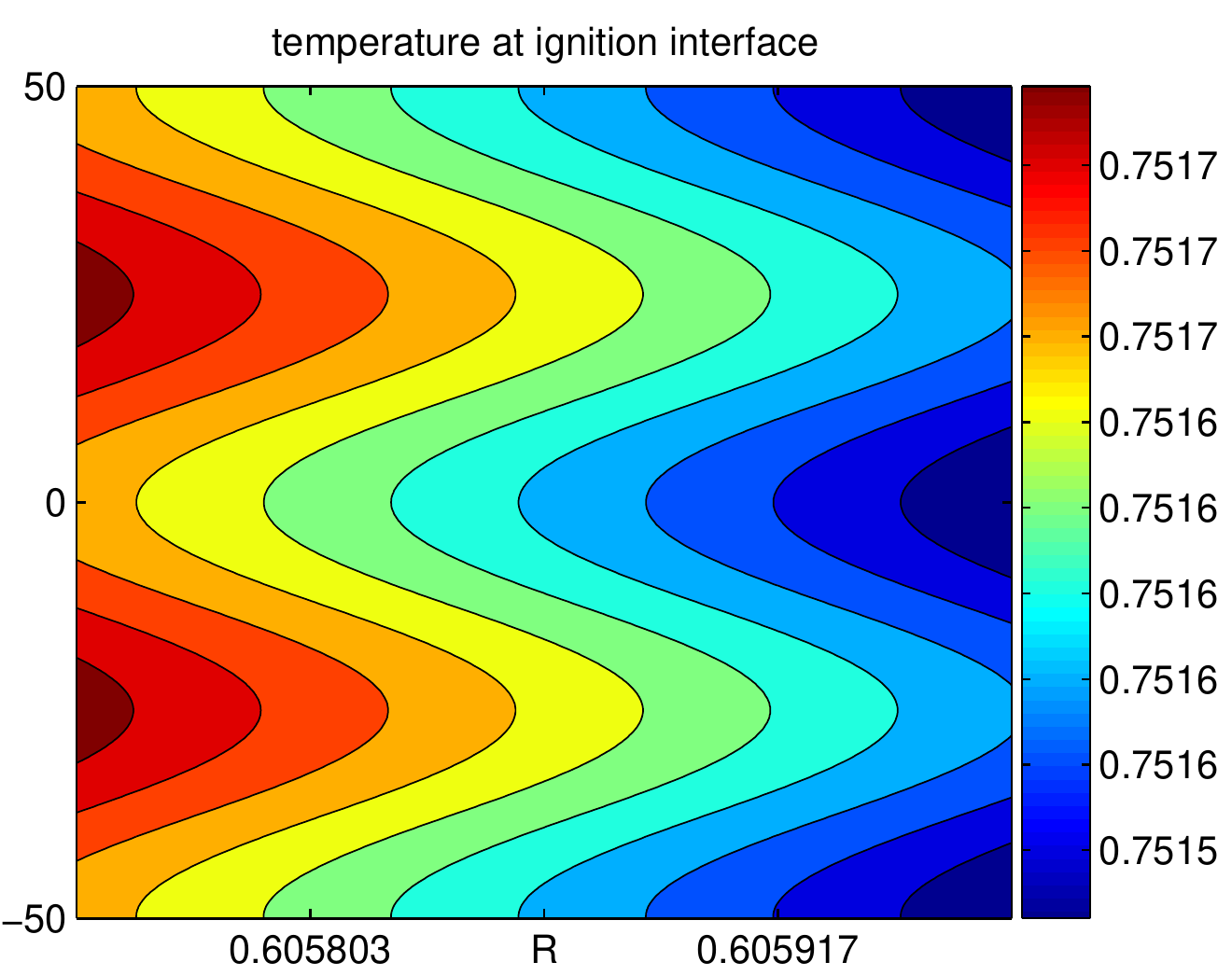}
	\end{minipage}} \\
\subfloat[${\rm Le} = 0.10$] {
	\begin{minipage}[t]{0.3\textwidth}
	\centering \label{trtemp:7510}
	\includegraphics[height=3cm,width=4cm]{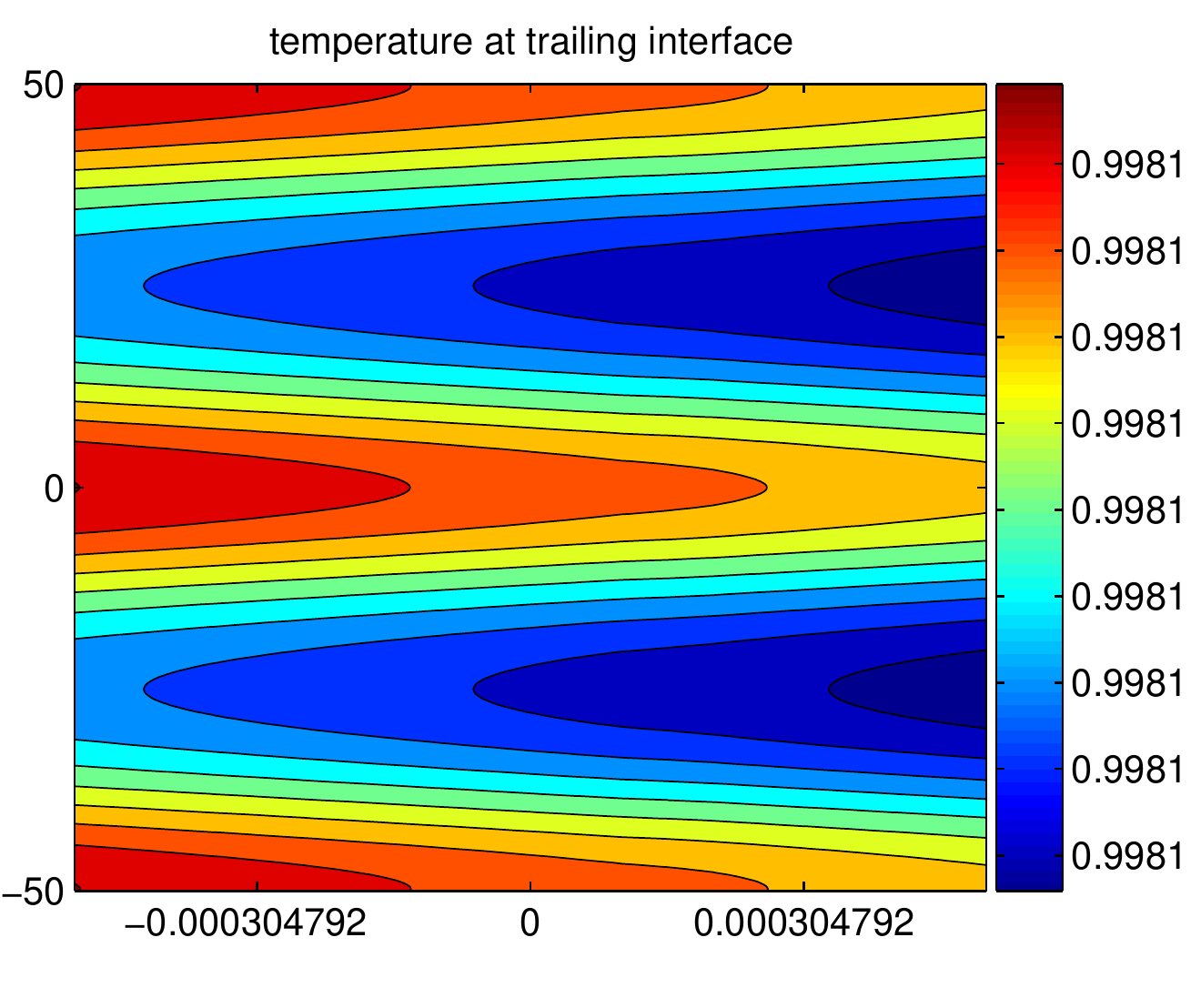}
	\end{minipage}}
\subfloat[${\rm Le} = 0.20$ ]{
	\begin{minipage}[t]{0.3\textwidth}
	\centering \label{trtemp:7520}
	\includegraphics[height=3cm,width=4cm]{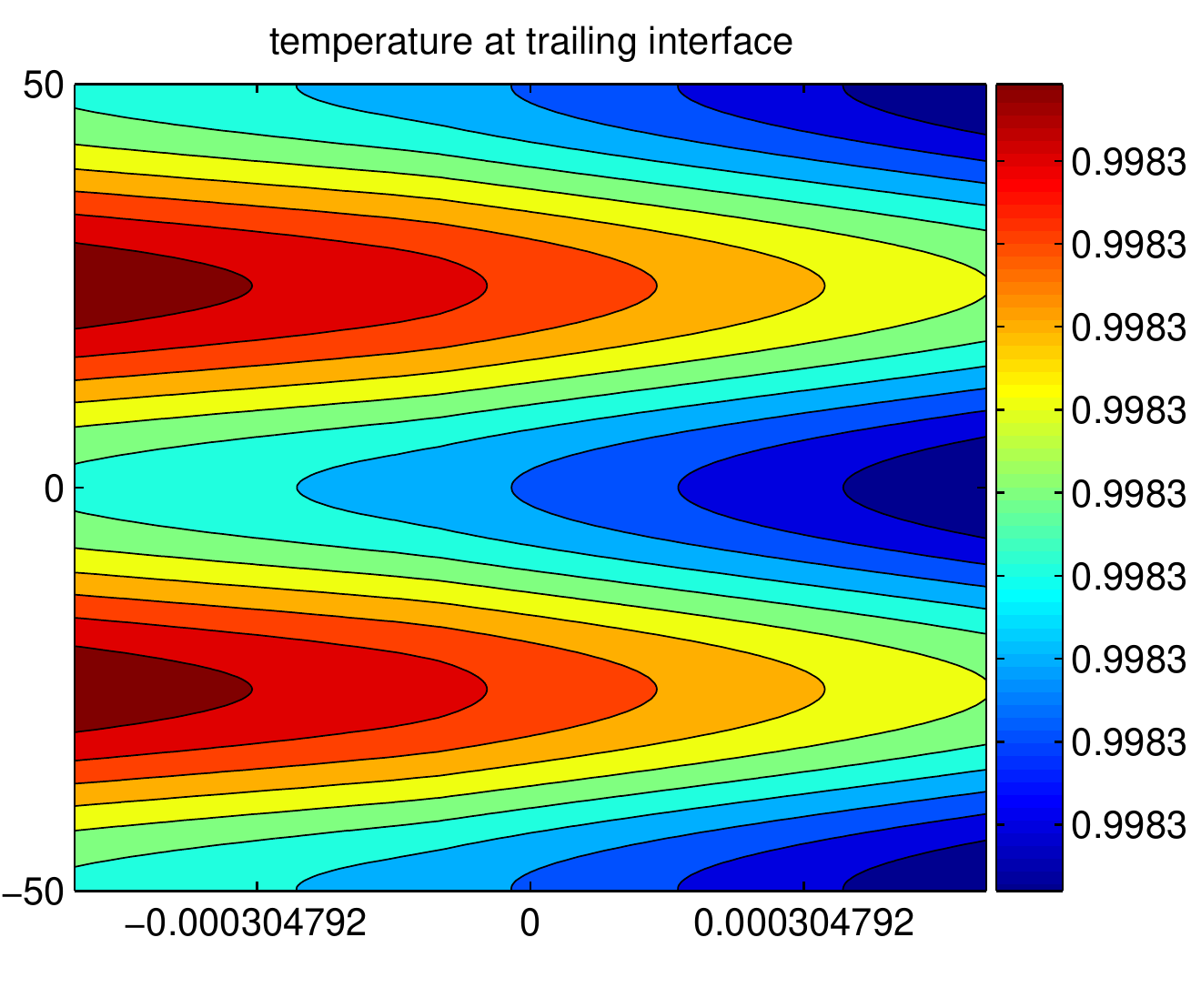}
	\end{minipage}}
\subfloat[${\rm Le} = 0.50$ ]{
	\begin{minipage}[t]{0.3\textwidth}
	\centering \label{trtemp:7550}
	\includegraphics[height=3cm,width=4cm]{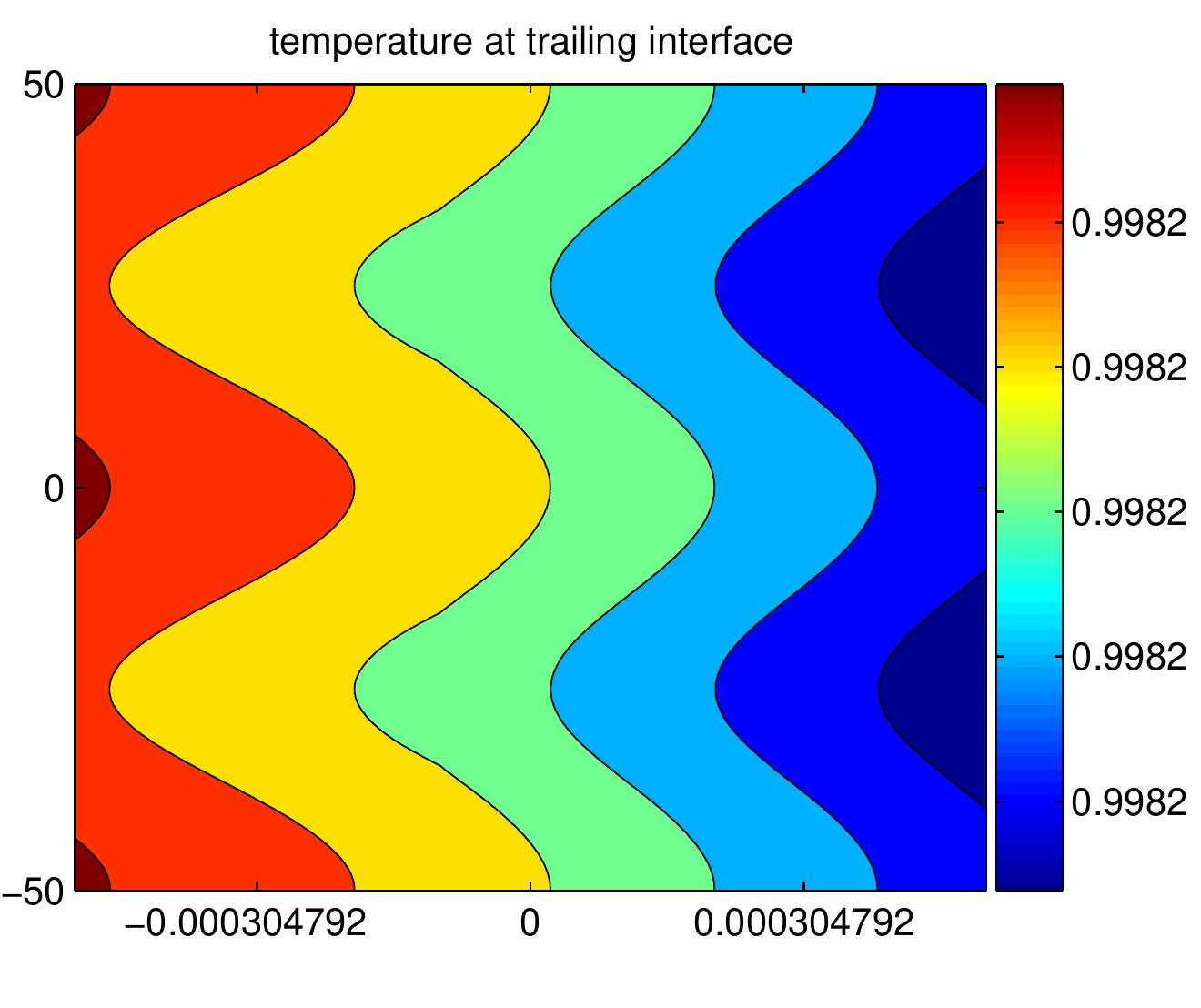}
	\end{minipage}}
\caption{Temperature levels around the ignition (above)
	and trailing (below) interfaces. Parameters are as in Figure \ref{ignition-trailing-fronts-1050}.}
\label{ignition-trailing-temperature-1050}
\end{figure}

\section*{Acknowlegments}
We are grateful to Grisha Sivashinsky for bringing reference \cite{BGKS15} to our attention. We also thank Peter Gordon for fruitful discussions and advices. The work of Z.W. was supported by National Natural Science Foundation of China (Grants No. 11761007, 11661004), Research Foundation of Education Bureau of Jiangxi Province (Grant No. GJJ160564) and Doctoral Scientific Research Foundation (Grant No. DHBK2017148).

\section{Appendix}
\footnotesize{
In this Appendix, we write explicitly the formulas for the fully nonlinear terms in system \eqref{asta}, namely ${\mathscr F}_1$, ${\mathscr F}_2$, ${\mathscr G}_1$, ${\mathscr G}_2$, ${\mathscr G}_3$.
They have to be computed separately in 8 intervals, as they are zero elsewhere: $[-2 \delta, - \delta] \cup [-\delta,0] \cup [0, \delta] \cup [\delta, 2 \delta] \cup [R- 2 \delta, R-\delta] \cup [R-\delta, R] \cup [R, R+\delta] \cup [R+\delta, R+2 \delta]$, see Figure \ref{fig-mollifiers}. The formulas for $\vp_\tau$ and $\vp_{\tau\xi}$ can be easily derived.

\medskip

{\bf (1)} For $\xi \in [-2 \delta, -\delta]$ and $\xi = \frac{A}{2} (x - 1)$,
\begin{align*}
{\mathscr F}_1(u, w) = & \frac{2 \Le}{A \Le -aR w(0^+)}u_x,\\[1mm]
{\mathscr F}_2(u, w) = &\frac{1}{\Le - R w(0^+)}\bigg\{\frac{2 Rw(0^+)}{A}u_x+\frac{8 \pi^2 R \beta}{A \ell^2}[w_{yy}(0^+)u_x +2w_y(0^+)u_{xy}]\bigg\} \\
& + \frac{\Le}{(\Le - R w(0^+))^2} \bigg\{\frac{16\pi^2 R^2 \beta w_y^2(0^+)}{\Le A \ell^2}\bigg (u_x + \frac{\beta}{A}u_{xx}\bigg )+
\frac{4Rw(0^+)}{A^2}\bigg (2 - \frac{R w(0^+)}{\Le} u_{xx}\bigg )\bigg\}.
\end{align*}
\medskip
\par

{\bf (2)} For $\xi \in [- \delta, 0]$ and $\xi = \frac{A}{2} (x - 1)$,
\begin{align*}
{\mathscr F}_1(u) = & \frac{2}{A}u_x, \qquad\;\,
{\mathscr F}_2(u, w) = \frac{8R \pi^2}{A\ell^2\Le}\bigg(w_{yy}(0^+) u_{xx} + w_y(0^+) u_{xy} + \frac{2 R w_y^2(0^+)}{A \Le}u_{xx}\bigg ).
\end{align*}

\medskip
\par

{\bf (3)} For $\xi \in [0, \delta]$ and $\xi = \frac{R}{2} (x + 1)$,
\begin{align*}
{\mathscr F}_1(u) = & \frac{1}{\Le}e^{-\frac{R}{2}(x+1)}+\frac{2}{R}u_x, \qquad\;\,
{\mathscr G}_1(w) =  \Le\ e^{-\frac{\Le R}{2}(x+1)} w(0^+) + \frac{2}{R} w_x,\\[1mm]
{\mathscr F}_2(u, w) = & \frac{4\pi^2 R}{\ell^2\Le^2}e^{-\frac{R}{2}(x+1)}\bigg [-w_{yy}(0^+) w(0^+) + w_y^2(0^+)\bigg (1-\frac{Rw(0^+)}{\Le}\bigg )\bigg ]\\
&+\frac{8\pi^2}{\ell^2 \Le}\bigg[w_{yy}(0^+) u_x + 2 w_y(0^+)\bigg (u_{xy} + \frac{w_y(0^+)}{\Le} u_{xx}\bigg )\bigg ],\\[1mm]
{\mathscr G}_3(w) = & \frac{4\pi^2R}{\Le\ell^2}e^{-\frac{\Le R}{2}(x+1)}\bigg [w_{yy}(0^+)w(0^+) + w_y^2(0^+)\bigg (\frac{2\pi}{\ell} - w(0^+)-1\bigg )\bigg ]\\
&+\frac{8\pi^2}{\Le^2 \ell^2}\bigg(w_{yy}(0^+)w_x  + 2\pi^2w_y(0^+)w_{xy}+\frac{2w_y^2(0^+)}{\Le}w_{xx} \bigg ).
\end{align*}
\medskip
\par

{\bf (4)} For $\xi \in [\delta, 2\delta]$ and $\xi = \frac{R}{2} (x + 1)$,
\begin{align*}
{\mathscr F}_1(u, w) = & \frac{\beta w(0^+)}{\Le + R w(0^+)}e^{-R (x+1)\over 2} + \frac{2\Le}{R(\Le + R w(0^+))}u_x, \\[1mm]
{\mathscr F}_2(u, w) = & \frac{1}{\Le + R w(0^+)} \bigg\{\frac{R \beta}{\Le} e^{-\frac{R}{2}(x+1)}\bigg [\frac{4\pi^2 \beta}{\ell^2}\big (2 w_y^2(0^+)-w(0^+) w_{yy}(0^+)\big )-w^2(0^+)\bigg ]\\
&\phantom{\frac{1}{\Le + R w(0^+)} \bigg\{\;} + 8\pi^2 \beta \big(2 w_y(0^+) u_{xy} - w_{yy}(0^+) u_x\big)-2w(0^+) u_x\bigg\} \\
& +\frac{\Le}{(\Le + Rw(0^+))^2}\bigg\{e^{-\frac{R}{2}(x+1)}\bigg [-\frac{4\pi^2 R\beta}{\ell^2\Le }\bigg (\frac{4\beta w(0^+) w_{xy}(0^+) w_y(0^+)}{\Le}+ R w_y^2(0^+)\bigg (\frac{\beta w(0^+)}{\Le} + \frac{1}{R}\bigg )\\
&\phantom{+\frac{\Le}{(\Le + Rw(0^+))^2}\bigg\{e^{-\frac{R}{2}(x+1)}\bigg [\;}+ 2 R w(0^+) w_y^2(0^+)\bigg  )
+\frac{R^2 w^2(0^+)}{\Le}\bigg (\frac{\beta w(0^+)}{\Le}+\frac{1}{R}\bigg )+2 R w^2(0^+)\bigg ]\\
&\phantom{+\frac{\Le}{(\Le + Rw(0^+))^2}\bigg\{\;}+\frac{16 \pi^2 \beta w_y(0^+)}{\Le \ell^2}[-2w_{xy}(0^+)u_x + w_y(0^+) u_{xx}]-4 u_{xx} w(0^+)\bigg (\frac{w(0^+)}{\Le} +\frac{2}{R}\bigg )\bigg\},\\[1mm]
{\mathscr G}_1(w) = & \frac{\Le}{\Le + R w(0^+)}\bigg [e^{-\frac{\Le R}{2}(x+1)} \beta w(0^+)\bigg (\Le - \frac{R w(0^+)}{\Le}\bigg )+ \frac{2}{R}w_x\bigg ],\\[1mm]
{\mathscr G}_2(w) = &\frac{\Le}{\Le + R w(0^+)}w - \frac{\beta w(0^+)\Le}{\Le + R w(0^+)}e^{-\frac{\Le R}{2} (x+1)},\\[1mm]
{\mathscr G}_3(w) = & \frac{1}{\Le + R w(0^+)} \bigg\{e^{-\frac{\Le R}{2} (x+1)}\bigg [\frac{4\pi^2 R \beta (2+\beta \Le)}{\Le \ell^2}[w(0^+) w_{yy}(0^+) + 2w_y^2(0^+)]
-R w(0^+) (1+\beta \Le)\bigg ]\\
&\phantom{\frac{1}{\Le + R w(0^+)} \bigg\{\;}+\frac{4\pi^2}{\ell^2\Le}\big [w_{yy}(0^+) (2\beta w_x -Rw)+ 2 w_y(0^+) (2\beta w_{xy} - R w_y)]-2 w(0^+) w_x\bigg\} \\
& + \frac{\Le}{(\Le\!+\! R w(0^+))^2} \bigg\{e^{-\frac{\Le R}{2}(x\!+\!1)}\bigg [\frac{4\pi^2 R \beta w_y^2(0^+)}{\Le^2\ell^2}\bigg (Rw(0^+) \big(2\beta\!-\!2\!-\!4 \Le\,\beta\!-\!\beta\Le(1\!+\!\beta \Le)\big)\!+\!\Le
(\beta\Le\!-\!2)\bigg ) \\
&\phantom{+ \frac{\Le}{(\Le + R w(0^+))^2} \bigg\{e^{-\frac{\Le R}{2} (x+1)}\bigg [\;}+ R w^2(0^+) \bigg(1+ (1+\beta \Le)\bigg (2+ \frac{R w(0^+)}{\Le}\bigg )\bigg )\bigg ]\\
&\phantom{+ \frac{\Le}{(\Le + R w(0^+))^2} \bigg\{\;}
+ \frac{4\pi^2 R w_y^2(0^+)}{\Le \ell^2}\bigg (\frac{2Rw}{\Le} +\frac{4\beta^2}{\Le R} w_{xx}-8\beta w_x\bigg )
- \frac{4 w(0^+)}{\Le R}\bigg (2+\frac{R w(0^+)}{\Le}\bigg ) w_{xx}\bigg\},
\end{align*}

\medskip
\par

{\bf (5)} For $\xi \in [R-2 \delta, R-\delta]$ and $\xi = \frac{R}{2} (x + 1)$,
\begin{align*}
{\mathscr F}_1(u) =  & \frac{1}{R (\theta_i+u(R^+))}\Big (2 \theta_i u_x - \beta_R u(R^+) e^{-\frac{R}{2}(x+1)}\Big ), \\[1mm]
 {\mathscr F}_2(u) = & \frac{1}{\theta_i + u(R^+)}\bigg\{e^{-\frac{R}{2}(x+1)}\bigg [\frac{4 \pi^2 \beta_R^2}{\theta_i \ell^2 R} \big(u(R^+) u_{yy}(R^+) + 2 u_y^2(R^+)\big)
+ \frac{\beta_R u(R^+)}{\theta_i R}\bigg ]\\
&\phantom{\frac{1}{\theta_i + u(R^+)}\bigg\{\;}-\frac{8 \pi^2 \beta_R}{\ell^2 R} \big(u_{yy}(R^+)u_x  + 2u_y(R^+)u_{xy} \big )\bigg\} \\
&+\frac{\theta_i}{(\theta_i + u(R^+))^2}\bigg\{e^{-\frac{R}{2}(x+1)}\bigg [\frac{4 \pi^2 \beta_R^2 u_y^2(R^+)}{\theta_i \ell^2 R}\bigg (\frac{\beta_R-4}{\theta_i} u(R^+)-1\bigg )\\
&\phantom{+\frac{\theta_i}{(\theta_i + u(R^+))^2}\bigg\{e^{-\frac{R}{2}(x+1)}\bigg [}\;\;\,-\frac{u^2(R^+)}{\theta_i R}\bigg (\frac{\beta_R u(R^+)}{\theta_i} -1 + 2 \beta_R\bigg )\bigg ]\\
&\phantom{+\frac{\theta_i}{(\theta_i + u(R^+))^2}\bigg\{;}\;+ \frac{16 \pi^2 \beta_R u_y^2(R^+)}{\theta_i \ell^2 R}\bigg (u_x + \frac{\beta_R}{R}u_{xx}\bigg )-\frac{4 u(R^+)}{R^2}u_{xx}\bigg (\frac{u(R^+)}{\theta_i} + 2\bigg )\bigg\},\\[1mm]
{\mathscr G}_1(u, w) = & \frac{\theta_i}{\theta_i + u(R^+)}\bigg [e^{-\frac{\Le R}{2}(x+1)}\frac{\Le\, \beta_R u(R^+)}{\theta_i R}\bigg (\frac{u(R^+)}{\theta_i} - \Le\bigg )
+\frac{2}{R}w_x\bigg ],\\[1mm]
{\mathscr G}_2(u,w) = & \frac{\theta_i}{\theta_i + u(R^+)} \bigg (e^{-\frac{\Le R}{2}(x+1)}\frac{\Le \beta_R u(R^+)}{\theta_i R} + w\bigg ),\\[1mm]
{\mathscr G}_3(u,w) = & \frac{\theta_i}{\Le(\theta_i+u(R^+))}\bigg\{e^{-\frac{\Le R}{2} (x+1)}\bigg [\frac{4\pi^2 \beta_R \Le}{\theta_i \ell^2 R}
(\Le \beta_R -2)\big(u_{yy}(R^+) u(R^+) + 2 u_y^2(R^+)\big)\\
&\phantom{\frac{\theta_i}{\Le(\theta_i+u(R^+))}\bigg\{e^{-\frac{\Le R}{2} (x+1)}\bigg [}\;\;\,+ \frac{\Le^2 u^2(R^+)}{\theta_i R} (\Le \beta_R-1)\bigg ]\\
&\phantom{\frac{\theta_i}{\Le(\theta_i+u(R^+))} \bigg\{\;}- \frac{4\pi^2 u_{yy}(R^+)}{\ell}\bigg (\frac{w}{\ell} + \frac{2\beta_R}{R} w_x\bigg )- \frac{2 \Le}{R} w_x
- \frac{8\pi^2 u_y(R^+)}{\ell^2} \bigg (w_y + \frac{2 \beta_R}{R}w_{xy}\bigg )\bigg\} \\
& +\frac{\theta_i}{\Le (\theta_i+u(R^+))^2} \bigg\{e^{- \frac{\Le R}{2}(x+1)}\bigg[\frac{8\pi^2 \beta_R \Le\, u_y^2(R^+)}{\theta_i \ell^2 R}
\frac{u(R^+)}{\theta_i}\bigg (1\!+\!\frac{\Le\, \beta_R}{2}(\Le \beta_R\!-\!5\!+\!2\Le)\bigg )\\
&\phantom{+\frac{\theta_i}{\Le (\theta_i+u(R^+))^2} \bigg\{e^{- \frac{\Le R}{2}(x+1)}\bigg[}\;\;\,
+\frac{8\pi^2 \beta_R\Le\, u_y^2(R^+)}{\theta_i \ell^2 R}\bigg (1+\frac{\Le^2 \beta_R}{2}\bigg )\bigg ] \\
&\phantom{+\frac{\theta_i}{\Le (\theta_i+u(R^+))^2} \bigg\{e^{-\frac{\Le R}{2}(x+1)}\bigg[\;}\;-\frac{\Le^2 u^2(R^+)}{\theta_i R}\bigg[\bigg (2+\frac{u(R^+)}{\theta_i}\bigg )
(\Le \beta_R-1) +\Le\bigg )\bigg ] \\
& \phantom{+\frac{\theta_i}{\Le (\theta_i+u(R^+))^2} \bigg\{\;}\;+\frac{8\pi^2 u_y^2(R^+)}{\theta_i \ell^2}\bigg (w\!+\! \frac{4 \beta_R}{R}w_x\!+\!
\frac{2 \beta_R^2}{R^2}w_{xx}\bigg )\!-\!\frac{4u(R^+)}{R^2}\bigg (2\!+\!\frac{u(R^+)}{\theta_i}\bigg )w_{xx}\bigg\},
\end{align*}

\medskip
\par

{\bf (6)} For $\xi \in [R-\delta, R]$ and  $\xi = \frac{R}{2} (x + 1)$,
\begin{align*}
 {\mathscr F}_1(u) = & -\frac{u(R^+)}{\theta_i R} e^{-\frac{R}{2}(x+1)} + \frac{2}{R}u_x, \quad {\mathscr G}_1(u, w) =  -\frac{\Le^2 u(R^+)}{\theta_i R} e^{- \Le R (x+1) \over 2} + \frac{2}{R}w_x, \\[1mm]
 {\mathscr F}_2(u) = &  e^{-\frac{R}{2}(x+1)}\frac{4 \pi^2}{\theta_i^2 \ell^2 R}\bigg [u(R^+) u_{yy}(R^+) + u_y^2(R^+)\bigg (1+ \frac{u(R^+)}{\theta_i}\bigg )\bigg]\\
 & + \frac{8 \pi^2}{\theta_i \ell^2 R}\bigg (\frac{2 u_y^2(R^+)}{\theta_i R} u_{xx}- u_x u_{yy}(R^+) - 2u_y(R^+) u_{xy}\bigg ),
\\[1mm]
{\mathscr G}_3(u, w) = & e^{-\frac{\Le R}{2}(x+1)}\frac{4 \pi^2}{\theta_i^2 \ell^2 R}\big [u(R^+) u_{yy}(R^+) + \Le\, u_y^2(R^+) (1+\Le\, u(R^+))\big]\\
&+\frac{8\pi^2}{\Le \theta_i \ell^2 R}\bigg (\frac{2 u_y^2(R^+)}{\theta_i R}w_{xx}- w_x u_{yy}(R^+) - 2u_y(R^+) w_{xy}\bigg ),
\end{align*}

\medskip
\par

{\bf (7)} For $\xi \in [R, R+\delta]$ and $\xi = \frac{B-R}{2}(x+1) + R$,

\begin{align*}
{\mathscr F}_1(u) = & u(R^+) e^{-\frac{B-R}{2}(x+1)} + \frac{2}{B-R}u_x,\\[1mm]
{\mathscr F}_2(u) = & -e^{-\frac{B-R}{2}(x+1)}\frac{4\pi^2}{\theta_i \ell^2}\bigg [u(R^+) u_{yy}(R^+) + \big (1+\frac{u(R^+)}{\theta_i}\big )u_y^2(R^+)\bigg ]\\
&-\frac{8 \pi^2}{\theta_i \ell^2 (B-R)}\big [u_{yy}(R^+) u_x + 2u_y(R^+) u_{xy} - \frac{2 u_y^2(R^+)}{\theta_i (B-R)}u_{xx}\big ],\\[1mm]
{\mathscr G}_1(u, w) = & -\frac{\Le^2 (e^{-\Le R}-1) u(R^+)}{\theta_i R}e^{-\frac{\Le(B-R)}{2}(x+1)} + \frac{2}{B-R}w_x, \\[1mm]
{\mathscr G}_3(u, w) = & e^{-\frac{\Le(B-R)}{2}(x+1)}\frac{4 \pi^2 \Le (e^{-\Le R} -1)}{\theta_i^2 \ell^2 R}\bigg [u(R^+) u_{yy}(R^+) + \bigg (\frac{\Le u(R^+)}{\theta_i} +1\bigg )u_y^2(R^+)\bigg ]\\
& +\frac{8 \pi^2}{\Le \theta_i \ell^2 (B-R)}\bigg [-u_{yy}(R^+) w_x -2u_y(R^+) w_{xy} + \frac{2 u_y^2(R^+)}{\theta_i (B-R)}w_{xx}\bigg ]
\end{align*}

\medskip
\par
{\bf (8)} For $\xi \in [R + \delta, R+2\delta]$ and $\xi = \frac{B-R}{2}(x+1)+R$,
\begin{align*}
{\mathscr F}_1(u) = & \frac{\theta_i}{\theta_i -u(R^+)}\bigg (\beta_R u(R^+) e^{-\frac{B-R}{2}(x+1)} + \frac{2}{B-R}u_x \bigg ),\\[1mm]
{\mathscr F}_2(u) = & \frac{1}{\theta_i -u(R^+)}\bigg\{e^{-\frac{B-R}{2}(x+1)}\bigg [-\frac{4\pi^2 \beta_R^2}{\ell^2}\big (u(R^+)u_{yy}(R^+) + 2 u_y^2(R^+)\big)
+\beta_R u^2(R^+)\bigg ]\\
& \phantom{\frac{1}{\theta_i -u(R^+)}\bigg\{\;}-\frac{8 \pi^2 \beta_R}{\ell^2 (B-R)}\big(u_{yy}(R^+)u_x+2 u_y(R^+)u_{xy}\big)+\frac{2u(R^+)}{B-R}u_x\bigg\} \\
&+\frac{\theta_i^2}{(\theta_i\!-\!u(R^+))^2}\bigg\{e^{-\frac{B\!-\!R}{2}(x\!+\!1)}\bigg [\frac{u^2(R^+)}{\theta_i}\bigg (\frac{\beta_R u(R^+)}{\theta_i}\!-\!1\!-\!2\beta_R\bigg )\\
&\phantom{+\frac{\theta_i^2}{(\theta_i\!-\!u(R^+))^2}\bigg\{e^{-\frac{B\!-\!R}{2}(x\!+\!1)}\bigg [}\;\;\,
-\frac{4\pi^2\beta_R^2 u_y^2(R^+)}{\theta_i^2 \ell^2}\big (3\beta_Ru(R^+)\!-\!\theta_i\!+\!2 u(R^+)\big )\bigg ] \\
&\phantom{\frac{\theta_i^2}{(\theta_i -u(R^+))^2}\bigg\{}\,
-\frac{16 \pi^2 \beta_R^2 u_y^2(R^+)}{\theta_i^2 \ell^2 (B-R)}\bigg (u_x - \frac{1}{B-R}u_{xx}\bigg )-\frac{4}{(B-R)^2}\bigg (\frac{u^2(R^+)}{\theta_i} -1\bigg )u_{xx}\bigg\},\\[1mm]
{\mathscr G}_1(u, w) = & \frac{\theta_i}{\theta_i -u(R^+)}\bigg [-e^{-\frac{\Le (B-R)}{2}(x+1)}\frac{\beta_R\Le (e^{-\Le R} -1) u(R^+)}{\theta_i R}\bigg (\frac{u(R^+)}{\theta_i} +\Le\bigg )+{2\over B-R}w_x\bigg],\\[1mm]
{\mathscr G}_2(u, w) = & \frac{\theta_i}{\theta_i -u(R^+)}\bigg (e^{-\frac{\Le (B-R)}{2}(x+1)}\frac{\beta_R \Le (e^{-\Le R} -1) u(R^+)}{\theta_i R}+w\bigg ),\\[1mm]
{\mathscr G}_3(u, w) = & \frac{1}{\theta_i-u(R^+)} \bigg\{e^{-\frac{\Le (B-R)}{2}(x+1)}\frac{e^{-\Le R}-1}{\theta_i R}\bigg[\frac{4\pi^2\beta_R (2+\beta_R\,\Le)}{\ell^2}\big(u_{yy}(R^+)u(R^+)+2u_y^2(R^+)\big)\\
&\phantom{\frac{1}{\theta_i-u(R^+)} \bigg\{e^{-\frac{\Le (B-R)}{2}(x+1)}\frac{e^{-\Le R}-1}{\theta_i R}\!\bigg[}\;\;\,
-\Le\, (1+\!\beta_R\Le)u^2(R^+)\bigg ]\\
&\phantom{\frac{1}{\theta_i -u(R^+)} \bigg\{\;}+\frac{4 \pi^2}{\Le \ell^2}\bigg [u_{yy}(R^+)\bigg (w - \frac{2 \beta_R}{B-R}w_x\bigg )+2 u_y(R^+)\bigg (w_y - \frac{2\beta_R}{B-R}w_{xy}\bigg )\bigg ]+\frac{2u(R^+)}{B-R}w_x\bigg\} \\
&+\frac{\theta_i}{(\theta_i - u(R^+))^2}\bigg\{e^{-\frac{\Le (B-R)}{2}(x+1)}\frac{e^{-\Le R}-1}{\theta_i R}\\
&\phantom{+\frac{\theta_i}{(\theta_i - u(R^+))^2}\bigg\{\quad}\times
\bigg [\frac{4 \pi^2 u_y^2(R^+)}{\ell^2}\bigg(\frac{\beta_R^2 \Le (7+ \beta_R\Le) + \beta_R(2- \Le^2) - \Le}{\theta_i}u(R^+)-2\beta_R-\Le(\beta_R^2 -1)\bigg ) \\
&\phantom{+\frac{\theta_i}{(\theta_i - u(R^+))^2}\bigg\{\quad\times
\bigg [\;}+ 2(1+ \beta_R\Le) \Le\, u^2(R^+)\bigg ]\\
&\phantom{+\frac{\theta_i}{(\theta_i - u(R^+))^2}\bigg\{\;} +\frac{8 \pi^2 u_y^2(R^+)}{\Le \theta_i \ell^2}\bigg (w-\frac{4 \beta_R}{B-R}w_x+\frac{2 (\beta_R^2 -1)}{(B-R)^2}w_{xx}\bigg )
+\frac{8 u(R^+)}{\Le (B-R)^2}w_{xx}\bigg\}.
\end{align*}
}

\begin{thebibliography}{11}

\bibitem{i8}
A. Bayliss, E. M. Lennon, M. C. Tanzy, V. A. Volpert,
\newblock
{\it Solution of adiabatic and nonadiabatic combustion problems using step-function reaction models},
\newblock
J. Eng. Math. {\bf 79} (2013), 101-124.

\bibitem{BSN85}
H. Berestycki, B. Nicolaenko, B. Scheurer,
\newblock
Traveling wave solutions to combustion models and their singular limits,
\newblock
SIAM J. Math. Anal. \textbf{16} (1985), 1207-1242.

\bibitem{BGKS15}
I. Brailovsky, P. V. Gordon, L. Kagan, G. Sivashinsky,
\newblock
\emph{Diffusive-thermal instabilities in premixed flames: stepwise ignition-temperature kinetics},
\newblock
Combustion and Flame \textbf{162} (2015), 2077-2086.

\bibitem{i7}
I. Brailovsky, G. I. Sivashinsky,
\newblock
{\it Momentum loss as a mechanism for deflagration to detonation transition},
\newblock
Combust. Theory Model. {\bf 2} (1998), 429-447.


\bibitem{BGZ}
C.-M Brauner, P.V. Gordon, W. Zhang,
\newblock
\emph{An ignition-temperature model with two free interfaces in premixed flames},
\newblock
Combustion Theory Model. \textbf{20} (2016), 976-994.

\bibitem{BHL13}
C.-M. Brauner, L. Hu, L. Lorenzi,
\newblock
\emph{Asymptotic analysis in a gas-solid combustion model with pattern formation},
\newblock
Chin. Ann. Math. Ser. B \textbf{34} (2013), 65-88. See also:
Partial differential equations: theory, control and approximation, 139-169, Springer, Dordrecht, 2014

\bibitem{BHLS10}
C.-M. Brauner, J. Hulshof, L. Lorenzi, G. Sivashinsky,
\newblock
\emph{A fully nonlinear equation for the flame front in a quasi-steady combustion model},
\newblock
Discrete Contin. Dyn. Syst. Ser. A \textbf{27} (2010), 1415-1446.

\bibitem{BHL00}
C.-M. Brauner, J. Hulshof, A. Lunardi,
\newblock
\emph{A general approach to stability in free boundary problems},
\newblock
J. Differential Equations \textbf{164} (2000), 16-48.

\bibitem{BL18}
C.-M. Brauner, L. Lorenzi,
\newblock Local existence in free interface problems with underlyning second-order Stefan condition,
\newblock
Romanian J. Pure Appl. Math, to appear.


\bibitem{BLSX10}
C.-M. Brauner, L. Lorenzi, G.I. Sivashinsky, C.-J. Xu,
\newblock
\emph{On a strongly damped wave equation for the flame front},
\newblock
Chin. Ann. Math. Ser. B \textbf{31} (2010),  819-840.

\bibitem{BL00}
\newblock
C.-M. Brauner, A. Lunardi,
\newblock
\emph{Instabilities in a two-dimensional combustion model with free boundary},
\newblock
Arch. Ration. Mech. Anal. \textbf{154} (2000), 157-182.


\bibitem{BRS95}
C.-M. Brauner, J.-M. Roquejoffre, Cl. Schmidt-Lain\'e,
\newblock
\emph{Stability of travelling waves in a parabolic equation with discontinuous source term},
\newblock
Comm. Appl. Nonlinear Anal. \textbf{2} (1995), 83-100.

\bibitem{BL83}
J.D. Buckmaster, G.S.S. Ludford,
\newblock
Theory of Laminar Flames,
\newblock
Cambridge University Press, 1982.

\bibitem{C78}
K.-C. Chang,
\newblock
On the multiple solutions of the elliptic differential equations with discontinuous nonlinear terms,
\newblock
Scientia Sinica \textbf{21} (1978), 139-158.

\bibitem{i5}
P. Colella, A. Majda, V. Roytburd,
\newblock
{\it Theoretical and structure for reacting shock waves},
\newblock
SIAM J. Sci. Statist. Comput. {\bf 7} (1986),  1059-1080.

\bibitem{i6}
J. H. Ferziger, T. Echekki,
\newblock
{\it  A Simplified Reaction Rate Model and its Application to the Analysis of Premixed Flames},
\newblock
Combust. Sci. Technol. {\bf 89} (1993), 293-315.

\bibitem{gilbarg}
D. Gilbarg, N. Trudinger,
\newblock
Elliptic partial differential equations of second order. Reprint of the 1998 edition.
\newblock
Classics in Mathematics.
\newblock
Springer-Verlag, Berlin, 2001.

\bibitem{DG11}
Y. Du, Z. Guo,
\newblock
\emph{Spreading-vanishing dichotomy in the diffusive logistic model with a free boundary, II},
\newblock
J. Differential Equations \textbf{250} (2011), 4336-4366.

\bibitem{DL10}
Y. Du, Z. Lin,
\newblock
\emph{Spreading-vanishing dichotomy in the diffusive logistic model with a free boundary},
\newblock
SIAM J. Math Anal. \textbf{42} (2010), 377-405.

\bibitem{DL10-bis}
Y. Du, Z. Lin,
\newblock
\emph{Erratum: Spreading-vanishing dichotomy in the diffusive logistic model with a free boundary}.
\newblock
SIAM J. Math. Anal. \textbf{45} (2013), 1995-1996.

\bibitem{G95}
R. Gianni,
\newblock
\emph{Existence of the free boundary in a multi-dimensional combustion problem},
\newblock
Proc. Roy. Soc. Edinburgh Sect. A {\bf 125} (1995), 525-544.

\bibitem{GH92}
R. Gianni, J. Hulshof,
\newblock
\emph{The semilinear heat equation with a Heaviside source term},
\newblock
EJAM \textbf{3}(1992), 367-379.

\bibitem{GM93a}
R. Gianni, P. Mannucci,
\newblock
\emph{Existence theorems for a free boundary problem in combustion theory},
\newblock
Quart. Appl. Math. \textbf{51}(1993), 43-53.

\bibitem{GM93b}
R. Gianni, P. Mannucci,
\newblock
\emph{Some existence theorems for an $N$-dimensional parabolic equation with a discontinuous source term},
\newblock
SIAM J. Math. Anal. \textbf{24}(1993), 618-633.

\bibitem{H80}
D. Henry,
\newblock
Geometric theory of parabolic equations,
\newblock
Lect. Notes in Math. \textbf{840},
\newblock
Springer, Berlin 1980.

\bibitem{HBSS13}
L. Hu, C.-M. Brauner, J. Shen, G. I. Sivashinsky,
\newblock
\emph{Modeling and Simulation of fingering pattern formation in a combustion model},
\newblock
Math. Models Methods Appl. Sci. \textbf{25} (2015), 685-720.

\bibitem{krylov}
N.V. Krylov,
\newblock
Lectures on elliptic and parabolic equations in Sobolev spaces,
\newblock
Graduate Studies in Mathematics \textbf{96},
\newblock
American Mathematical Society, Providence, RI, 2008

\bibitem{lorenzi-1}
L. Lorenzi,
\newblock
\emph{Regularity and analyticity in a two-dimensional combustion model},
\newblock
\emph{Adv. Differential Equations} \textbf{7} (2002), 1343-1376.

\bibitem{lorenzi-2}
L. Lorenzi,
\newblock
{\it A free boundary problem stemmed from combustion theory. I. Existence, uniqueness and regularity results},
\newblock
J. Math. Anal. Appl. \textbf{274} (2002), 505-535.

\bibitem{lorenzi-3}
L. Lorenzi,
\newblock
{\it A free boundary problem stemmed from combustion theory. II. Stability, instability and bifurcation results},
\newblock
J. Math. Anal. Appl. \textbf{275} (2002), 131-160.

\bibitem{lunardi}
\newblock
A. Lunardi,
\newblock
Analytic Semigroups and Optimal Regularity in Parabolic Problems,
\newblock
Birkh\"auser, Basel, 1996.

\bibitem{i4}
E. Mallard, H.L. Le Ch\^atelier,
\newblock
{\em  Recherches exp\'erimentales et th\'eoriques sur la
combustion des m\'elanges gazeux explosifs. Premier m\'emoire: Temp\'erature d'inflammation},
\newblock
Ann. Mines {\bf 4} (1883), pp. 274-295.


\bibitem{MS79}
B.J. Matkowski, G.I. Sivashinsky,
\newblock
\emph{An asymptotic derivation of two models in flame theory associated
with the constant density approximation},
\newblock
SIAM J. Appl. Math. {\bf 37} (1979), 686-699.

\bibitem{NS87}
\newblock
J. Norbury, A. M. Stuart,
\newblock
Parabolic free boundary problems arising in porous medium combustion,
\newblock
IMA J. Appl. Math. \textbf{39} (1987), 241-257.

\bibitem{NS89}
\newblock
J. Norbury, A. M. Stuart,
\newblock
A model for porous medium combustion,
\newblock
Quart. J. Mech. Appl. Math. \textbf{42} (1989), 159-178.

\bibitem{STW}
J. Shen, T. Tang, L.L. Wang,
\newblock
Spectral Methods. Algorithms, Analysis and Applications,
\newblock
Springer Series in Computational Mathematics \textbf{41}. Springer, Heidelberg, 2011.

\bibitem{S80}
G.I. Sivashinsky,
\newblock
\emph{On flame propagation under condition of stoichiometry},
\newblock
SIAM J. Appl. Math. \textbf{39} (1980), 67-82.

\bibitem{WWQ}
Z. Wang, H. Wang, S. Qiu,
\newblock
\emph{A new method for numerical differentiation based on direct and inverse problems of partial differential equations},
\newblock
Applied Mathematics Letters   \textbf{43} (2015), 61-67.

\bibitem{WZ17}
M. Wang, J. Zhao,
\newblock
\emph{A Free Boundary Problem for a predator-prey model with double free boundaries},
\newblock
 J. Dynam. Differential Equations \textbf{29} (2017), 957-979.


\bibitem{WZW}
Z. Wang, W. Zhang, B. Wu,
\newblock
\emph{Regularized optimization method for determining the space-dependent source in a parabolic equation without iteration},
\newblock
Journal of Computational Analysis \& Applications   \textbf{20} (2016), 1107-1126.
\end{thebibliography}
\end{document}